\pdfoutput=1

\documentclass[12pt,a4paper,english]{article}
\usepackage[utf8]{inputenc}
\usepackage[T1]{fontenc}
\usepackage[a4paper]{geometry}
\usepackage{mathptmx}%
\usepackage{amsmath}
\usepackage{amsthm}
\usepackage{amssymb}
\usepackage{mathrsfs}
\usepackage{stmaryrd}%
\usepackage{array}

\usepackage[dvipdfmx]{graphicx}
\usepackage{bmpsize}

\usepackage{braket}
\usepackage{float}
\usepackage{mathtools}
\usepackage{thmtools,thm-restate}
\usepackage[all]{xy}
\usepackage{indentfirst}%
\usepackage{tikz}
\usepackage{hyperref}
\hypersetup{colorlinks=true, citecolor=blue, filecolor=black, linkcolor=blue, urlcolor=blue}
\usepackage{setspace}%

\theoremstyle{plain}
\newtheorem{thm}{Theorem}[section]
\newtheorem{prop}[thm]{Proposition}
\newtheorem{cor}[thm]{Corollary}
\newtheorem{lem}[thm]{Lemma}

\newtheorem{exe}[thm]{Example}

\newtheorem{facts}[thm]{Facts}
\theoremstyle{definition}
\newtheorem{defi}[thm]{Definition}

\theoremstyle{remark}
\newtheorem{rmk}{Remark}

\begin{document}
	
	\title{Meromorphic Projective Structures,\\ Opers and Monodromy}
	\author{\emph{Titouan Sérandour}\footnote{Université Côte d'Azur, CNRS, Laboratoire J.-A. Dieudonné, Parc Valrose, F-06108 Nice Cedex~2, France \url{titouan.serandour@univ-cotedazur.fr}}}
	\date{September 2023}%
	\maketitle
	
	\begin{abstract}
		The complex projective structures considered is this article are compact curves locally modeled on $\mathbb{CP}^1$. To such a geometric object, modulo marked isomorphism, the monodromy map associates an algebraic one: a representation of its fundamental group into $\operatorname{PGL}(2,\mathbb{C})$, modulo conjugacy. This correspondence is neither surjective nor injective. Nonetheless, it is a local diffeomorphism [Hejhal, 1975]. We generalize this theorem to projective structures admitting poles (without apparent singularity and with fixed residues): the corresponding monodromy map (including Stokes data) is a local biholomorphism.\\
		
		

	\end{abstract}
	
	\tableofcontents

\section{Introduction}

The projective structures considered is this article are surfaces locally modeled on open subsets of the Riemann sphere $\mathbb{P}^1$. Since the automorphism group of the latter is equal to the group of Möbius transformations,
this amounts to saying they are Riemann surfaces, but with a constraint stronger than the mere holomorphy of the coordinate changes: precisely, they must be linear fractional transformations.

Let us consider a projective structure $P$. The analytic continuations of an initial projective chart $\varphi_0$ along paths traveling across the surface give rise to a multivalued function, that can be uniformized by passing to the universal covering. We thus obtain a holomorphic immersive function $f$ with values in the Riemann sphere, called developing map, that "globalizes" the initial chart. This developing map is equivariant with respect to the monodromy representation \[\rho_f:\pi_1(P)\longrightarrow\operatorname{Aut}(\mathbb{P}^1)\simeq\operatorname{PGL}(2,\mathbb{C})\] of $\varphi_0$. The latter "globalizes" the projective changes of coordinates.

The developing map $f$, and in turn the associated representation $\rho_f$, depend on the initial projective chart $\varphi_0$. However, the conjugacy class of the monodromy representation does not. In this way, a projective structure determines a conjugacy class of representations of its fundamental group into $\operatorname{PGL}(2,\mathbb{C})$.

Given an oriented connected compact smooth real surface $S$, we generally consider the set of projective structures on $S$ (i.e. inducing a differentiable structure and an orientation compatible with the one of $S$), each element of which defines a representation of the group $\pi_1(S)$, up to conjugacy. This amounts to fix a genus $g$ and an orientation. Two projective structures on $S$ can be isomorphic even if their monodromy representations are not conjugated: in that sense, the monodromy is not an invariant. To remedy this and define a true invariant, we introduce the notion of an isomorphism (or equivalence) of marked projective structure: it is an isomorphism that is moreover isotopic to the identity of $S$. We denote by $\mathcal{P}(S)$ the set of equivalence classes of marked projective structures on $S$. This allows us to define the monodromy map \[\operatorname{Mon}_S:\mathcal{P}(S)\longrightarrow\mathcal{R}(S):=\operatorname{Hom}(\pi_1(S),\operatorname{PGL}(2,\mathbb{C}))/\operatorname{PGL}(2,\mathbb{C}).\]

To a geometric object, this map associate an algebraic one. It builds a bridge between two types of mathematical objects of different natures, and we may hope that adopting the most adequate perspective will make the resolution of some problems easier, that is, by crossing the bridge if only it is possible to do so in both directions. Having this picture in mind, two questions arise naturally:
\begin{enumerate}
	\item \emph{Under which condition an equivalence class of representations in $\mathcal{R}(S)$ arises from a monodromy representation of a complex projective structure on $S$?}
	\item \emph{Are two marked projective structures defining equivalent monodromy representations necessarily equivalent?}
\end{enumerate}

The first question concerns the characterization of the image of the map $\operatorname{Mon}_S$. It has been solved by Gallo, Kapovich and Marden \cite{MR1765706} when $S$ has genus $g\geq 2$. The second question, of whether $\operatorname{Mon}_S$ is one-to-one, has in general a negative answer. Thus, the monodromy representation does not characterize completely the complex projective structure to which it is associated. However, several weaker injectivity results are known. For example, it appeared as early as in the work of Poincaré that for all complex structure $C$ on $S$, the monodromy map $\operatorname{Mon}_S$ is injective in restriction to the set of projective structures compatible with $C$.
Alternatively, one can endow the sets $\mathcal{P}(S)$ and $\mathcal{R}(S)$ with a topology, or even with a real or complex differentiable structure, and ask a local version of the second question (and thus of the first one, in a sense)
The monodromy map is regular, and Hejhal showed in 1975 that if the genus $g$ of $S$ is at least $2$, then it is a local $\mathcal{C}^\infty$-diffeomorphism \cite[Thm 1 p. 20]{MR463429}. A few years later, Earle and Hubbard showed via a different approach that it is a local biholomorphism \cite{MR624807,MR624819}. The genus~$1$ case has to be treated differently, but leads to a similar statement; we refer to \cite[Sec. 1.4]{MR2647972}. Let us mention that a new proof of Hejhal's theorem for genus~$2$ surfaces was given in \cite[Sec. 8.4.2]{MR3941853}.

The main result of our paper (Theorem \ref{thm:local-injectivity}) is a generalization of Hejhal's theorem to meromorphic projective structures.

\paragraph{$(G,X)$-structures.} A complex projective structure can be equivalently defined as a $(G,X)$-structure, for the group $G=\operatorname{Aut}(\mathbb{P}^1)$ and the variety $X=\mathbb{P}^1$. The monodromy map can be defined in the more general setting of $(G,X)$-structures (for other $G$ and $X$) and the arguments of Ehresmann and Thurston \cite[Prop. 5.1]{Thurston} show that it is also a local homeomorphism. See also \cite[Thm. 7.2.1]{Goldman} and \cite[Sec. 7]{MR2827816} and the references therein.

\paragraph{Quadratic differentials.} Projective structures where studied as soon as in the \textsc{xix}\textsuperscript{th} century in relation with Schwarzian equations and their monodromy, and also with the uniformization problem (see \cite{MR3494804}). Schwarzian equations are the homogeneous linear equations of order two of the form
\begin{equation*}
y''+\frac{q(x)}{2}y=0,
\end{equation*}
where $q$ is a meromorphic function over a domain of the complex plane. A brief computation shows that the quotient $\varphi=y_1/y_2$ of two linearly independent solutions of such an equation satisfies
\begin{equation}\label{equation-intro-anglais}
\mathcal{S}(\varphi):=\left(\frac{\varphi''}{\varphi'}\right)'-\frac{1}{2}\left(\frac{\varphi''}{\varphi'}\right)^2=q
\end{equation}
where $\mathcal{S}(\varphi)$ is the Schwarzian derivative of $\varphi$. From the basic properties of the latter, we deduce that two solutions differ by a linear fractional transformation acting on the target: thus, such functions $\varphi$ define a projective atlas.

Those basic properties are the following: first, the Schwarzian derivative of an immersive function vanishes identically if and only if this function is the restriction of a linear fractional transformation. Intuitively, this derivative measures the failure of an immersion to be the restriction of a linear fractional transformation. In addition, the composition rule $\mathcal{S}(f\circ g)=(g')^2\mathcal{S}(f)\circ g+\mathcal{S}(g)$ it satisfies, applied to projective changes of coordinates $g$ (for which $\mathcal{S}(g)\equiv 0$), becomes exactly the change of variable formula for global quadratic differential forms.

In fact, the set of projective structures on a given compact Riemann surface $C$ is endowed with an affine space structure for the vector space of quadratic differential forms on $C$. The projective charts of the translation of $P$ by $\phi$ are defined as follows: if $x$ is a chart of $P$ with respect to which $\phi=\frac{q(x)}{2}dx\otimes dx$, then the charts of $P+\phi$ are the solutions of the equation \eqref{equation-intro-anglais}. On the other hand, the difference between two projective structures on $C$ is measured by a quadratic differential form defined locally as the Schwarzian derivative of the difference between two charts, one for each of those projective structures. It is this relation between projective structures and quadratic differential forms that allows us to endow the moduli space $\mathcal{P}(S)$ with a smooth complex structure (more precisely: with an holomorphic affine bundle structure over the Teichmüller space $\mathcal{T}(S)$, for the vector bundle of quadratic differential forms).

Thus, a projective structure might be seen as a "global scalar differential equation on a manifold", that is, as a connection.

\paragraph{Meromorphic projective structures and opers.} The viewpoint on projective structures that turns out to be the most fruitful in the setting of this paper is the one of opers (as defined by Beilinson and Drinfeld). To a projective structure on $S$, we can associate in a canonical way a $G$-oper on the underlying Riemann surface $C$, for the group $G=\operatorname{PGL}(2,\mathbb{C})$, and conversely. It is a triple $(\pi:Q\rightarrow C,\mathcal{F},\sigma)$ composed of a $\mathbb{P}^1$-bundle $\pi$, a holomorphic foliation $\mathcal{F}$ on $Q$ transverse to the fibers of $\pi$ and a section $\sigma$ of the bundle $\pi$ transverse to the foliation $\mathcal{F}$. Such a foliation is said to be Riccati with respect to $\pi$, because it is induced by a Riccati differential equation in any local trivialization of this $\mathbb{P}^1$-bundle. Riccati foliations play a preeminent role in the theory of holomorphic foliations.

A demonstration of Hejhal's theorem based on this viewpoint was sketched by Loray and Marín \cite{MR2647972}; the ideas implemented here are in fact due to Ehresmann and Thurston. We have completed this demonstration and adapted it to the case of meromorphic projective structures.

Let us mention that a generalization (a weak version) of Hejhal's theorem for $G$-opers (without singularity), valid for more general choices of group $G$, was given by Sanders \cite[Thm. 6.3]{https://doi.org/10.48550/arxiv.1804.04716}.

The relation between projective structures and quadratic differential forms is also essential to the precise definition of projective structures with poles. Let us consider a projective structure $P$ on $C$. A meromorphic projective structure on $C$ can be defined as the translation of $P$ by a meromorphic quadratic differential form $\phi$ on $C$ (away from the poles of the latter). The poles $p_i$ of $P$, as well as their orders $n_i$, are by definition the ones of $\phi$.

To meromorphic projective structures correspond meromorphic $\operatorname{PGL}(2,\mathbb{C})$-opers, that is to say triples $(\pi:Q\rightarrow C,\mathcal{F},\sigma)$ as above, except $(\pi:Q\rightarrow C,\mathcal{F})$ is now a singular Riccati foliation. This correspondence is not bijective, though, in contrast with the non-singular case: two opers not holomorphically equivalent can very well define the same meromorphic projective structure. Nonetheless, there exists a unique minimal $\operatorname{PGL}(2,\mathbb{C})$-oper associated with a meromorphic projective structure without apparent singularity, minimal in the sense that its polar divisor is minimal up to bimeromorphic gauge transformation.\\

Defining the monodromy map is a much more delicate task than in the non-singular case. Let us denote by $\Sigma$ the set of poles of a meromorphic projective structure $P$ on $C$. The monodromy representation of $P$ is defined as the one of the non-singular projective structure induced by $P$ on $C^*=C\smallsetminus\Sigma$. It includes the monodromy originating from the topology of $S$, but also local monodromies, around each singular points. Then comes, for poles of order at least $3$, the Stokes data, forming the so-called generalized monodromy data. This additions are inspired by the theory of meromorphic linear rank $2$ connections, of which Riccati foliations are the projectivized versions. Note that the monodromy representation of a meromorphic projective structure is also the one of the associated Riccati foliation.

In fact, every singular Riccati foliation is the projectivization of meromorphic rank $2$ connection $(E,\nabla)$, that is to say a global linear differential system. Knowing this, we can use the existing works on linear connections and deduce some analogous results regarding Riccati foliation. Notably, it is generally in the context of linear connections that can be found, in the literature, the constructions of moduli spaces and isomonodromic deformations. We keep in mind, though, that in our work no mathematical necessity imposes to speak about linear connections.  In a similar way, a $\operatorname{PGL}(2,\mathbb{C})$-oper lifts to a $\operatorname{GL}(2,\mathbb{C})$-oper, that is to say a triple $(E,\nabla,L)$ composed of a rank $2$ vector bundle on $C$, a meromorphic connection $\nabla$ on $E$ and a line subbundle $L$ of $E$ satisfying a transversality condition.

The sort of "bridge" established by the monodromy map is well-known for linear connections: it is the Riemann-Hilbert correspondence. The study of the monodromy allows a qualitative study of differential equations in the complex domain without having to solve them explicitly, which we can rarely achieve. In the case of meromorphic connections, it is necessary, in order to establish such a bijective correspondence, to include additional data to the monodromy representation. Those are notably Stokes matrices, at poles of orders at least $2$. The analysis of singularities reveals a profound dichotomy in the behavior of solutions -- and in turn in the local analytic classification of equations -- according to the orders of those singularities: equal to $1$ or at least $2$ (this is known as the Stokes phenomenon, discovered by Stokes in the \textsc{xix}\textsuperscript{th} century). We say they are regular (or moderate, or tame) singularities and irregular (or wild) singularities, respectively. The moduli spaces of generalized monodromy data are called wild character varieties.

The same goes for meromorphic projective structures (but for pole orders at most $2$ or at least $3$). Allegretti and Bridgeland constructed a smooth (possibly non-Hausdorff) character variety $\mathcal{X}^*(\mathbb{S},\mathbb{M})$ (using Fock-Goncharov coordinates \cite{MR2233852}) containing generalized monodromy data together with a framing, and showed that the monodromy map
\begin{align}
\operatorname{Mon}_{\mathbb{S},\mathbb{M}}:\mathcal{P}^*(\mathbb{S},\mathbb{M})\longrightarrow \mathcal{X}^*(\mathbb{S},\mathbb{M})
\end{align}
is holomorphic \cite{Allegretti_2020}. Here, $\mathcal{P}^*(\mathbb{S},\mathbb{M})$ denotes the moduli space of (equivalence classes of) meromorphic projective structures marked by $(\mathbb{S},\mathbb{M})$, signed, and without apparent singularity. The marking of a meromorphic projective structure on $S$ by a real bordered surface $\mathbb{S}$ with marked points $\mathbb{M}$ is a marking of the surface obtained by the real blow-up of $S$ at each irregular singularity, with marked points corresponding to regular singular points and Stokes directions on this blown-up surface. The signing encodes a choice of fixed point of the local monodromy at regular singularities, needed in the definition of the monodromy map of Allegretti and Bridgeland (in order to define the framings). Except in a few special cases (always in genus $0$ and with a small number of poles counted with multiplicity), $\mathcal{P}^*(\mathbb{S},\mathbb{M})$ is a smooth variety.

In this paper, we fix the value $\lambda_{-1}^{(i)}$ of the residue at each pole $p_i$ (as well as a signing), even at irregular singularities. The associated moduli space \[\mathcal{P}^\circ(\mathbb{S},\mathbb{M},(\lambda_{-1}^{(i)}))\] is again smooth. In particular, the number of pole $d$ and their orders $(n_i)$ are fixed.\\

Several authors have developed independently some alternative constructions of generalized monodromy data and of wild character varieties for linear connections \cite{MR2250948,MR2649335,MR3126570,boalch2015twisted,MR3383320,MR3802126,MR3932256}. In order to obtain a \emph{smooth} moduli space, it is necessary to impose some good conditions on the monodromy data (stability, non-degeneracy...).

On the other hand, Inaba \cite{MR4545855} constructed a \emph{smooth} moduli space for linear connections. Here too, in reality, it is necessary to restrict to connections satisfying a stability condition. We can see $\operatorname{GL}(2,\mathbb{C})$-opers as a special kind of connections, satisfying a special condition: they admit a transverse line subbundle. We are going to show that this condition implies in particular that they are stable in the sense of Inaba. Now, the moduli space of Inaba (when the curve $C$, the polar divisor and the formal data are fixed) is in bijection with a subset of the wild character variety via the irregular Riemann-Hilbert correspondence. Fixing moreover the trace connexion on the determinant bundle, we deduce a smooth moduli space for $\operatorname{PGL}(2,\mathbb{C})$-opers, and we apply the irregular Riemann-Hilbert correspondence, in its version for singular Riccati foliations. This is how we obtain a smooth wild character variety \[\bar{\mathcal{R}}^*(\mathbb{S},\mathbb{M},(\lambda_{-1}^{(i)})),\] without having to formulate a condition directly on the generalized monodromy data.

Note that the monodromy map for meromorphic projective structures differs from the Riemann-Hilbert map, since the complex structure on $S$ is not fixed and since we consider only connection of a special kind, namely opers.\\

Our main result is stated as follows:

\paragraph{\textcolor{blue}{Theorem} \ref{thm:local-injectivity}}
\emph{Assume that if $g=0$, then $|\mathbb{M}|\geq3$, and that if $g=1$, then $|\mathbb{M}|\geq 1$. Then, the monodromy map \[\operatorname{Mon}_{S,(n_i),(\lambda_{-1}^{(i)})}:\mathcal{P}^\circ(\mathbb{S},\mathbb{M},(\lambda_{-1}^{(i)}))\longrightarrow \bar{\mathcal{R}}^*(\mathbb{S},\mathbb{M},(\lambda_{-1}^{(i)}))\] is a local biholomorphism}.\\

Our theorem generalizes and unifies a number of analogous results that where already known before in some specific cases, and which where obtained through various techniques.
Let us cite the works of Bakken \cite{MR508244} (cf. \cite[p. 198]{MR0486867}) in genus $g=0$, with a unique pole of order $n\geq5$; Iwasaki \cite[Thm. 5.9]{MR1146359} for projective structures admitting a fixed number of apparent singularities; Luo \cite{MR1202134} for the case where all singularities are regular (of order exactly $2$) with an index  $\theta\in\mathbb{C}\smallsetminus\mathbb{Z}$ (indices can vary freely in this set), that is to say with non-trivial and non-parabolic local monodromy; Hussenot Desenonges \cite{MR3956187} for a case where all singularities are regular (of order $2$) with a fixed index $\theta=0\in\mathbb{Z}$ (hence a parabolic, in particular non-trivial, local monodromy); Gupta and Mj \cite{MR4232545} for the case where all singularities are irregular (and the index vary freely in $\mathbb{C}$).

Recently, LeBarron Alley \cite{MR4338230} studied meromorphic cyclic $\operatorname{SL}(n,\mathbb{C})$-opers admitting a unique pole on the Riemann sphere and concluded that the associated monodromy map is a holomorphic immersion if the pole order is a multiple of $n$.

\paragraph{Ingredients of the demonstration.} Beside the works of Allegretti, Bridgeland and Inaba, our demonstration rely on the construction of universal isomonodromic deformations by Heu \cite{MR2667785}. Locally on $\mathcal{P}^\circ(\mathbb{S},\mathbb{M},(\lambda_{-1}^{(i)}))$, we construct an analytic family of meromorphic $\operatorname{PGL}(2,\mathbb{C})$-opers that locally embeds in the universal family of singular Riccati equations derived from Inaba's work. Then, the isomonodromic flow allows us to retract two nearby opers having the same generalized monodromy data one on the other, leading to the local injectivity of the monodromy map. \autoref{thm:local-injectivity} then follows from the holomorphy of the monodromy map, thanks to the open mapping theorem.

\paragraph{About the image of the monodromy map.} A natural and related problem (which is the extension of question~$1$), is to determine the image of the monodromy map of meromorphic projective structures. Let us mention that analogs of the Gallo-Kapovich-Marden theorem (with variable residues) where recently proved in a series of work by Faraco, Gupta and Mj \cite{MR4078090,MR4286048,https://doi.org/10.48550/arxiv.2109.04044} and Nascimento \cite{https://doi.org/10.48550/arxiv.2105.07084}. The case of branched complex projective structures with fixed branch divisor was handled by Le Fils \cite{MR4575871}.

\paragraph{An open question in symplectic geometry.}
The question of whether the monodromy map of non-singular projective structures is a symplectomorphism was addressed by Kawai \cite{MR1386110} and Loustau \cite{MR3352248} (see also \cite{MR3735629} and \cite{https://doi.org/10.48550/arxiv.1804.04716}). Indeed, it is well-known since the work of Goldman \cite{MR762512,MR2094117} that the character variety (i.e. the target of the monodromy map) has a natural symplectic structure. On the other hand, the space of projective structures possesses several cotangent symplectic structures depending on the different Schwarzian parametrizations. Loustau gave a necessary and sufficient condition on the parametrization which ensures that the monodromy map is a symplectomorphism. Is the monodromy map of \emph{meromorphic} projective structures a symplectomorphism? The works of \cite{MR1178029} and \cite{MR3888789} gave a positive answer to this question is some specific cases involving only regular singularities.

\paragraph{Structure of the paper.} This article is organized as follows. The first section is dedicated to the introduction of meromorphic projective structures, while section 2 present their relation to meromorphic $\operatorname{PGL}(2,\mathbb{C})$-opers. In particular, we describe their minimal birational models. In section 3, we show that opers belong to Inaba's smooth moduli space. Section 4 contains the proof of the main theorem.

\paragraph{Acknowledgments.} This work was carried out during my PhD thesis~\cite{serandour:tel-04053917} at the University of Rennes, in the laboratory IRMAR - UMR CNRS 6625; I was supported by the Centre Henri Lebesgue program ANR-11-LABX-0020-0; I finished this article after I integrated the laboratory UMPA - UMR 5669 CNRS, at ENS de Lyon.
I would like to express all my gratitude to my PhD advisor, Frank Loray, for its guidance and for all the enriching discussions we had. %
I would also like to thank Tom Bridgeland and Sorin Dumitrescu, who reported my thesis. Their careful reading of my manuscript has led to several improvements.
Finally, let me thank Professor Michi-aki Inaba for answering in great details a question about one of its papers, thus clarifying a result I use in the present article.

\section{Meromorphic projective structures}

We start this first section with an introduction to complex projective structures without poles and their monodromy map. The Reader who is not yet familiar with the subject might profitably consult the surveys \cite{MR2497780,MR2647972} and the book \cite{MR3494804}. In the first section, we explain in particular that the set of projective structures on a fixed complex curve is an affine space for the vector space of quadratic differentials. This is fundamental because it gives rise to a structure of a complex manifold on the moduli space of projective structures of genus $g$. Moreover, it allows us to define meromorphic projective structures. We state Hejhal's theorem, and in the last section we recall the definition of the moduli space of marked meromorphic projective structures of genus $g$ with poles of prescribed (in a somewhat loose sense) orders, constructed by Allegretti and Bridgeland in~\cite{Allegretti_2020}.

\subsection{Complex projective structures (without pole) and their monodromy map}\label{sec:cplx-proj-struct}

Let $S$ be a connected $\mathcal{C}^\infty$-smooth oriented real surface.

\begin{defi}\label{complexprojectivestructure}
	A \emph{complex projective structure on S} is a maximal atlas of charts mapping open sets in $S$ into open sets in the complex projective line $\mathbb{P}^1$ and such that the transition maps are \emph{restrictions} of elements in the group of automorphisms of $\mathbb{P}^1$ as a complex manifold (which we canonically identify with $\operatorname{PGL}(2,\mathbb{C})$).
	
	(It is always implicitly assumed that both the $\mathcal{C}^\infty$-smooth structure and the orientation induced by this complex structure are the ones initially fixed on~$S$).
\end{defi}

This is also the definition of a \emph{$(G,X)$-structure on $S$}, with $G=\operatorname{PGL}(2,\mathbb{C})$ and $X=\mathbb{P}^1$.

Two projective structures on $S$ are said to be \emph{isomorphic} (resp. \emph{marked isomorphic}) if there exists an orientation-preserving diffeomorphism of $S$ (resp. a diffeomorphism of $S$ homotopic to the identity\footnote{Or isotopic, this is equivalent. Moreover, the orientation is automatically preserved.}) pulling back any projective chart of one of them to a projective chart of the other. We will denote by $\mathcal{P}(S)$ the set of marked isomorphism classes of projective structures on $S$. The marking will be needed in the definition of the monodromy map.

Every projective structure on $S$ induces a complex structure on $S$, hence there is a forgetful map $\mathcal{P}(S)\rightarrow\mathcal{T}(S)$ (where $\mathcal{T}(S)$ denotes the Teichmüller space of $S$) and we may as well speak of a projective structure on a complex curve. Every complex structure is induced by a projective one --meaning that the forgetful map is surjective-- as shown by the following example.

\begin{exe} \label{standard-proj-structure}
	Let $U$ be an open subset of the complex projective line and assume that it is invariant under a free and properly discontinuous action of a group of Möbius transformations $\Gamma$. Then, the local inverses of the associated quotient map $U\rightarrow U/\Gamma$, which is a covering map, form a projective atlas on the quotient space.

	The Poincaré-Koebe uniformization theorem ensures that any connected smooth complex curve $C$ is biholomorphic to such a quotient space. The projective structure induced on the latter can be pulled back to the curve $C$, giving a projective structure compatible with the complex one. The resulting structure is called the \emph{canonical projective structure} on $C$.

	The standard atlas on $\mathbb{P}^1$ is also projective, and all projective atlas on $\mathbb{P}^1$ are compatible with it. However, there is an infinite number of distinct projective structures on~$C$ as soon as it is different from $\mathbb{P}^1$.
\end{exe}

Let us denote by $u:\tilde{S}\longrightarrow S$ the universal cover of $S$, and let $C$ be a complex structure on $S$.

\begin{prop}[See Lem. 1 in \cite{MR624819}]\label{developingmap}
	Given a projective structure $P$ on $C$:
	\begin{enumerate}
		\item There exists an analytic map \[f:\tilde{C}\longrightarrow\mathbb{P}^1\] such that on each contractible open set $U\subset C$ on which the composition $f\circ u^{-1}$ is defined (here, $u^{-1}$ denote any local inverse for $u$), it is a projective chart. Any other such map is of the form $\sigma\circ f$ for some $\sigma\in\operatorname{PGL}(2,\mathbb{C})$. We say $f$ is \emph{a developing map} for the projective structure $P$.
		\item To every developing map $f$ there corresponds a unique morphism \[\rho_f:\pi_1(C)\longrightarrow\operatorname{PGL}(2,\mathbb{C})\] (called \emph{the monodromy representation}) such that $\rho_f(\gamma)\circ f = f\circ\gamma$ (\emph{equivariance}) and $\rho_{\sigma\circ f}=\sigma\circ \rho_f\circ\sigma^{-1}$.
	\end{enumerate}
\end{prop}

\begin{figure}[H]
	\centering
	\def\svgwidth{0.6\textwidth}
	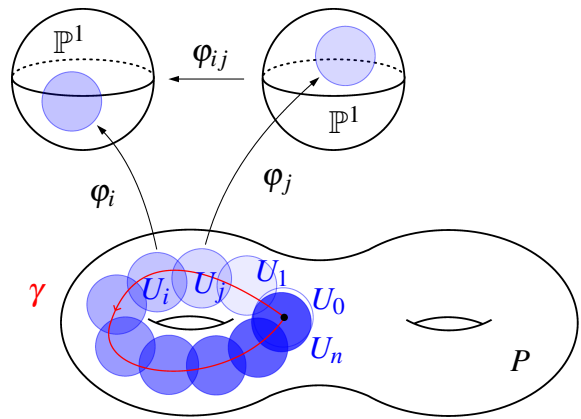
	\caption{If the $\varphi_{ij}:=\varphi_i\circ\varphi_j^{-1}$ denote the change of coordinate charts of $P$ as well as the associated Möbius transformations, then the monodromy along the loop $\gamma$ covered by the charts $U_0,\dots,U_n$ is obtained by composition of the $\varphi_{ij}$:  $\rho_f(\gamma)=\varphi_{n,(n-1)}\circ\cdots\circ\varphi_{2,1}\circ\varphi_{1,0}$.}
\end{figure}

A developing map can be constructed by analytic continuation of a projective chart by other projective charts. In this sense, it globalizes a projective chart, whereas the monodromy representation globalizes the transition maps. A projective structure on $S$ can be described by the data of a $\mathcal{C}^\infty$ immersion $f:\tilde{S}\rightarrow\mathbb{P}^1$ together with a representation $\rho\in\operatorname{Hom}(\pi_1(S),\operatorname{PGL}(2,\mathbb{C}))$ with respect to which it is equivariant: a so-called \emph{development-holonomy pair}\footnote{Here, "holonomy" and "monodromy" are synonyms.} $(f,\rho)$.

\begin{defi}\label{monodromymap}
	Let us denote $\mathcal{R}(S)$ the set of $\operatorname{PGL}(2,\mathbb{C})$-conjugacy classes of representations of the fundamental group of $S$ in $\operatorname{PGL}(2,\mathbb{C})$. Thanks to \autoref{developingmap}, to a marked isomorphism class of projective structures we can associate the conjugacy class of the monodromy representation of one of its developing maps. This defines the so-called \emph{monodromy map} \[\operatorname{Mon}_S:\mathcal{P}(S)\longrightarrow\mathcal{R}(S).\]
\end{defi}

Note that the marking is essential to define this map; it is also key to the smoothness of the Teichmüller space, hence of the moduli space $\mathcal{P}(S)$, that we will discuss latter.

\begin{thm}[Poincaré] \label{Poincare-thm}
	The monodromy map $Mon_S$ is injective in restriction to each fiber of the forgetful map $\mathcal{P}(S)\rightarrow\mathcal{T}(S)$.
\end{thm}

A short proof of this fact is given in \cite{MR2647972}.

\paragraph{Quadratic differentials and affine structure.} Let $C$ denote a (connected) smooth curve (i.e. a Riemann surface). We have seen that the set of complex projective structures over $C$ is non-empty (cf. \autoref{standard-proj-structure}). We can expect it to be large, since being a Möbius transformation is a much stronger condition than being holomorphic. In fact, it can be endowed with an affine space structure for the vector space $H^0(C,(T^*C)^{\otimes 2})$ of global holomorphic quadratic differentials (see \cite[Lem. 4]{MR624819} or \cite[Cor. 2]{MR207978}). We recall bellow this affine structure and introduce some notations. If $C$ is compact, the Riemann-Roch theorem implies that the above vector space is of dimension $3g-3$ (resp. $1$, $0$) if $C$ is of genus $g\geq2$ (resp. $g=1,0$) (see, for example, \cite[Cor. 5.4.2]{MR2247485}).

Let $f$ be a locally injective holomorphic function of one complex variable. The \emph{Schwarzian derivative of $f$} is \[\mathcal{S}(f)=\left(\frac{f''}{f'}\right)'-\frac{1}{2}\left(\frac{f''}{f'}\right)^2.\]
Now let $f,g$ two locally injective holomorphic maps between open sets of the complex plane. The key properties of the Schwarzian derivative are
\begin{itemize}\label{schwarzian-basic-property}
	\item $\mathcal{S}(f\circ g)=(g')^2\mathcal{S}(f)\circ g+\mathcal{S}(g)$,
	\item $\mathcal{S}(g)\equiv 0$ if and only if $g$ is the restriction of a Möbius transformation, i.e. $g(x)=\frac{ax+b}{cx+d}$, $(a,b,c,d)\in\mathbb{C}^4$, $ad-bc\neq 0$.
\end{itemize}
The Schwarzian derivative thus depends on the coordinate it is computed, so we usually specify it by a subscript: $\mathcal{S}_x(f)$.

Let $P_1$ and $P_2$ be two projective structures on $C$, with local coordinates $x_1$ and $x_2$, respectively, and denote\footnote{There is a slight abuse here: it is understood that projective charts take there values in $\mathbb{C}$ and that transition maps are invertible linear fractional functions. This is an alternative definition of a projective atlas.} $\psi=x_2\circ x_1^{-1}$. Then, we define $\phi:=P_2-P_1$ locally as \[\phi:=\frac{\mathcal{S}_{x_1}(\psi)}{2}dx_1^{\otimes 2},\] where $\mathcal{S}_{x_1}(\psi)$ denotes the Schwarzian derivative of $\psi$ with respect to the variable $x_1$. Those local expressions glue together by virtue of the basic properties of the Schwarzian derivative, defining a global holomorphic quadratic differential. Conversely, given a projective structure $P_1$ and a holomorphic quadratic differential $\phi$ written locally as $\phi=\frac{q(x_1)}{2}dx_1^{\otimes 2}$, the charts of $P_2:=P_1+\phi$ are defined to be the solutions of the equation
\begin{equation}\label{schwarzian-equation}
\mathcal{S}_{x_1}(\varphi)=q
\end{equation}
in $\varphi$. Furthermore, according to a theorem of Schwarz, it is equivalent to define projective charts as quotients $\varphi=y_1/y_2$ where $y_1$ and $y_2$ range over all independent solutions of the equation
\begin{equation}\label{second-order-reduced-scalar-eq}
y''+\frac{q(x_1)}{2}y=0.
\end{equation}
Thus, a projective structure may be interpreted as a global generalization of such a second order homogeneous linear scalar equation. For each chart of $C$, we get a coefficient $q$ and their collection forms a \emph{projective Cartan connection} (see \cite{Na2019}). On the intersections, the corresponding second order equations are \emph{projectively equivalent}, in the sense that their sets of quotients of solutions coincide (see \cite[Prop. VIII.3.8]{MR3494804}).

\begin{rmk}
	An equation of the form \eqref{second-order-reduced-scalar-eq} but with opposite sign before $q/2$ is sometimes used in the literature, e.g. in \cite{Allegretti_2020}. Let us apply the projective change of coordinate $\tilde{x_1}=ix_1$ to equation \eqref{second-order-reduced-scalar-eq}. We obtain an equation projectively equivalent to $y''-\frac{q(\tilde{x_1})}{2}y=0$ (see \cite[Prop. VIII.3.8]{MR3494804}).
\end{rmk}

The above affine structure endows the set of projective structures on a fixed curve $C$ with a structure of a complex manifold. The aim of the following paragraph is to describe a natural smooth complex structure on the whole $\mathcal{P}(S)$, i.e. letting the complex structure on $S$ vary.

\paragraph{The moduli space $\mathcal{P}(X/B)$.} For any holomorphic family $f:X\rightarrow B$ of compact smooth curves, let us denote \[\mathcal{P}(X/B):=\{(t,P):t\in B\textmd{ and $P$ is a projective structure on $C_t:=f^{-1}(t)$}\}.\]
Most of the time, there will be no ambiguity regarding the family $f$ to which this notation refers. There is a holomorphic vector bundle \[\mathcal{Q}(X/B)\rightarrow B,\] defined as $\mathcal{Q}(X/B)=f_*((T^*_{X/B})^{\otimes 2})$ (where $T_{X/B}$ stands for the vertical tangent bundle of $X$, relative to $f$), whose fiber over a point $t\in B$ is the vector space of quadratic differential forms $\mathcal{Q}(X/B)_t=H^0(C_t,(T^*C_t)^{\otimes 2})$.

A choice of a section $s$ of the map $\operatorname{pr}_1:\mathcal{P}(X/B)\rightarrow B$ furnishes a "choice of origin" in each fibers, which in turn provides an identification
\begin{align*}
\mathcal{P}(X/B)&\longrightarrow\mathcal{Q}(X/B)\\
(t,P)&\longmapsto(t,P-\operatorname{pr}_2(s(t)))
\end{align*}
(here, $\operatorname{pr}_i$ denotes the restriction to $\mathcal{P}(X/B)$ of the projection on the $i$-th factor). Two such sections $s_1$ and $s_2$ will induce the same complex structure on $\mathcal{P}(X/B)$ if and only if their difference $s_2-s_1$ is a holomorphic section of $\mathcal{Q}(X/B)$. Hubbard showed there is a natural complex structure on $\mathcal{P}(X/B)$, though, which is the only one such that firstly the map $\mathcal{P}(X/B)\rightarrow B$ is holomorphic and secondly the \emph{analytic families of projective structures} given by its holomorphic sections are induced by $\emph{relative projective structures}$ on $f$ \cite[Prop. 1]{MR624819}. The map $\mathcal{P}(X/B)\rightarrow B$ is a holomorphic affine bundle for the vector bundle $\mathcal{Q}(X/B)\rightarrow B$.

\begin{defi}
	A \emph{relative projective atlas} on a holomorphic family $f:X\rightarrow B$ of (compact or not) smooth curves is an atlas $(U_i,\varphi_i)$ for $X$ such that for each $t\in B$, the restriction $(U_i\cap C_t,\varphi_{i|C_t})$ (where the range of $\varphi_{i|C_t}$ is understood to be appropriately restricted as well) forms a projective atlas on $C_t$. A \emph{relative projective structure} is a maximal relative projective atlas.
\end{defi}

\begin{exe}\label{family-genus-1}
	The Teichmüller space of genus $1$ curves is isomorphic to the upper half-plane $\mathbb{H}$. Consider the first projection $p_1:\mathbb{H}\times\mathbb{C}\rightarrow\mathbb{H}$. A relative projective atlas on the holomorphic family \[f:(\mathbb{H}\times\mathbb{C})/\braket{(\tau,z+1),(\tau,z+\tau)}\longrightarrow\mathbb{H}\] of elliptic curves induced by $p_1$ is given by local inverses of the quotient map \[\mathbb{H}\times\mathbb{C}\rightarrow (\mathbb{H}\times\mathbb{C})/\braket{(\tau,z+1),(\tau,z+\tau)},\] composed with the projection on the second factor $\mathbb{C}$, regarded as embedded in $\mathbb{P}^1$.
\end{exe}

The aforementioned natural complex structure on $\mathcal{P}(X/B)$ makes the projection $\mathcal{P}(X/B)\rightarrow B$ into a holomorphic affine bundle for the vector bundle $\mathcal{Q}(X/B)$, with the action $(t,\phi)\cdot(t,P)=(t,P+\phi)$. Note that the key ingredient in Hubbard's proof is the existence, at least locally over the base $B$, of a relative projective structure. It is worth mentioning that not every holomorphic family of curves carry a relative projective structure\footnote{In contrast, the constant rank theorem provides any holomorphic families of curves with a relative complex atlas.}. For example, as showed by Zhao in \cite{Zhao}, a non-isotrivial holomorphic family of compact curves over a compact curve $B$ cannot support any relative projective structures. Hubbard verified the existence of a relative projective structure on a family of genus $g\geq 2$ curves under the assumption that the base $B$ is Stein.

Assume that $B=\mathcal{T}(S)$ is the Teichmüller space of marked curves (i.e. marked Riemann surfaces) on the compact smooth surface $S$, and $f:X\rightarrow B$ is the universal holomorphic family of curves over $B$. Then, the bundle $\mathcal{Q}(X/B)\rightarrow B$ is trivial, because the Teichmüller space is contractible and Stein. Moreover, in this case we can reinterpret elements $(t,P)\in\mathcal{P}(X/B)$ as follow: $P$ is a marked projective structure, with its marking being the only one compatible with the marking of $C_t$. Then, there is a natural bijection $\mathcal{P}(S)\simeq\mathcal{P}(X/B)$ (remember that marked complex curves have no non-trivial automorphisms; this rigidity is key to the smoothness of Teichmüller spaces and the existence of universal family of curves).

In the genus $1$ case, the section $s$ of $\mathcal{P}(S)\rightarrow\mathcal{T}(S)$ provided by the uniformization theorem (as in \autoref{family-genus-1}) induces the natural complex structure on $\mathcal{P}(S)$. However, this is false as soon as $g\geq2$ (note that if $g=1$ the monodromy representation of the canonical projective structure is complex, whereas if $g\geq2$ it is real). A way to construct a relative projective structure in the latter case is to use the simultaneous uniformization theorem of Bers. This leads to a so-called \emph{quasi-Fuchsian section}; for more details about this, see the discussion in \cite[Sec. 3.3]{MR2497780}.

The existence of a relative projective structure locally over the base of an arbitrary holomorphic family of compact curves may then be deduced from the universal property of Teichmüller spaces.

\begin{prop}\label{existence-relative-proj-struct}
	Let $f:X\rightarrow B$ be a holomorphic family of compact smooth curves. For any point $t_0$ in $B$, there exists a neighborhood $B_{t_0}$ of $t_0$ in $B$ such that the holomorphic family $f_{t_0}: f^{-1}(B_{t_0})\rightarrow B_{t_0}$ induced by $f$ supports a relative projective structure.
\end{prop}

\paragraph{Hejhal's theorem.} Gunning showed that a compact smooth curve $C$ admits an \emph{affine structure}\footnote{A projective structure such that, in coordinates, transition maps belong to the group $\operatorname{Aff}(\mathbb{C})$.} if and only if its genus is $1$ (this is false for meromorphic projective structures, as it will become apparent latter). Moreover, every projective structure on a genus $1$ curve can be reduced to an affine structure (see \cite[Cor. 3 p. 173, p. 192]{MR0207977} or \cite[Ex. 1.7]{MR2647972}). This implies that the monodromy representation of any developing map of a projective structure on a surface of genus $g\geq 2$ have non-commutative image \cite[Cor. p. 260]{MR624819} (and are irreducible). After choosing a set of generators of the group $\pi_1(S)$, the set of representations $\operatorname{Hom}(\pi_1(S),\operatorname{PGL}(2,\mathbb{C}))$ can be identified with an analytic space in $\operatorname{PGL}(2,\mathbb{C})^{2g}$. The open subset of representations with non-commutative image in $\operatorname{Hom}(\pi_1(S),\operatorname{PGL}(2,\mathbb{C}))$ is a sub-manifold of $\operatorname{PGL}(2,\mathbb{C})^{2g}$ and its quotient $\mathcal{R}^{\operatorname{nc}}(S)$ modulo conjugacy has a unique structure of complex manifold such that the quotient map is analytic \cite[Prop. 4]{MR624819}. It is of complex dimension $6g-6$, just as $\mathcal{P}(S)$.

\begin{thm}\emph{(\cite[Thm. 1 p. 20]{MR463429}, \cite[Cor. 2]{MR624807}, \cite[Thm. p. 272]{MR624819})}
	If $S$ has genus $g\geq2$, then the monodromy map \[\operatorname{Mon}_{S}^{g\geq2}:\mathcal{P}(S)\longrightarrow\mathcal{R}^{\operatorname{nc}}(S)\] is a local biholomorphism.
\end{thm}

In the genus $1$ case, the monodromy representation always belongs to the set $\mathcal{A}(S):=\operatorname{Hom}(\pi_1(S),\operatorname{Aff}(\mathbb{C}))/\sim$ of affine representation modulo conjugacy. It is also possible to show that the monodromy map $\operatorname{Mon}_{S}^{g=1}:\mathcal{P}(S)\longrightarrow\mathcal{A}(S)$ is a local biholomorphism (see \cite[Sec. 1.4]{MR2647972}). Here, the \emph{affine distortion} $f''/f'$ can play the role of the Schwarzian derivative.

In the genus $0$ case, there is only one projective structure and it is simply connected, so it would not make sense to seek for a similar statement.

\subsection{The space of meromorphic projective structures}\label{section-meromorphic-projective-structures}

The definition of a meromorphic projective structure on $C$ relies on the affine space structure on the set of complex projective structures over a complex curve.

\begin{defi}\label{meromorphicprojectivestructure}
	A \emph{meromorphic projective structure $P$ on a complex curve $C$}
	is a projective structure $P^*$ on the complement $C^*=C\smallsetminus\Sigma$ of a finite subset $\Sigma\subset C$, such that given a holomorphic projective structure $P_0$ on $C$, the quadratic differential $\phi=P^*-P_{0|C^*}$ on $C^*$ extends to a meromorphic quadratic differential on $C$, that is to say a meromorphic section of $(T^*C)^{\otimes 2}$.
\end{defi}

Such a structure $P^*$ is therefore determined by a pair $(P_0,\phi)$. Though it is not unique, every other pair $(\tilde{P_0},\tilde{\phi})$ describing $P$ satisfies $(\tilde{P_0},\tilde{\phi})=(P_0+\phi_{\operatorname{hol}},\phi-\phi_{\operatorname{hol}})$ for some holomorphic quadratic differential $\phi_{\operatorname{hol}}$. Hence the condition on $\phi$ is independent of $P_0$: points in $\Sigma$ are called \emph{poles} of the structure $P$ and the \emph{pole order} of $P$ at a point can be defined as the pole order at that point of one of its so-called \emph{polar differentials} $\phi$. Hence if $n_p$ is the order of $p\in\Sigma$, we can define the (effective) \emph{polar divisor} of $P$ by $D_P=\sum_{p\in\Sigma}n_pp$.

\paragraph{The monodromy representation.} The monodromy representation of a meromorphic projective structure~$P$ is well-defined, up to $\operatorname{PGL}(2,\mathbb{C})$-conjugacy, as the one of $P^*$ in (the notations of \autoref{meromorphicprojectivestructure}). It is a class of representation of the fundamental group of the punctured curve $C\smallsetminus\Sigma$. The \emph{local monodromy} around a pole is the image of a loop encircling it once in the counterclockwise direction.

\begin{rmk}\label{monodromy-scalareq-system}
	If $(y_1,y_2)$ is a fundamental system of solutions of a second order homogeneous scalar linear differential equation, then \[\begin{pmatrix}
	y_1 & y_2\\
	y_1' & y_2'
	\end{pmatrix}\] is a fundamental solution of its companion systems, and reciprocally. As a consequence, the monodromy matrices of the scalar equation and its companion system (both acting by right multiplication) around a pole are the same, say $M$.

	The monodromy of the projective structure associated to an equation of the form \eqref{second-order-reduced-scalar-eq}, however, is given in the homogeneous coordinates $[y_1:y_2]$ by the projectivization of the \emph{transpose} $^\top M$. But it is a classical fact that a matrix with complex coefficients is always conjugated to its transpose (in particular, a matrix has the same eigenvalues as its transpose), thus the local monodromy of the projective structure is, up to conjugacy, the projectivization of the local monodromy of the associated scalar equation and of that of its companion linear system.
\end{rmk}

\paragraph{Regular singularities.}  This terminology comes from second order scalar differential equations (and their companion first order linear differential systems). The term \emph{regular singular} refers to points where solutions of the scalar equation (or the linear system) exhibit at most polynomial growth on some sectors (see~\cite{IlyashenkoEncyOfMath}). For the scalar equation \eqref{second-order-reduced-scalar-eq} (and by definition for projective structures), this is equivalent to $q$ having a pole order at most $2$ (we also say it is a \emph{Fuchsian singularity}), but not for linear systems. On the one hand, a singularity with \emph{Poincaré rank} equal to $0$ is regular singular. On the other, only the \emph{minimal} Poincaré rank (minimal in the equivalence class of the system up to meromorphic gauge transformation\footnote{cf. \autoref{minimal-polar-divisor-linear-connection}.}) at a regular singularity must be equal to $0$. Singularities that are not regular are of course called \emph{irregular singularities}. Though this terminology is well-established, some authors prefer the adjectives \emph{tame} (or \emph{moderate}) and \emph{wild}, which are less confusing.

A regular singularity is called \emph{apparent}\footnote{cf. \cite[Sec. 2.3]{MR4286048} for a review of other definitions appearing in the literature.} if $P^*$ has trivial local monodromy, else it is called \emph{logarithmic}. %

Apparent singularities correspond exactly to \emph{branch points} of the projective structure, i.e. points for which there exists an integer $k\in\mathbb{N}$ such that projective charts read $\varphi(x) = x^k$ in some local coordinates. About branch points, see also \autoref{apparent-sing-and-tangency}. In the present paper, we will mainly focus on projective structures without branch points. Branched projective structures where studied by Mandelbaum \cite{MR288253}.

Let $p$ be a regular singularity of $P$, and write $\phi=\frac{q(x)}{2}dx^{\otimes 2}$ for a polar differential in some coordinate chart $x$ centered at $p$. We define the \emph{residue}\footnote{Here, we follow the terminology and notations of \cite[Sec. IX.1]{MR3494804}, which are a slightly different from the ones in \cite{Allegretti_2020}: our $q$ is equal to their $-2\varphi$, hence their definitions of "leading coefficient", "residue" and "exponent" does not match ours.} (of order $2$) of  $P$ at $p$ to be \begin{equation}\label{def:residue-of-order-two}
\operatorname{res}_2(p)=\lim_{x\to0}x^2q(x).
\end{equation}
This definition is independent of the polar differential and of the coordinate chart. Latter, we will also call "residue at $p$" a quantity different from $\operatorname{res}_2(p)$, but related (see \autoref{rk:residues}).

The \emph{index} $\theta(p)$ of $P$ at $p$, is then defined up to a multiplication by $-1$ as \[\theta(p)=\pm\sqrt{1-2\operatorname{res}_2(p)},\] in other words, writing the Schwarzian equation~\eqref{schwarzian-equation} for a projective chart $\varphi$ around $p$ (in a complex coordinate $x$ centered at $p$) we have \[\mathcal{S}_x(\varphi)=q(x)=\frac{1-\theta(p)^2}{2 x^2}+o(x^{-2}).\] Projective charts and monodromy around regular singularities may be studied according to the value of the index $\theta(p)$. For a detailed analysis, the reader may refer to~\cite[Sec. IX.1]{MR3494804}.

Consider the companion system of the second order equation \eqref{second-order-reduced-scalar-eq} (cf. \autoref{monodromy-scalareq-system}), or more naturally equation \eqref{system},
\begin{equation}\label{companion-system}
dY+\begin{pmatrix}
0 & -1\\
q(x)/2 & 0
\end{pmatrix}Ydx=0.
\end{equation}
Up to conjugacy, the projectivization of the monodromy of this system corresponds to the local monodromy of the projective structure.

\begin{prop}[cf. \cite{Allegretti_2020}, Lem. 5.1]\label{eigenvalues}
	Assume that the system \eqref{companion-system} has a regular singularity at $z=0$. Then, the eigenvalues of its monodromy around $z=0$ are \[\xi^\pm=\exp\left(2i\pi\left(\frac{1\pm\theta}{2}\right)\right),\] where $\theta=\theta(0)$ is defined up to a multiplication by $-1$ as above. In particular, the eigenvalues $\xi^\pm$ are distinct if and only if $\pm\theta\in\mathbb{C}\smallsetminus\mathbb{Z}$.
\end{prop}

Several possibilities occur for the local monodromy of the projective structure, depending on the value of the index:

\begin{itemize}
	\item if $\pm\theta\in\mathbb{C}\smallsetminus\mathbb{Z}$, it is non-trivial and non-parabolic,
	\item if $\pm\theta\in\mathbb{Z}$, it is trivial or parabolic.
\end{itemize}

\begin{rmk}\label{rmk:parabolic-local-monodromy}
	In order to construct the \emph{monodromy framed local system} of a projective structure, Allegretti and Bridgeland needed to chose a point of $\mathbb{P}^1$ fixed by the local monodromy of each regular singularity.
	Such a choice is encoded by the notion of \emph{signed} meromorphic projective structure: a signing is just a choice of sign "$+$" or "$-$" at some of the regular singularities. In \cite[Prop. 8.4]{Allegretti_2020}, the signing considered is the one given by a "choice of sign" of the index $\theta(p)=\pm\sqrt{1-2\operatorname{res}_2(p)}$ at each regular singularity $p$, i.e. a determination of the square root. The square root branches over $\theta=0$, so that no sign is chosen in this case, which is fine because the local monodromy is then parabolic~\cite[Sec. IX.1]{MR3494804}, hence it has a unique fixed point. However, the square root does not branch over $\theta\in\mathbb{Z}\smallsetminus\{0\}$, hence a choice is made here although it is not really needed.

	Those points of the moduli space of projective structures where parabolic local monodromy occur (we always exclude apparent singularities) causes a failure in the local injectivity of the monodromy map.
	
	In the introduction, we have mentioned several works generalizing Hejhal's theorem for meromorphic projective structures. Luo \cite{MR1202134} considered the case where all singularities are regular (of order exactly $2$) with an index  $\theta\in\mathbb{C}\smallsetminus\mathbb{Z}$, thus excluding non-trivial and non-parabolic local monodromy. On the other hand, Hussenot Desenonges \cite{MR3956187} considered a case where all singularities are regular (of order $2$) with an index $\theta=0\in\mathbb{Z}$ (hence a parabolic local monodromy). Thus, Luo does not fix the indices but excludes parabolic local monodromies, whereas Hussenot Desenonges considers only parabolic local monodromies but fixes the indices. In the present paper, we will fix the values of the indices and we will not exclude parabolic local monodromies.

	As for Gupta and Mj \cite{MR4232545}, they dealt with projective structures with only irregular singularities, hence this issue does not arise in their work.
\end{rmk}

\begin{exe}\label{hypergeometric-equation}
	The family of Gauss hypergeometric equations \[x(x-1)\frac{d^2y}{dx^2}+[(\alpha+\beta+1)x-\gamma]\frac{dy}{dx}+\alpha\beta y=0,\textmd{ with $\alpha,\beta,\gamma\in\mathbb{C}$ and $x\in\mathbb{C}$}\] induces a family of meromorphic projective structures on $\mathbb{P}^1$ with $3$ regular singularities at $0$, $1$ and $\infty$ \cite[Sec. IX.2]{MR3494804}.		
	
	In fact, every projective structure of this type is induced by a hypergeometric equation. Projective charts take the form \[\varphi(x)=f(x)^\theta\textmd{ or }\varphi(x)=f(x)^\theta+\operatorname{log}(f(x))\] where $\theta\in\mathbb{C}$ depends on $(\alpha,\beta,\gamma)$, and $f$ is a local complex coordinate around a singularity $x_0$, $f(x_0)=0$.
\end{exe}

\begin{exe}
	The Heun equations induces projective structures on $\mathbb{P}^1$ with $4$ regular singularities (see \cite[Sec. IX.3.1]{MR3494804}).
\end{exe}

\begin{exe}
	An example on the torus is provided by the Lamé equation (see \cite[IX.3.3]{MR3494804}).
\end{exe}

The subsequent paragraph is devoted to the description of a moduli space of meromorphic projective structures.

\paragraph{The moduli space $\mathcal{P}(X/B;\mathcal{D})$.} Let $d\in\mathbb{N}$ be an integer and $(n_i)\in\mathbb{N}^d$ be a collection of integers and
\begin{equation*}\xymatrix{
	f:X \ar[r] & B \ar@/_1pc/_-{p_i}[l]
}\end{equation*} %
a holomorphic family of compact smooth curves, together with $d$ disjoint holomorphic sections $p_1,\dots,p_d$. We can define an effective divisor $\mathcal{D}$ as $\mathcal{D}:=\sum_{i=1}^dn_i\mathcal{D}_i$ with $\mathcal{D}_i=\operatorname{im}(p_i)$. Its restriction to $C_t=f^{-1}(t)$, for some $t\in B$, will be denoted $D_t=\sum_{i=1}^dn_ip_i(t)$. There is a holomorphic vector bundle \[\mathcal{Q}(X/B;\mathcal{D})\rightarrow B\] whose fiber over a point $t\in B$ is the vector space $\mathcal{Q}(X/B;\mathcal{D})_t=H^0(C_t,(T^*C_t)^{\otimes2}(D_t))$ of meromorphic quadratic differentials on $C_t$ having poles of order at most $n_i$ at the points $p_i(t)$ and no other pole. As a consequence of the Riemann-Roch theorem, \begin{equation}\label{Riemann-Roch}
\dim(H^0(C_t,(T^*C_t)^{\otimes2}(D_t)))=3g-3+\sum_{i=1}^{d} n_i.
\end{equation}

Let us denote $\mathcal{P}(X/B)_t$ the fiber of the projection of $\mathcal{P}(X/B)$ to $B$. Then, define $\mathcal{P}(X/B;\mathcal{D})$ as the affine bundle (for the vector bundle $\mathcal{Q}(X/B;\mathcal{D})$) whose fiber over $t$ is the quotient \[\mathcal{P}(X/B;\mathcal{D})_t:=(\mathcal{P}(X/B)_t\times Q(X/B;\mathcal{D})_t)/\sim\] for the equivalence relation \[((t,P_1),\phi_1)\sim((t,P_2),\phi_2)\Leftrightarrow(P_2,\phi_2)=(P_1+\phi_{\operatorname{hol}},\phi_1-\phi_{\operatorname{hol}})\] for some $\phi_{\operatorname{hol}}\in Q(X/B)$. The resulting quotient space is in bijection with the set
\begin{align*}
\{(t,P): ~&t\in B\textmd{ and $P$ is a meromorphic projective structure on $C_t$}\\
&\textmd{having poles of order at most $n_i$ at the points $p_i(t)$ and no other pole}\}.
\end{align*}
The space $\mathcal{Q}(X/B;\mathcal{D})_t$ acts transitively and freely on equivalence classes by $(\tilde{\phi},t)\cdot[(t,P),\phi]=[(t,P),\phi+\tilde{\phi}]$.

A \emph{family of meromorphic projective structures} on $f$ relative to the divisor $\mathcal{D}$ is defined to be a holomorphic section of this affine bundle $\mathcal{P}(X/B;\mathcal{D})\rightarrow B$. It induces a relative projective structure on the restriction of $f$ to $X\smallsetminus\mathcal{D}$ \cite[Sec. 7.3]{Allegretti_2020}.

The subset corresponding to projective structures with poles of orders exactly $(n_i)$ forms a dense open set \[\mathcal{P}(X/B;\mathcal{D},(n_i))\subset\mathcal{P}(X/B;\mathcal{D})\] (see \cite[Lem. 8.1]{Allegretti_2020}).

As usual, denote by $S$ a smooth oriented compact real surface of genus $g$. If $B=\mathcal{T}(S,d)$ is the Teichmüller space of marked curves (i.e. marked Riemann surfaces) on $S$ equipped with $d$ marked points, and if $f$ is the universal holomorphic family of marked curves, we chose the sections $p_i$ such that for $t\in\mathcal{T}(S,d)$, the $p_i(t)$ correspond to marked points on $C_t$.

\paragraph{Marked meromorphic projective structures.}

Take a meromorphic projective structure $P$ on $C$ with pole orders $(n_i)$, and perform a real oriented blow-up\footnote{Locally, the blow down is: $\mathbb{R}_{\geq 0}\times S^1\longrightarrow\mathbb{R}^2\text{, }(r,\theta)\longmapsto(r\operatorname{cos}(\theta),r\operatorname{sin}(\theta))$.} at each irregular singularity. Consider one of them, say $p_i$. Then, the $n_i-2$ corresponding \emph{asymptotic horizontal directions} (see~\cite{Allegretti_2020}) of any polar differential of $P$ at $p_i$ define $n_i-2$ points on the boundary component of the resulting blown-up surface $\mathbb{S}_P$ lying over $p_i$. Together with the set $\mathbb{P}_P$ of points in the interior of $\mathbb{S}_P$ corresponding to regular singularities (the \emph{punctures}), they form the set $\mathbb{M}_P\subset\mathbb{S}_P$ of \emph{marked points} (cf. \autoref{fig:blow-up}).

A \emph{marked bordered surface} is defined to be a pair $(\mathbb{S},\mathbb{M})$ consisting of a compact, connected, oriented, smooth surface with boundary $\mathbb{S}$ and a finite non-empty set $\mathbb{M}\subset\mathbb{S}$ of marked points such that each boundary component of $\mathbb{S}$ contains at least one marked point (we refer to the definition given in \cite[Sec. 3.6]{Allegretti_2020}). Let $\mathbb{S}'$ denote the surface obtained by taking the real oriented blow-up of $\mathbb{S}$ at each puncture: this replaces punctures with boundary components containing no marked points. A marked bordered surface $(\mathbb{S},\mathbb{M})$ is determined up to isomorphism by its genus and a collection of non-negative integers $\{k_1,\dots,k_d\}$ giving the number of marked points on the $d$ boundary components of $\mathbb{S}'$. For example, $(\mathbb{S}_P,\mathbb{M}_P)$ is a marked bordered surface determined by the genus of $P$ and the numbers \begin{equation}\label{k_iversusn_i}
k_i=\begin{cases*}0\text{ if }n_i\leq 2\\n_i-2\text{ otherwise.}\end{cases*}
\end{equation}

\begin{figure}[H]
	\centering
	\def\svgwidth{0.6\textwidth}
	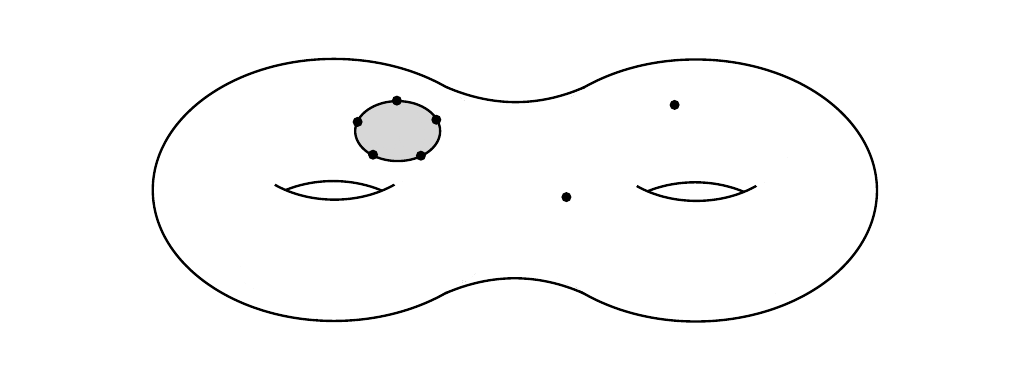
	\caption{The surface $\mathbb{S}_P$ with marked points $\mathbb{M}_P$.}
	\label{fig:blow-up}
\end{figure}

A \emph{marking of the meromorphic projective structure $P$ by a marked bordered surface $(\mathbb{S},\mathbb{M})$} is defined to be an isotopy class (relative to the boundaries and the marked points) of orientation preserving $\mathcal{C}^\infty$-diffeomorphisms $\theta:\mathbb{S}\rightarrow\mathbb{S}_P$ such that $\theta(\mathbb{M})=\mathbb{M}_P$.\footnote{The terminology is a little confusing here: we just defined a "marking" of $P$ by a "marked" bordered surface. The "marking" of $P$ refers to $\theta$, whereas "marked" bordered surface is a surface $\mathbb{S}$ with marked points $\mathbb{M}$.}

\begin{defi}
	A \emph{marked meromorphic projective structure} is a triple $(C,P,\theta)$, where $C$ is a complex curve, $P$ is a meromorphic projective structure on $C$, and $\theta$ is a marking of $P$ by a marked bordered surface $(\mathbb{S},\mathbb{M})$.
	
	Fix a marked bordered surface $(\mathbb{S},\mathbb{M})$. Two meromorphic projective structures marked by $(\mathbb{S},\mathbb{M})$, $(C_i,P_i,\theta_i)$, $i=1,2$, are said to be \emph{equivalent} if there exists a biholomorphism $C_1\rightarrow C_2$ that preserves the projective structures and whose lift to the blown-up surfaces commutes with the markings.
\end{defi}

\paragraph{The moduli space $\mathcal{P}(\mathbb{S},\mathbb{M})$.} Let $(\mathbb{S},\mathbb{M})$ be a fixed marked bordered surface determined by a genus $g$ and a non-empty collection of non-negative integers $\{k_1,\dots,k_d\}$. From now on, suppose $B=\mathcal{T}(S,d)$. Let $\mathcal{D}$ be the divisor on $X$ defined by \[\mathcal{D}:=\sum_{i=1}^d (k_i+2)\mathcal{D}_i.\] Thanks to \cite[Lem. 8.1, Prop. 8.2]{Allegretti_2020}, we know that there is a dense subset $\mathcal{V}$ of $\mathcal{P}(X/B;\mathcal{D})$ and a covering map \begin{align}\label{marquing-covering-map}
M:\mathcal{P}(\mathbb{S},\mathbb{M})\rightarrow \mathcal{V}\subset\mathcal{P}(X/B;\mathcal{D})
\end{align} whose fibers over a point $[(t,P),\phi]$ is the set of markings of $P$ by $(\mathbb{S},\mathbb{M})$ inducing (by blowing down the boundary component of $\mathbb{S}$) the marking of $S$ corresponding to $t$. This covering space $\mathcal{P}(\mathbb{S},\mathbb{M})$ is in bijection with the set of equivalence classes of marked meromorphic projective structures marked by $(\mathbb{S}, \mathbb{M})$.

\begin{rmk}\label{rmk:pole-orders}
	The moduli space $\mathcal{P}(\mathbb{S},\mathbb{M})$ was introduced by Allegretti and Bridgeland in \cite{Allegretti_2020}. We would like to emphasize that the marked bordered surface $(\mathbb{S},\mathbb{M})$ does not always determine the pole orders of the projective structures (the exact orders of regular singularities are not specified). Thus, if $p_i\in C$ corresponds to a point in the interior of $\mathbb{S}$, the open subset $\mathcal{V}$ contains projective structures with a simple pole at $p_i$ and others with a double pole at $p_i$. Projective structures with simple poles appear in the boundary of spaces of projective structures with double poles.
	
	This is natural for Allegretti and Bridgeland because of the result of Bridgeland and Smith \cite{MR3349833} which identifies the corresponding unions of components of moduli spaces of meromorphic quadratic differentials with spaces of stability conditions.
	
	However, in this paper we will soon fix the residues of the meromorphic projective structures (see \autoref{prop:fix-residues}), hence the orders of regular poles as well.
\end{rmk}

Allegretti and Bridgeland showed the following proposition.

\begin{prop}[Proposition 8.2 of \cite{Allegretti_2020}]\label{prop-smooth-moduli-proj-struct}
	Assume that if $g=0$, then $|\mathbb{M}|\geq3$. Then $\mathcal{P}(\mathbb{S},\mathbb{M})$ is a complex manifold of dimension \begin{equation*}
	6g-6+\sum_{i=1}^d(k_i+3).
	\end{equation*}
\end{prop}

According to the relation \eqref{k_iversusn_i}, special cases satisfying $g=0$ and $|\mathbb{M}|\leq2$ correspond to non-singular projective structures and cases when $(n_i)$ is equal, up to a permutation of entries, to $(1)$, $(2)$, $(3)$, $(4)$, $(1,1)$, $(1,2)$, $(2,2)$, $(1,3)$, $(2,3)$ and $(3,3)$ (about those special cases, see \cite[Sec. 6.4]{Allegretti_2020}).

\paragraph{Submanifolds of $\mathcal{P}(\mathbb{S},\mathbb{M})$.} Let us denote by $\mathbb{P}$ the set of marked points points in the interior of $\mathbb{S}$ (i.e. those corresponding to regular singularities). There is a map \[R:\mathcal{P}(\mathbb{S},\mathbb{M})\longrightarrow\mathbb{C}^\mathbb{P}\]
sending a class of marked meromorphic projective structures to the corresponding residue at each regular singularity. This map is a holomorphic submersion \cite[Prop. 8.4]{Allegretti_2020}. We introduce the following submanifolds of $\mathcal{P}(\mathbb{S},\mathbb{M})$. First, assigning the values $(\mu_i)\in\mathbb{C}^\mathbb{P}$ for the residues leads to \[\mathcal{P}(\mathbb{S},\mathbb{M},(\mu_i)):=R^{-1}((\mu_i)_i).\]
Second, if $Z=\cup_{\theta\in\mathbb{N}}\{z\in\mathbb{C}:z-(1-\theta^2)/2=0\}$ (which is a discrete set of points of the real line in $\mathbb{C}$), then
\[\mathcal{P}^\bullet(\mathbb{S},\mathbb{M}):=R^{-1}(\mathbb{C}^\mathbb{P}\smallsetminus Z^\mathbb{P}),\]
is a dense open set in $\mathcal{P}(\mathbb{S},\mathbb{M})$ corresponding to projective structures having non-trivial and non-parabolic local monodromies around regular singular points.
Third, there is an open subset \[\mathcal{P}^\circ(\mathbb{S},\mathbb{M})\] in $\mathcal{P}(\mathbb{S},\mathbb{M})$ corresponding to marked projective structures without apparent singularity \cite[Lem. 8.3]{Allegretti_2020} (it contains $\mathcal{P}^\bullet(\mathbb{S},\mathbb{M})$).

Latter, we will introduce a new notion of "residue" which apply not only to regular singularities but also to irregular ones, and we will consider the submanifolds of $\mathcal{P}^\circ(\mathbb{S},\mathbb{M})$ defined by fixing those residues (cf. \autoref{prop:fix-residues}).

Let us mention the work of Billon \cite{billon2022deformations}, who constructed smooth moduli spaces of branched projective structures, i.e. meromorphic projective structures having only apparent singularities.

\section{A smooth family of opers}

We have seen already that a complex projective structure may be described in several ways: as a projective atlas, as a development-holonomy pair or as a projective connection. Yet another viewpoint turns out to be the most appropriate to develop our arguments: \emph{opers}. A \emph{$\operatorname{PGL}(2,\mathbb{C})$-opers} is a Riccati foliation of a $\mathbb{P}^1$-bundle over a curve, equipped with a transverse section. In the meromorphic case, this leads a priori to consider singular Riccati foliations of a $\mathbb{P}^1$-bundle, equipped with a section that might be tangent or might go through singularities of the foliation. However, in contrast with the holomorphic case, it is not uniquely determined by the projective structure: there can be different birational models \cite{MR3328860}. As observed by Loray and Pereira \cite{MR2337401}, there is however a unique birational model minimizing the polar divisor and such that the section does not intersect the singular set of the foliation. In the second section, we deduce that if the projective structure has no apparent singularity, then the above ambiguity can be resolved, so that there is a bijection between the set of meromorphic projective structures without apparent singularity on a curve $C$ and the set of meromorphic $\operatorname{PGL}(2,\mathbb{C})$-oper on $C$ with minimal polar divisor. Finally, in the third section, we construct a smooth family of such opers.

\subsection{Opers for the group \texorpdfstring{$\operatorname{PGL}(2,\mathbb{C})$}{PGL(2,C)} over a smooth curve \texorpdfstring{$C$}{C}}

Let us define opers, first in the case without poles.

\begin{defi}\label{defi:holo-PGL-oper}
	An \emph{oper for the group $\operatorname{PGL}(2,\mathbb{C})$ on $C$} (or $\operatorname{PGL}(2,\mathbb{C})$-oper on $C$) is defined to be a triple $(\pi:Q\rightarrow C,\mathcal{F},\sigma)$, where
	\begin{itemize}
		\item $\pi:Q\rightarrow C$ is a holomorphic $\mathbb{P}^1$-bundle over C,
		\item $\mathcal{F}$ is a regular holomorphic foliation on $Q$ transverse to the fibers of $\pi$ (i.e. a Riccati foliation, cf. below),
		\item $\sigma:C\rightarrow Q$ is a holomorphic section of the bundle, transverse to the foliation $\mathcal{F}$.
	\end{itemize}
\end{defi}

\begin{figure}[H]
	\centering
	\def\svgwidth{0.6\textwidth}
\begingroup%
  \makeatletter%
  \providecommand\color[2][]{%
    \errmessage{(Inkscape) Color is used for the text in Inkscape, but the package 'color.sty' is not loaded}%
    \renewcommand\color[2][]{}%
  }%
  \providecommand\transparent[1]{%
    \errmessage{(Inkscape) Transparency is used (non-zero) for the text in Inkscape, but the package 'transparent.sty' is not loaded}%
    \renewcommand\transparent[1]{}%
  }%
  \providecommand\rotatebox[2]{#2}%
  \newcommand*\fsize{\dimexpr\f@size pt\relax}%
  \newcommand*\lineheight[1]{\fontsize{\fsize}{#1\fsize}\selectfont}%
  \ifx\svgwidth\undefined%
    \setlength{\unitlength}{491.08856441bp}%
    \ifx\svgscale\undefined%
      \relax%
    \else%
      \setlength{\unitlength}{\unitlength * \real{\svgscale}}%
    \fi%
  \else%
    \setlength{\unitlength}{\svgwidth}%
  \fi%
  \global\let\svgwidth\undefined%
  \global\let\svgscale\undefined%
  \makeatother%
  \begin{picture}(1,0.60078165)%
    \lineheight{1}%
    \setlength\tabcolsep{0pt}%
    \put(0,0){\includegraphics[width=\unitlength,page=1]{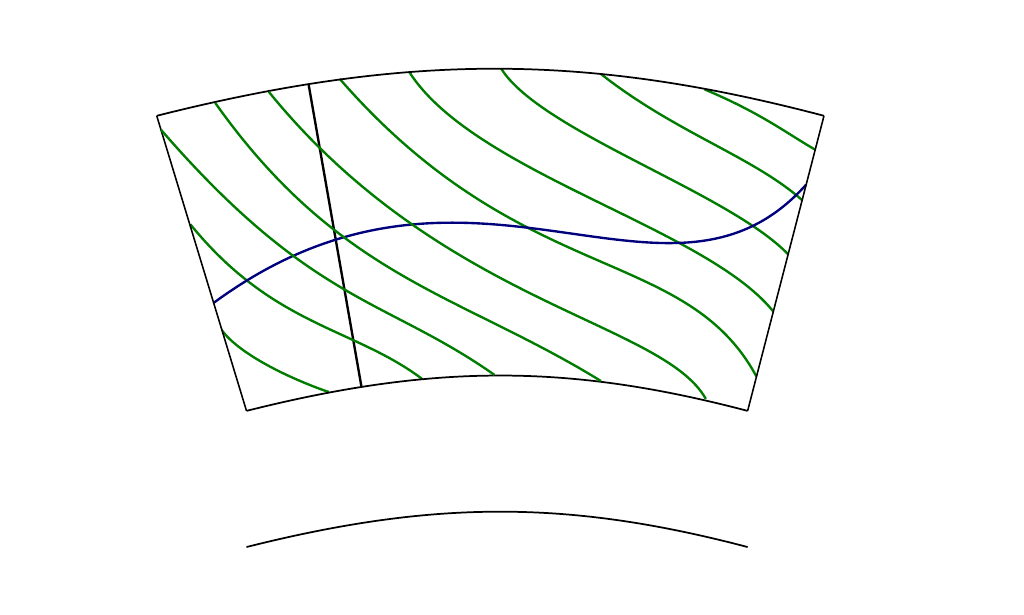}}%
    \put(0.75848321,0.08571218){\color[rgb]{0,0,0}\makebox(0,0)[lt]{\lineheight{1.25}\smash{\begin{tabular}[t]{l}$C$\end{tabular}}}}%
    \put(0.1049744,0.40030167){\color[rgb]{0,0,0}\makebox(0,0)[lt]{\lineheight{1.25}\smash{\begin{tabular}[t]{l}$Q$\end{tabular}}}}%
    \put(0.79864917,0.41156755){\color[rgb]{0,0,0.49019608}\makebox(0,0)[lt]{\lineheight{1.25}\smash{\begin{tabular}[t]{l}$\sigma$\end{tabular}}}}%
    \put(0.4818453,0.21601762){\color[rgb]{0,0,0}\rotatebox{-90}{\makebox(0,0)[lt]{\lineheight{1.25}\smash{\begin{tabular}[t]{l}$\longrightarrow$\end{tabular}}}}}%
    \put(0.59891054,0.27743859){\color[rgb]{0,0.49019608,0}\makebox(0,0)[lt]{\lineheight{1.25}\smash{\begin{tabular}[t]{l}$\mathcal{F}$\end{tabular}}}}%
    \put(0.51941891,0.16625081){\color[rgb]{0,0,0}\makebox(0,0)[lt]{\lineheight{1.25}\smash{\begin{tabular}[t]{l}$\pi$\end{tabular}}}}%
  \end{picture}%
\endgroup%

	\caption{A $\operatorname{PGL}(2,\mathbb{C})$-oper on $C$.}
\end{figure}

This definition makes sense for non-compact curves $C$ as well. This terminology was introduced by Beilinson and Drinfeld \cite{https://doi.org/10.48550/arxiv.math/0501398} (see also \cite{MR2332156}) and is a particular case of a much more general class of objects, including opers for groups such as $\operatorname{GL}(2,\mathbb{C})$ or $\operatorname{SL}(2,\mathbb{C})$. Another formulation of the above definition will be given later, in terms of equivalence classes of opers for the group $\operatorname{GL}(2,\mathbb{C})$ (cf. \autoref{defi:GL-oper}).

	\paragraph{Opers and projective structures.} Let $C$ be a smooth curve, and $P$ be a projective structure on $C$, with its projective atlas denoted by $(U_i,\varphi_i)_i$.
	
	Then, the coordinate changes $\varphi_{ij}:=\varphi_i\circ\varphi_j^{-1}:\varphi_j(U_i\cap U_j)\rightarrow\varphi_i(U_i\cap U_j)$ are restrictions of Möbius transformations and satisfy, on $\varphi_j(U_i\cap U_k\cap U_j)$, the cocycle relation $\varphi_{ik}\circ\varphi_{kj}=\varphi_{ij}$. Whence, we can define a map $g_{ij}$ with value in $\operatorname{PGL}(2,\mathbb{C})$: to $x\in U_i\cap U_j$ we associate the automorphism $g_{ij}(x)$ of $\mathbb{P}^1$ whose restriction equals $\varphi_{ij}$. The $g_{ij}$ also satisfy a cocycle relation, hence they define a $\mathbb{P}^1$-bundle \[\pi:Q=\sqcup_i(U_i\times\mathbb{P}^1)/((j,x,y)\sim (i,x,g_{ij}(x)\cdot y)) \longrightarrow C.\]
	
	By construction, the $g_{ij}$ are locally constant thus the transition functions of the bundle $\pi$ do not depend on the variable $x$ and it makes sense to speak of the horizontal foliation on a trivial chart of $\pi$. This foliation descends to a foliation $\mathcal{F}$ on $P$ transverse to the fibers of $\pi$.
	
	Furthermore, the graph of each projective chart $\varphi_i$ defines a local section on $U_i\times\mathbb{P}^1$. The open cover and the gluing data of the bundle $\pi$ correspond exactly to those of the projective atlas of $P$, so that projective charts define a global section $\sigma$ of $\pi$. This section is transverse to the foliation $\mathcal{F}$ because the derivatives of the $\varphi_i$ never vanish.
	
	\begin{figure}[H]
		\centering
		\def\svgwidth{0.8\textwidth}
		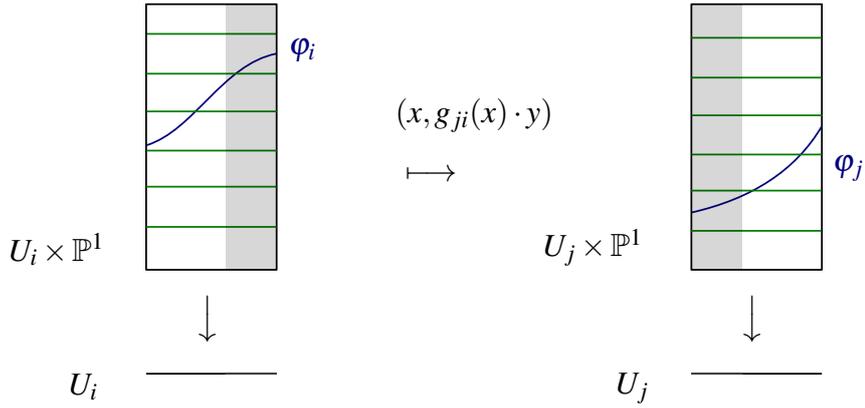
		\caption{The gluing construction.}
	\end{figure}
	
	Finally, we have constructed a $\operatorname{PGL}(2,\mathbb{C})$-oper $(\pi:Q\rightarrow C,\mathcal{F},\sigma)$ on $C$.
	
	Let us describe the inverse construction. Given the triple $(\pi:Q\rightarrow C,\mathcal{F},\sigma)$, we can recover the projective structure $P$ in the following way. First, the flow box theorem ensures that on the trivialization charts there exists a bundle isomorphism sending the Riccati foliation to the horizontal one. (Concretely, this bundle isomorphism may be described using the sharp $3$-transitivity of the action of the group of Möbius transformations on $\mathbb{P}^1$: choose three distinct leaves and send them to the constant sections $\{y= 0\}$, $\{y=1\}$ and $\{y=\infty\}$). Then, restricted to the trivializing open sets, the graph of the section $\sigma$ defines a projective charts. Finally, the compatibility between charts follows from the gluing data of the bundle. (This projective structure induces the complex structure $C$ because the section is analytic).
	
	Let us formulate this correspondence as a proposition.

\begin{prop}\label{indigenousbundle}
	The above correspondence defines a bijection between the set of projective structures on $C$ and the set $\operatorname{PGL}(2,\mathbb{C})$-opers on $C$.
\end{prop}

Note that the conjugacy class of the monodromy representations of $\mathcal{F}$ and the associated projective structure are the same, because transition maps of the projective atlas are the same as cocycles of the bundle.

The isomorphism class of the couple $(\pi:Q\rightarrow C,\mathcal{F})$ can be recovered from the data of the complex curve $C$ together with its monodromy representation up to conjugacy. This is the classical Riemann-Hilbert correspondence, which in this context amounts to the suspension of the monodromy representation.\\

In a way analogous to \autoref{indigenousbundle}, meromorphic projective structures correspond to \emph{meromorphic opers}. The rest of the present section is devoted to the definition of those objects and to the clarification of their relation to meromorphic projective structures.

\paragraph{Singular Riccati foliations.} Let $\pi:Q\rightarrow C$ be a holomorphic $\mathbb{P}^1$-bundle. A singular holomorphic foliation $\mathcal{F}$ on $Q$ is said to be \emph{Riccati} with respect to $\pi$ if it is transverse to the generic fiber of $\pi$. Let $(U,x)$ be a coordinate chart of $C$ such that the bundle is trivial over $U$, i.e. $\pi^{-1}(U)\simeq U\times\mathbb{P}^1$, and choose a standard coordinate $y$ on $\mathbb{P}^1$. Then, is these coordinates the foliation $\mathcal{F}$ is described by a Pfaffian equation of the form
\begin{equation}
d(x)dy+[a(x)y^2+b(x)y+c(x)]dx=0,
\end{equation}
where $a,b,c$ and $d$ are holomorphic functions and $d$ is non-identically vanishing (a detailed study of Riccati foliations is exposed in \cite[Chap. 4.1]{MR3328860}; see also \cite[Prop. VIII.1.1]{MR3494804}). A fiber of $\pi$ over a point corresponding to a zero of $d$ is invariant by $\mathcal{F}$, whereas fibers over points where $d$ is non-zero are transverse to $\mathcal{F}$. The set of singularities $\operatorname{sing}(\mathcal{F})$ of the foliation is contained in the union of invariant fibers. If $F$ is an invariant fiber, then the set $\operatorname{sing}(\mathcal{F})\cap F$ contains one or two elements since they are obtained by solving a quadratic equation.

Dividing by $d$, we obtain a Riccati equation (hence the terminology):
\begin{equation}\label{riccati}
\mathcal{F}_{|U}:dy+y^2\alpha +y\beta+\gamma=0,
\end{equation} where $\alpha,\beta$ and $\gamma$ are meromorphic $1$-forms on $U$. Note that equation \eqref{riccati} also determines $\mathcal{F}$: the leaves are obtained as the graphs of its solutions -- this is possible only outside fibers corresponding to the poles of $\alpha,\beta$ and $\gamma$, but then the resulting foliation extends uniquely outside a codimension $1$ analytic subset (this is basic foliation theory, see for instance \cite[Prop. 27]{loray2} or \cite[Thm. 2.20]{MR2363178}). Invariant fibers correspond exactly to the poles of $\alpha,\beta$ and $\gamma$.

\begin{defi}\label{def-Riccati-polar-divisor}
	To a singular Riccati foliation $(\pi:Q\rightarrow C,\mathcal{F})$, we can associate a \emph{polar divisor} on $C$ defined as the direct image by $\pi$ of the tangency divisor between $\mathcal{F}$ and the foliation induced by the fibers of $\pi$.
\end{defi}

A Möbius transformation fixes the point at infinity in the projective line if and only if it is affine. Recall that the affine group is generated by translations and dilatations by a non-zero complex number. Affine holomorphic gauge transformations act on the Riccati equation as follow. A pullback by a gauge transformation of the form $y=a(x)\tilde{y}$ transforms the Riccati equation \eqref{riccati} into\begin{equation}\label{action-gauge-multiplication}
d\tilde{y}+a\tilde{y}^2\alpha+\tilde{y}\left(\beta+\frac{da}{a}\right)+\frac{\gamma}{a}=0
\end{equation}
whereas a pullback by a gauge transformation of the form $y=\tilde{y}+b(x)$ transforms it into \begin{equation}\label{action-gauge-translation}
d\tilde{y}+\tilde{y}^2\alpha+\tilde{y}(\beta+2b\alpha)+(b^2\alpha+db+b\beta+\gamma)=0.
\end{equation}

The following proposition gives the relation between Riccati foliations and the quadratic differentials associated with projective structures.

\begin{prop}[see Prop. 2.1 in \cite{MR2647972}]\label{foliation-vs-schwarzianeq}
	Let $(\pi:Q\rightarrow C,\mathcal{F},\sigma)$ be a $\operatorname{PGL}(2,\mathbb{C})$-oper defining a projective structure $P$ on a smooth curve $C$. Let $x$ be a coordinate chart of $C$ defined on an open set $U$ over which the bundle $\pi$ is trivial\footnote{In fact, it is trivial on any proper open subset $U\subsetneq C$, because $\pi$ is the projectivization of a rank $2$ vector bundle, which is trivial on any non-compact curve by Grauert-Röhrl theorem.}. Then
	\begin{itemize}
		\item[(i)] there exists a unique coordinate chart $y$ on $\mathbb{C}_y\subset\mathbb{P}^1$ in which the triple takes the normal form \begin{equation}\label{normal-form}
		\begin{cases}
		\mathcal{F}_{|U}:dy+\left(y^2+\frac{q(x)}{2}\right)dx=0\textmd{ with $q\in\mathcal{O}(U)$}\\
		\sigma_{|U}:\{y=\infty\}.
		\end{cases}
		\end{equation}
		\item[(ii)] Moreover, on $U$, the projective charts of the structure $P$ are exactly the solutions of the Schwarzian equation $\mathcal{S}_x(\varphi)=q$.
	\end{itemize}
\end{prop}

Note that, after the substitution $y=-1/\tilde{y}$, we see that the same is true for the form \begin{equation}
\begin{cases}
\mathcal{F}_{|U}:d\tilde{y}+\left(\frac{q(x)}{2}\tilde{y}^2+1\right)dx=0\textmd{ with $q\in\mathcal{O}(U)$}\\
\sigma_{|U}:\{\tilde{y}=0\}.
\end{cases}
\end{equation}
which is the projectivization $\tilde{y}=\tilde{y}_1/\tilde{y}_2$ of the companion system of equation \eqref{second-order-reduced-scalar-eq},
\begin{equation}\label{system}
d\tilde{Y}+\begin{pmatrix}
0 & -1\\
q(x)/2 & 0
\end{pmatrix}\tilde{Y}dx=0,
\end{equation}
with the section $\sigma$ corresponding to the line subbundle $L=\braket{^t(0,1)}$. See \cite[Rk. 3.3(ii)]{Allegretti_2020} for the corresponding computation directly on the rank $2$ connection side.

\begin{proof}[Proof of \autoref{foliation-vs-schwarzianeq}]
	Let $U$ be an open set of $C$ such that the bundle $\pi$ is trivial over $U$.
	
	\emph{(i)} By transitivity of Möbius transformations, there exist trivializing coordinates $(x,y)$ on $\pi^{-1}(U)\simeq U\times\mathbb{P}^1$ in which $\sigma_{|U}=\{y=\infty\}$. As above, the Pfaffian equation reads \[\mathcal{F}_{|U}:dy+y^2\alpha +y\beta+\gamma=0,\] where $\alpha,\beta$ and $\gamma$ are holomorphic $1$-forms on $U$, with $\alpha$ non-vanishing (due to the transversality between the section and the foliation). The action of affine holomorphic gauge transformations (i.e. those fixing the section at infinity) on the Riccati equation is such that $\alpha$ and $\beta$ can be arbitrarily chosen, except $\alpha$ must remain non-vanishing (see equations \eqref{action-gauge-multiplication} and \eqref{action-gauge-translation}). After such a choice of $\alpha$ and $\beta$ is made, no non-trivial gauge transformation leaves them unchanged. Hence $\gamma$ is uniquely determined by the projective structure $P$. Choose $\alpha=dx$, $\beta=0$ and denote $\gamma=q/2dx$. (The factor $1/2$ is a matter of convention: if it wasn't written here, a factor $2$ would pop up in the Schwarzian equation of point \emph{(ii)} anyway). In practice, one can first normalize the quadratic coefficient to $1$ via $y=a(x)\tilde{y}$, then cancel the linear coefficient via $y=\tilde{y}+b(x)$ (see equations \eqref{action-gauge-multiplication} and \eqref{action-gauge-translation}).
	
	\emph{(ii)} Just as in the proof of \autoref{indigenousbundle}, we can find coordinates $(x,y)$ on $U\times\mathbb{P}^1$ in which the foliation $\mathcal{F}$ is given by $y'=0$ (i.e. horizontal) and the section $\sigma$ defines a projective chart $\varphi=\sigma_{|U}$.
	After a gauge transformation sending the section to infinity
	\[(x,y)\longmapsto(x,y_1):=\left(x,\frac{1}{y-\varphi}\right),\]
	the equation takes on the form $y_1'-\varphi'y_1^2=0$.
	Next, we apply change of the $y_1$ coordinate preserving $\{y_1=\infty\}$: those are exactly the elements of the affine group. First, the transformation
	\[(x,y_1)\longmapsto(x,y_2):=\left(x,-\varphi'y_1\right)\]
	followed by a multiplication of the equation by $-\varphi'$ normalizes the coefficient of $y_2^2$ to $1$ without perturbing the coefficient of $y_2'$. Second, the transformation \[(x,y_2)\longmapsto(x,y_3):=\left(x,y_2-\frac{\varphi''}{2\varphi'}\right)\] annihilates the coefficient of $y_3$ and we get
	\[y_3'+y_3^2+\frac{\mathcal{S}_x(\varphi)}{2}=0.\]
	The uniqueness in \emph{(i)} yields $\mathcal{S}_x(\varphi)=q$.
	
	Finally, two solutions of the Schwarzian equation differ from a Möbius transformation (compare \eqref{schwarzian-basic-property}), hence the conclusion.
\end{proof}

\paragraph{Meromorphic $\operatorname{PGL}(2,\mathbb{C})$-opers.}

\begin{defi}\label{def:mero-PGL-oper}
	A \emph{meromorphic $\operatorname{PGL}(2,\mathbb{C})$-oper on $C$} is a triple of the form $(\pi:Q\rightarrow C,\mathcal{F},\sigma)$, where
	\begin{itemize}
		\item $\pi:Q\rightarrow C$ is a holomorphic $\mathbb{P}^1$-bundle over C,
		\item $\mathcal{F}$ is a singular holomorphic foliation on $Q$ transverse to the generic fibers of $\pi$ (a singular Riccati foliation on $Q$),
		\item $\sigma:C\rightarrow Q$ is a holomorphic section of the bundle, transverse to the foliation $\mathcal{F}$.
	\end{itemize}
\end{defi}

\begin{figure}[H]
	\centering
	\def\svgwidth{0.6\textwidth}
\begingroup%
  \makeatletter%
  \providecommand\color[2][]{%
    \errmessage{(Inkscape) Color is used for the text in Inkscape, but the package 'color.sty' is not loaded}%
    \renewcommand\color[2][]{}%
  }%
  \providecommand\transparent[1]{%
    \errmessage{(Inkscape) Transparency is used (non-zero) for the text in Inkscape, but the package 'transparent.sty' is not loaded}%
    \renewcommand\transparent[1]{}%
  }%
  \providecommand\rotatebox[2]{#2}%
  \newcommand*\fsize{\dimexpr\f@size pt\relax}%
  \newcommand*\lineheight[1]{\fontsize{\fsize}{#1\fsize}\selectfont}%
  \ifx\svgwidth\undefined%
    \setlength{\unitlength}{491.08856441bp}%
    \ifx\svgscale\undefined%
      \relax%
    \else%
      \setlength{\unitlength}{\unitlength * \real{\svgscale}}%
    \fi%
  \else%
    \setlength{\unitlength}{\svgwidth}%
  \fi%
  \global\let\svgwidth\undefined%
  \global\let\svgscale\undefined%
  \makeatother%
  \begin{picture}(1,0.60078165)%
    \lineheight{1}%
    \setlength\tabcolsep{0pt}%
    \put(0,0){\includegraphics[width=\unitlength,page=1]{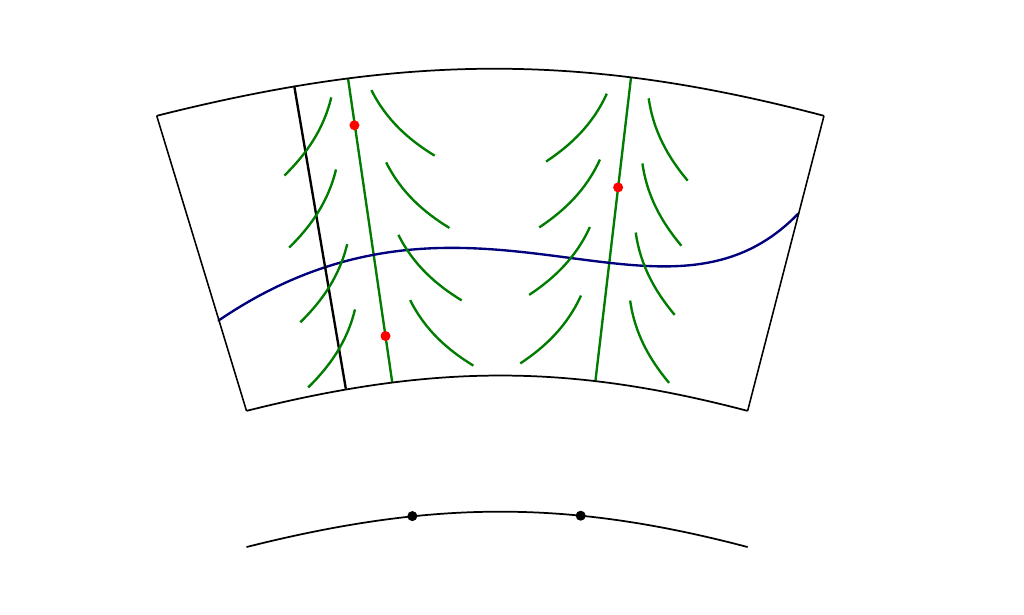}}%
    \put(0.75848321,0.08571197){\color[rgb]{0,0,0}\makebox(0,0)[lt]{\lineheight{1.25}\smash{\begin{tabular}[t]{l}$C$\end{tabular}}}}%
    \put(0.79813286,0.40072604){\color[rgb]{0,0,0.49019608}\makebox(0,0)[lt]{\lineheight{1.25}\smash{\begin{tabular}[t]{l}$\sigma$\end{tabular}}}}%
    \put(0.4818453,0.21056817){\color[rgb]{0,0,0}\rotatebox{-90}{\makebox(0,0)[lt]{\lineheight{1.25}\smash{\begin{tabular}[t]{l}$\longrightarrow$\end{tabular}}}}}%
    \put(0.11015068,0.41568474){\color[rgb]{0,0,0}\makebox(0,0)[lt]{\lineheight{1.25}\smash{\begin{tabular}[t]{l}$Q$\end{tabular}}}}%
    \put(0.52937541,0.15348003){\color[rgb]{0,0,0}\makebox(0,0)[lt]{\lineheight{1.25}\smash{\begin{tabular}[t]{l}$\pi$\end{tabular}}}}%
    \put(0.66500466,0.24253348){\color[rgb]{0,0.49019608,0}\makebox(0,0)[lt]{\lineheight{1.25}\smash{\begin{tabular}[t]{l}$\mathcal{F}$\end{tabular}}}}%
  \end{picture}%
\endgroup%

	\caption{A meromorphic $\operatorname{PGL}(2,\mathbb{C})$-oper on $C$.}
\end{figure}

The \emph{polar divisor} of a meromorphic $\operatorname{PGL}(2,\mathbb{C})$-oper on $C$ is defined to be the one of the associated singular Riccati foliation (cf. \autoref{def-Riccati-polar-divisor}).

\paragraph{"Singular" triples $(\pi,\mathcal{F},\sigma)$ associated with a meromorphic projective structure.} Let $P$ be a meromorphic projective structure on $C$ described by a pair $(P_0,\phi)$, following the notations of \autoref{meromorphicprojectivestructure}. Consider the trivial $\mathbb{P}^1$-bundle over $C$, a standard coordinate $y$ on $\mathbb{P}^1$, and the holomorphic section $\sigma=\{y=\infty\}$. Suppose that, in the charts of a projective atlas $(U_i,x_i)$ of $P_0$, the quadratic differential reads $\phi=\frac{q_i(x_i)}{2}dx_i^{\otimes2}$. Define \begin{equation}\label{construction-oper}
\mathcal{F}_i:dy+\left(y^2+\frac{q_i(x_i)}{2}\right)dx_i=0.
\end{equation}
Since $x_i$ is a projective chart for $P_0$, those expressions glue together to give a singular Riccati foliation $\mathcal{F}$ on the trivial $\mathbb{P}^1$-bundle over $C$. We get a triple $(\pi,\mathcal{F},\sigma)$, where $\pi$ is the trivial bundle, which by \autoref{foliation-vs-schwarzianeq} \emph{defines} the projective structure $P=P_0+\phi$ in the sense that we can recover a projective atlas outside of the polar locus just as in \autoref{indigenousbundle}.

Consider a triple $(\pi:Q\rightarrow C,\mathcal{F},\sigma)$, where $\pi$ is a holomorphic $\mathbb{P}^1$-bundle on $C$, $\mathcal{F}$ is a singular Riccati foliation on $\pi$ and $\sigma$ is a holomorphic section of $\pi$ generically transverse to $\mathcal{F}$. We say that the triple \emph{defines} a meromorphic projective structure $P$ on $C$ if we can recover a projective atlas outside of the polar locus just as in \autoref{indigenousbundle}.

However, in contrast with the latter situation, triples not holomorphically equivalents may define the same meromorphic projective structure on $C$. The object we are dealing with here is really \emph{the triple $(\pi:Q\longrightarrow C,\mathcal{F},\sigma)$ modulo gauge transformations of the bundle that are bimeromorphic and such that their indeterminacy locus are contained in the fibers above poles of the projective structure}.

\begin{prop}\label{foliation-vs-schwarzianeq-mero}
	\autoref{foliation-vs-schwarzianeq} holds for meromophic projective structures, if we allow bimeromorphic gauge transformations and replace $q\in\mathcal{O}(U)$ by $q\in\mathcal{M}(U)$.
\end{prop}

Note that, when the triple is locally expressed in a form (\ref{normal-form}), but with $q$ possibly having a pole at $x=0$, we see by the change of variable $Y=1/y$ that $q$ admits a pole of order at least $1$ if and only if $\sigma:=\{y=\infty\}$ intersects the singular set of $\mathcal{F}$. (This is the unique singular point on the invariant fiber). Hence the section $\sigma$ is merely \emph{generically} transverse to the foliation $\mathcal{F}$: it is transverse above points in $C^*$, and it goes through singularities of $\mathcal{F}$ above poles of $P$. Thus, we are not dealing with an oper. We can reduce to the situation where the section is \emph{everywhere} transverse to the foliation by applying bimeromorphic gauge transformations of the bundle, though. It is the object of \autoref{prop-transversality-of-sigma} in the next section.

\subsection{A unique meromorphic oper associated with a meromorphic projective structure without apparent singularity: the minimal birational model}

\begin{defi}
	Let $\pi:Q\rightarrow C$ be a $\mathbb{P}^1$-bundle and $q\in Q$. The blow up of $q$, composed with the contraction (which is possible thanks to Castelnuovo's criterion) of the proper (or strict) transform of the fiber $\pi^{-1}(\pi(q))$ is called the \emph{elementary transformation centered at $q$} and denoted by $\operatorname{elem}_q$.
\end{defi}

An elementary transformation is a birational transformation, and its inverse is an elementary transformation. The image of a $\mathbb{P}^1$-bundle by an elementary transformation is again a holomorphic $\mathbb{P}^1$-bundle. Every $\mathbb{P}^1$-bundle over $C$ can be obtained from the trivial bundle $C\times\mathbb{P}^1\rightarrow C$ by some successive elementary transformations (see~\cite[p. 229]{MR1392959}). All birational transformations between $\mathbb{P}^1$-bundles over $C$ are obtained by composing elementary transformations.

For instance, consider the trivial bundle $\mathbb{D}\times\mathbb{P}^1\rightarrow\mathbb{D}$, with coordinate $x$ on the disc, homogeneous coordinates $[y_1:y_2]$ on the projective line. Then, $\operatorname{elem}_{(0,[y_1:0])}(x,[y_1,y_2])=(x,[xy_1,y_2])$.

Let $\sigma$ be a section of the $\mathbb{P}^1$-bundle $\pi$. After the elementary transformation $\operatorname{elem}_q$, the self-intersection number of the strict transform $\tilde{\sigma}$ of $\sigma$ satisfies
\[\tilde{\sigma}\cdot\tilde{\sigma}=\begin{cases}
\sigma\cdot\sigma + 1\text{ if }q\notin\sigma\\
\sigma\cdot\sigma - 1\text{ if }q\in\sigma.
\end{cases}\]

\begin{prop}\label{prop-transversality-of-sigma}
	Let $P$ be a meromorphic projective structure on $C$, and let $(\pi:Q\rightarrow C,\mathcal{F},\sigma)$ be a triple defining it. There exists a $\operatorname{PGL}(2,\mathbb{C})$-oper $(\tilde{\pi}:\tilde{Q}\rightarrow C,\tilde{\mathcal{F}},\tilde{\sigma})$ defining $P$ where $\tilde{\pi}$ is bimeromorphically equivalent to $\pi$ via elementary gauge transformations sending $\mathcal{F}$ to $\tilde{\mathcal{F}}$ and $\sigma$ to $\tilde{\sigma}$ (the point here is that $\tilde{\sigma}$ is \emph{everywhere} transverse to $\tilde{\mathcal{F}}$).
\end{prop}

\begin{proof}
	Suppose $P$ admits at least one pole, otherwise there is nothing to prove. Start by applying \autoref{foliation-vs-schwarzianeq-mero} for a local coordinate $x$ centered at a pole $p$ of $P$ of order $n\in\mathbb{N}^*$. Sticking to the same notations, we get \[\mathcal{F}_{|U}:dy+\left(y^2+\frac{\tilde{q}(x)}{2x^n}\right)dx=0\textmd{ with $\tilde{q}:=x^nq\in\mathcal{O}(U)$ and $\tilde{q}(0)\neq0$.}\]
	As we have already noticed, the only point where $\sigma$ goes through a singularity of the foliation $\mathcal{F}_{|U}$ is $(x,y)=(0,\infty)$. The idea is to apply an appropriate number of elementary transformations centered at this point.
	
	Apply first the gauge transformation $Y=1/y$, then multiply by $-Y^2$ before performing $l\in\mathbb{N}$ elementary gauge transformations $Y=x^l\tilde{Y}$. After a division by $x^n$, we get the Pfaffian equation \[\tilde{\mathcal{F}}_{|U}:x^ld\tilde{Y}+\left[-x^{2l-n}\frac{\tilde{q}(x)}{2}\tilde{Y}^2+lx^{l-1}\tilde{Y}-1\right]dx=0,\] so that fixing $l=(n+1)/2$ or $n/2$ depending if $n$ is odd or even, respectively, we get $-dx=0$ at $(x,\tilde{Y})=(0,0)$. This $1$-form is both non-zero and transverse to the pull-back $\tilde{\sigma}:=\{\tilde{Y}=0\}$ of the section $\sigma$.
	
	Applying this kind of transformations for every poles of $P$ leads to the desired oper $(\tilde{\pi}:\tilde{Q}\rightarrow C,\tilde{\mathcal{F}},\tilde{\sigma})$.	The projective structure associated to this new triple is the same as that of the initial one, because elementary transformations are biholomorphic outside of the fibers containing the points at which they are centered.
\end{proof}

\begin{figure}[H]
	\centering
	\def\svgwidth{1\textwidth}
	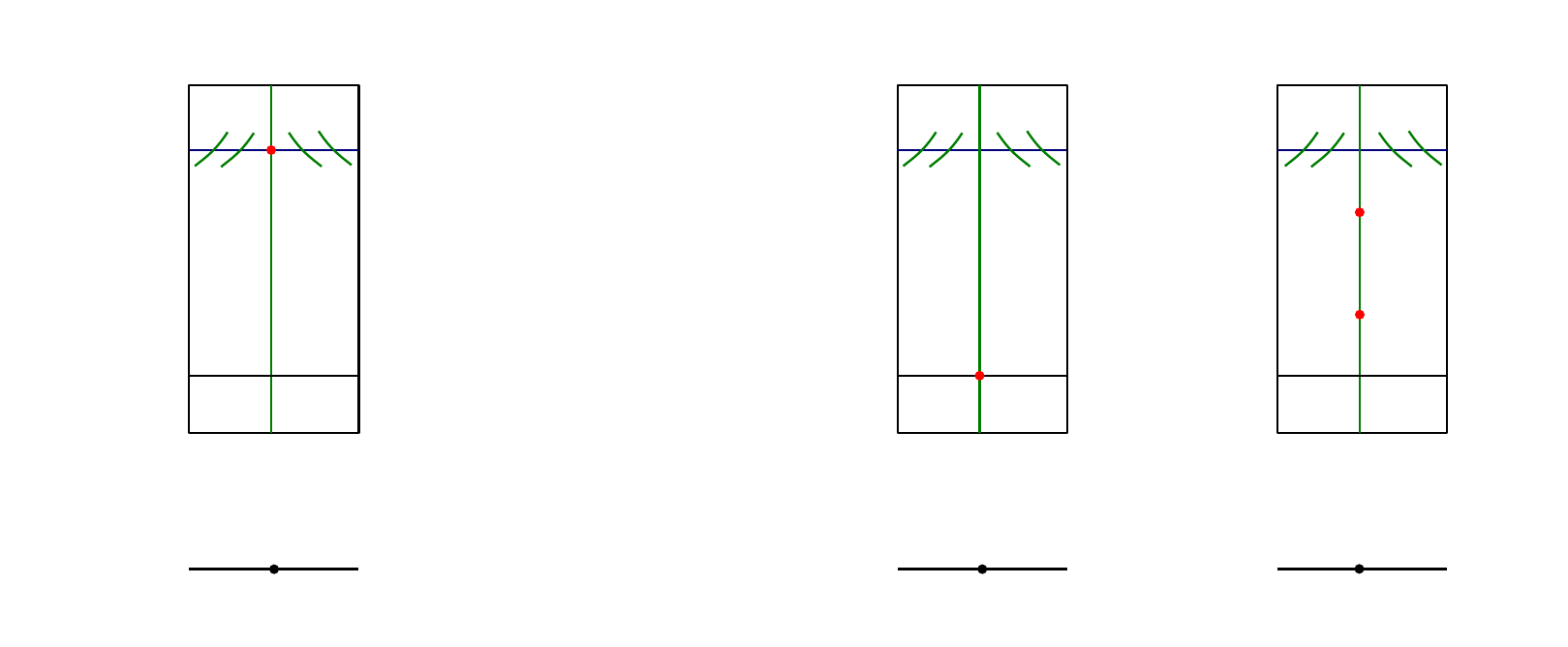
	\caption{\autoref{prop-transversality-of-sigma}.}
\end{figure}

An oper defining a meromorphic projective structure (with at least one pole) is not unique modulo bimeromorphic gauge transformation. If we add a condition on the order of tangency between the foliation $\mathcal{F}$ and the fibers of $\pi$, though, we get a uniqueness statement.

Just as the transversality between the section and the foliation, the order of tangency between the foliation and the fibers can be changed by a bimeromorphic gauge transformation.

\begin{rmk}\label{rk-elem-and-pole-order}
	Consider a singular Riccati foliation $(\pi:Q\rightarrow C,\mathcal{F})$, and assume it has an invariant fiber. Chose a coordinate system $(x,y)$ on a local trivialization of the bundle such that the set $\{x=0\}$ corresponds to an invariant fiber. By $3$-transitivity of the action of Möbius transformations on the projective line, we can assume without loss of generality that $(0,0)$ is a singular point of the foliation and that $(0,\infty)$ is not. Then, there exists an integer $m\in\mathbb{N}^*$ and holomorphic functions $a,b$ and $c$ such that the Pfaffian equation defining locally the foliation reads \[x^mdy+(a(x)y^2+b(x)y+c(x))dx=0,\] with $a(0)\neq0$ and $c(0)=0$. By a holomorphic gauge transformation, we can normalize $a$ to $1$ while preserving $\{y=0\}$ and $\{y=\infty\}$ (cf. equation \eqref{action-gauge-multiplication}). This changes $b$ and $c$ to, say, $\tilde{b}$ and $\tilde{c}$, respectively. After the meromorphic transformation $(x,y):=(x,x^k\tilde{y})$, $k\in\mathbb{Z}$, we get \[x^{m+k}d\tilde{y}+[x^{2k}\tilde{y}^2+(kx^{m+k-1}+\tilde{b}(x)x^k)y+\tilde{c}(x)]dx=0,\] with $\tilde{b}, \tilde{c}$ holomorphic ($\tilde{c}(0)=0$). If $k>0$, this transformation (which also preserves $\{y=0\}$ and $\{y=\infty\}$ away from $\{x=0\}$) corresponds to an elementary transformations centered at $(0,\infty)$, and we see from the latter equation that the order of vanishing of the coefficient before $d\tilde{y}$ at $0$ has either decreased or remained equal to $m$ after $k=1$ elementary transformation (depending on the order of vanishing of $\tilde{b}$ and $\tilde{c}$, and after division by the appropriate power of $x$). If $k<0$, we are applying elementary transformations centered at $(0,0)$. If $k=-1$ we see (after a multiplication by $x^{2}$) that the pole order of the Riccati equation has increased from $m$ to $m+1$.

	This proves that \emph{an elementary transformation centered \emph{at a singular point} of the foliation decreases or leaves unchanged the order of the pole of the Riccati equation, whereas an elementary transformation centered \emph{outside} of the singular locus of the foliation on an invariant fiber always increase the order of the pole of the Riccati equation.}
	
	A birational model of $(\pi:Q\rightarrow C,\mathcal{F})$ minimizing the polar divisor of the foliation is not unique (see \cite{MR3328860}).
\end{rmk}

\begin{prop}\label{pole1form-vs-poleschwarzian}
	Let $(\pi:Q\rightarrow C,\mathcal{F},\sigma)$ be a triple defining a meromorphic projective structure on $C$.%
	
	\begin{itemize}
		\item[(i)] Suppose $\mathcal{F}$ reads locally around a singularity centered at $x=0$ as in  \autoref{foliation-vs-schwarzianeq-mero}, with $q$ having a pole of order $n\in\mathbb{N}^*$. Then, applying elementary gauge transformations centered at a point in $\{x=0\}$ we can get a $1$-form defining the foliation having poles of order $\lceil \frac{n}{2} \rceil$ (where $\lceil \cdot \rceil$ denotes the ceiling function). This is the minimum order that can be achieved by applying meromorphic gauge transformations.
		
		\item[(ii)] Suppose $\mathcal{F}$ is defined locally in a coordinate $x$ by a $1$-form admitting a pole of order $m$ at $x=0$, then $q$ (\autoref{foliation-vs-schwarzianeq-mero}) has a pole of order at most $2m$.
	\end{itemize}
\end{prop}

\begin{proof}
	\emph{(i)} Recall that any birational gauge transformation is the composition of elementary transformations. Consider the normal form \eqref{normal-form}. The only singularity of the foliation $\mathcal{F}$ being $(x,y)=(0,\infty)$, the idea (cf. \autoref{rk-elem-and-pole-order}) is to apply an elementary transformation centered at this point, and repeat the process as long as it is possible, say $k\in\mathbb{N}$ times. Pulling-back the $1$-form via $(x,\tilde{y})\mapsto(x,y):=(x,x^{-k}\tilde{y})$ leads to \[2x^{n-k}d\tilde{y}+\left[2x^{n-2k}\tilde{y}^2-kx^{n-1-k}\tilde{y}+\tilde{q}(x)\right]dx=0,\] with $\tilde{q}(x):=x^nq(x)\in\mathcal{O}(U)$ and $\tilde{q}(0)\neq0$. Whence the result.
	
	\emph{(ii)} This can be seen by carrying out the computations in the proof of \autoref{foliation-vs-schwarzianeq}(i) but with $a$ and $b$ meromorphic and possibly admitting a pole at $x=0$.
\end{proof}

In fact, as observed by Loray and Pereira in \cite[Sec. 3.2]{MR2337401}:

\begin{prop}\label{unique-birational-model}
	Let $(\pi:Q\rightarrow C,\mathcal{F},\sigma)$ be a singular Riccati foliation $\mathcal{F}$, together with a section $\sigma$ generically transverse to $\mathcal{F}$. There is a unique birational model minimizing the polar divisor such that the section does not intersect the singular set of the foliation.%
\end{prop}

If we further assume that the corresponding projective structure has no apparent singularity, then by \autoref{apparent-sing-and-tangency} (bellow) the section of the above minimal model cannot be tangent to the foliation, hence it defines an oper. In conclusion, we get a correspondence similar to \autoref{indigenousbundle} for meromorphic projective structures without apparent singularity on the smooth curve $C$.

\begin{prop}\label{indigenousbundle-meromorphic}
	The above correspondence defines a bijection between the set of meromorphic projective structures without apparent singularity on $C$ and the set of meromorphic $\operatorname{PGL}(2,\mathbb{C})$-oper $(\pi:Q\rightarrow C,\mathcal{F},\sigma)$ such that the polar divisor of $\mathcal{F}$ is minimal.
\end{prop}

\paragraph{Apparent singularities, branch points and tangency between the section and the foliation.} \autoref{pole1form-vs-poleschwarzian} allows us to characterize apparent singularities and give there relation to branch points (see also \cite[Prop. IX.1.2]{MR3494804}).

\begin{prop}\label{apparent-sing-and-tangency}
	Let $(\pi:Q\rightarrow C,\mathcal{F},\sigma)$ be a $\operatorname{PGL}(2,\mathbb{C})$-oper. Suppose the associated projective structure admits an apparent singularity at a point $p\in C$. Then, there exist elementary gauge transformations desingularizing the foliation over $p$, while possibly introducing a tangency point between $\sigma$ and $\mathcal{F}$. The converse is true.
\end{prop}

\begin{proof}
	Suppose first that the section $\sigma$ is tangent to $\mathcal{F}$ at some point in the fiber over a point $p\in C$ along which the foliation is regular. Then, there are holomorphic functions $a,b$ and an integer $k\in\mathbb{N}^*$ such that, in a bundle chart $U_x\times\mathbb{C}_y$ such that $p=0$ and $\sigma=\{y=0\}$,
	\[\mathcal{F}_{|U}:dy+[a(x)y^2+b(x)y+x^{k-1}]dx=0.\]
	Here, $\operatorname{tang}(\mathcal{F},\sigma)=k-1$. The composition
	\[(x,y)\longmapsto(x,\tilde{y}):=(x,y/x^k)\]
	of $k$ elementary gauge transformations (the first one being centered at $(x,y)=(0,0)$) leads to
	\[d\tilde{y}+\left[a(x)x^k\tilde{y}^2+\left(\frac{k}{x}+b(x)\right)\tilde{y}+\frac{1}{x}\right]dx=0.\]
	We see, after multiplication by $x$, that the section is now transverse to the foliation around $x=0$. Moreover, a calculation similar to the proof of \autoref{foliation-vs-schwarzianeq} shows that the projective structure has a pole of order at most $2$ at $x=0$. But the local monodromy is trivial since we started with a non-singular foliation.
	
	Conversely: Thanks to the \autoref{pole1form-vs-poleschwarzian}, the foliation may be described by a $1$-form with a pole of order at most $1$. This falls into the non-degenerate case of \cite[Prop. 4.2]{MR3328860}. Furthermore, the monodromy being trivial, we can perform (see \cite[page 44 and 45]{MR3328860}) a change of variable (not necessarily a gauge transformation preserving $\sigma$) to get, for some $n\in\mathbb{Z}$
	\[dy-\frac{ny}{x}dx=0.\]
	Then, a composition of elementary transformations (the first one being centered at $(x,y)=(0,0)$) leads to \[dy=0,\]
	which is non-singular. Finally, the corresponding section may be tangent to the foliation. We can't ensure the transversality.
	
	Other types of singularities (in \cite{MR3328860}) cannot be desingularized. Indeed, there monodromy are non-trivial. Hence, apparent singularities are exactly those that can be desingularized.
\end{proof}

Around those points, the projective charts take the form $\varphi(x)=x^{k-1}$ up to a change of coordinate: they are branch points.

\subsection{Construction of a family of meromorphic \texorpdfstring{$\operatorname{PGL}(2,\mathbb{C})$}{PGL(2,C)}-opers}

We use de notations of section~\ref{section-meromorphic-projective-structures}.

\begin{lem}\label{family-of-opers}
	For all point $\tau_0\in\mathcal{P}^\circ(\mathbb{S},\mathbb{M},(n_i))$, there exists an open neighborhood $U$ of $\tau_0$ and a holomorphic family of curves $F:\mathcal{C}\rightarrow U$, together with
	\begin{itemize}
		\item a holomorphic $\mathbb{P}^1$-bundle $\Pi:\mathcal{Q}\rightarrow \mathcal{C}$,
		\item a dimension $1$ singular foliation $\mathcal{F}$ on $\mathcal{Q}$,
		\item and a holomorphic section $\mathfrak{S}$ of $\Pi$
	\end{itemize}
	such that $F\circ\Pi$ is a first integral of $\mathcal{F}$, and in restriction to any fiber $C_\tau:=F^{-1}(\tau)$ of the family~$F$, the triple $(\Pi_{|C_\tau},\mathcal{F}_{|Q_\tau},\mathfrak{S}_{|C_\tau})$ is the unique $\operatorname{PGL}(2,\mathbb{C})$-oper with minimal polar divisor on $C_\tau$ associated to meromorphic projective structures corresponding to $\tau$.
\end{lem}

\begin{proof}
	Recall from section~\ref{section-meromorphic-projective-structures} that $B$ denotes the Teichmüller space $\mathcal{T}(S,d)$ and $f:X\rightarrow B$ denotes the universal family of marked curves. Pick a point $\tau_0\in\mathcal{P}^\circ(\mathbb{S},\mathbb{M},(n_i))$ and denote by $b_0$ its image in $B$. Replacing $B$ by an open ball, we know from \autoref{existence-relative-proj-struct} that there exists a relative projective structure (without pole) $\alpha$ on $f:X\rightarrow B$. It induces a holomorphic section $s_0$ of $\mathcal{P}(X/B;\mathcal{D})\rightarrow B$. This section in turn induces a bijection
	\begin{align*}
	\Phi_{s_0}:\mathcal{P}(X/B;\mathcal{D})&\longrightarrow\mathcal{Q}(X/B;\mathcal{D})\\
	[(t,P),\phi]&\longmapsto (t,(P+\phi)-(P_{s_0(t)}+\phi_{s_0(t)}))\text{ if }s_0(t)=[(t,P_{s_0(t)}),\phi_{s_0(t)}].
	\end{align*}
	
	Then, define $U$ as the preimage of $B$ in $\mathcal{P}^\circ(\mathbb{S},\mathbb{M},(n_i))$ and form the fibred product \[\mathcal{C}:=U\times_B X\] and define $F:\mathcal{C}\rightarrow U$ as the pullback of $f$.
	
	\[\xymatrix{
		\mathcal{C}=U\times_B X \ar[d]^{F} \ar[rrr] & & & X\ar[d]^{f}\\
		U \ar[r]^{M_{|U}} &\mathcal{P}(X/B;\mathcal{D},(n_i)) \ar@{^{(}->}[r] &\mathcal{P}(X/B;\mathcal{D}) \ar[dr]^{\Phi_{s_0}} \ar[r] & B \ar@/_1pc/_-{s_0}[l]\\
		& & & \mathcal{Q}(X/B;\mathcal{D}) \ar[u]
	}\] 
	
	For all $\tau\in U$, let us denote $t=\operatorname{pr}_1\circ \Phi_{s_0}\circ M(\tau)$ and $\phi_\tau:=\operatorname{pr}_2\circ \Phi_s\circ M(\tau)$ ($\operatorname{pr}_i$ means the projection to the $i$-th factor). The projective structure $P_t:=(P_{s_0(t)}+\phi_{s_0(t)})+\phi_\tau$ on $C_t$ induced by the above relative projective structure is also a projective structure $P_\tau$ on $C_\tau:=F^{-1}(\tau)$ (by construction of the fibered product). This defines a holomorphic family of meromorphic projective structure on $U$ (this is the tautological family: roughly speaking, it is $\tau$ on the fiber over $\tau$).

	Denote by $\tilde{D}$ and $\tilde{\alpha}$ the pullbacks of $\mathcal{D}$ and $\alpha$ on $\mathcal{C}$, respectively. According to \cite[Lem. 7.4]{Allegretti_2020}, the holomorphic family defined in the preceding paragraph induces a relative projective structure without poles on the family of curves $\mathcal{C}\smallsetminus\tilde{\mathcal{D}}\rightarrow U$ and subtracting (the restriction of) $\tilde{\alpha}$ to it gives the holomorphic family of relative quadratic differentials (see \cite[Prop. 7.3]{Allegretti_2020}) corresponding to $\phi_\tau$.

	Now consider the trivial bundle $\Pi_0:\mathcal{C}\times\mathbb{P}^1\rightarrow\mathcal{C}$. Chose moreover a standard coordinate $y$ on $\mathbb{P}^1$ and consider the section $\mathfrak{S}_\infty:=\{y=\infty\}$ of this trivial bundle. From the pair $(P_\tau,\phi_\tau)$, we know (see \autoref{construction-oper}) how to construct a triple $(\Pi_{0|C_\tau}:C_\tau\times\mathbb{P}^1\rightarrow C_\tau,\mathcal{F}_\tau,\mathfrak{S}_{\infty|C_\tau})$ describing the projective structure associated to $\tau$. 
	
	The polar divisor on $C_\tau$ of the Riccati foliation $\mathcal{F}_\tau$ is not necessarily minimal, and $\mathfrak{S}_{\infty|C_\tau}$ is not necessarily everywhere transverse to $\mathcal{F}_\tau$. To get an oper with minimal polar divisor, we apply elementary transformations along the locus of the divisor $\mathcal{D}$ (this is a family version of Propositions \ref{prop-transversality-of-sigma}, \ref{pole1form-vs-poleschwarzian} and \ref{unique-birational-model}). Each component $\mathcal{D}_i$ is an hypersurface in $\mathcal{C}$. An elementary transformation centered at $\mathfrak{S}_\infty(\mathcal{D}_i)$ can be described as follows: first blow-up $\mathfrak{S}_\infty(\mathcal{D}_i)$ and then contract the strict transform of $\Pi_0^{-1}(\mathcal{D}_i)$. We apply elementary transformations until we get an oper above $C_\tau$ for all $\tau\in U$.

	By construction, the resulting objects satisfy the requirements of the statement.
\end{proof}

This lemma remains true if we replace $\mathcal{P}^\circ(\mathbb{S},\mathbb{M},(n_i))$ by one of its submanifolds.

\section{A smooth family of singular Riccati foliations}

Let $C$ be a smooth curve and $E\rightarrow C$ be a rank $k$ holomorphic vector bundle, that we will frequently denote simply by $E$. Its sheaf of holomorphic sections is a locally free sheaf of $\mathcal{O}_C$-modules of rank $k$, that we denote by $E$ as well. The sheaf of $\mathcal{O}_C$-modules associated with a divisor $D$ on $C$ is denoted by $\mathcal{O}_C(D)$, and it is defined by setting
\[U\longmapsto\mathcal{O}_C(D)(U):=\{\text{$f:U\rightarrow\mathbb{C}$ $|$ $f$ is meromorphic and $\operatorname{div}(f)\geq -D$}\}.\]
We denote $\Omega^1_C(D):=\Omega^1_C\otimes_{\mathcal{O}_C}\mathcal{O}_C(D)$ the sheaf of meromorphic $1$-forms with polar divisor "at most" $D$.

\begin{defi}
	A \emph{meromorphic connection $\nabla$ on the holomorphic vector bundle $E$} is a $\mathbb{C}$-linear morphism of sheaf
	\[\nabla:E\longrightarrow E\otimes_{\mathcal{O}_C}\Omega^1_C(D)\]
	for some effective divisor $D$, satisfying the \emph{Leibniz rule}, i.e. for any open set $U\subset C$, any function $f\in\mathcal{O}_C(U)$ and any local section $s$ of $E$ over $U$, $\nabla(fs)=f\nabla s+s\otimes df$.
\end{defi}

Note that because $C$ has dimension $1$, the connection is automatically flat.

Often, in the context of moduli spaces, it is the couple $(E,\nabla)$ that we call a connection. We say that the connection $(E,\nabla)$ has rank $k$, referring to the rank of the vector bundle $E$.

Consider for instance the rank $2$ case. Concretely, over an open set $U\subset C$ where the bundle is trivial and for a given local frame $(e_1,e_2)$, such a connection is represented by a matrix $1$-form $\Omega=(\omega_{ij})_{ij}\in\operatorname{Mat}(2,\Omega^1_C(D)(U))$ called the \emph{connection matrix}, whose coefficients are defined by the relations $\nabla e_j=\sum_{i=1}^2e_i\otimes \omega_{ij}$. Conversely, given a matrix $1$-form $\Omega$ on $U$, we can define a connection $\nabla$ on the trivial bundle $U\times\mathbb{C}^2$ as follows. The image of a section $s=\sum_{j=1}^2 y_je_j$ is of course defined by \[\nabla s:=\sum_{i=1}^2 e_i\otimes\left(dY+\Omega Y\right)_i,\] where $Y:=\begin{pmatrix}
y_1\\
y_2
\end{pmatrix}$. Thus, we may abusively write
\begin{equation*}
\nabla=d+\Omega
\end{equation*}
if we remember that this expression depends on the local frame that we have chosen. Whenever we pullback by a holomorphic \emph{gauge transformation} $\tilde{Y}\mapsto G\tilde{Y}$, the connection matrix becomes \begin{equation*}
\tilde{\Omega}=G^{-1}\Omega G+G^{-1}dG.
\end{equation*}
A collection of matrix $1$-forms -- one for each local trivialization of $E$ -- satisfying this gluing property represents a unique connection on $E$. Note that since $G^{-1}dG$ is holomorphic, $G$ conjugates the principal part of $\Omega$ to the principal part of $\tilde{\Omega}$. The poles and pole orders of the connection are defined to be those of its matrices. We can chose $D$ to be the effective \emph{polar divisor} of $\nabla$, though strictly speaking our definition does not require this.

The \emph{Poincaré rank} of $\nabla$ at a pole of order $m$ is defined to be $m-1$. It might change after a meromorphic gauge transformation. The \emph{minimal} Poincaré rank in the equivalence class of the system up to meromorphic gauge transformation is sometimes called the \emph{true} Poincaré rank.

\begin{defi}\label{minimal-polar-divisor-linear-connection}
	The polar divisor of a meromorphic connection $(E,\nabla)$ on $C$ is said to be \emph{minimal} if the pole order at each pole cannot be decreased by a meromorphic gauge transformation.
\end{defi}

\begin{defi}
	We say that a connection $(E,\nabla)$ is \emph{reducible} if there exists an \emph{invariant subbundle}, i.e. a subbundle $0\subsetneq L\subsetneq E$ satisfying $\nabla(L)\subset L\otimes\Omega^1_C(D)$. Otherwise, we say it is \emph{irreducible}.
\end{defi}

\paragraph{Fuchs relation.}	The \emph{trace connection} $(\operatorname{det}(E),\operatorname{tr}(\nabla))$ is the connection induced by $(E,\nabla)$ on the determinant bundle of $E$. In local trivialization coordinates, its connection matrices are the traces of the matrices of $\nabla$. We say a connection is \emph{trace-free} if its trace is isomorphic to the horizontal connection $d$ on the trivial line bundle~$\mathcal{O}_C$.

The following relation holds for arbitrary rank and is a consequence of the residue formula (for a proof, see \cite[Cor. 17.35]{MR2363178}).

\begin{prop}[\emph{Fuchs relation}]
	If $\nabla$ is a meromorphic connection on a smooth vector bundle $E$ over a smooth compact curve $C$, then
	\begin{equation}\label{Fuchs-relation}
	\sum_{p\in C} \operatorname{res}_{p}(\operatorname{tr}(\nabla))+\operatorname{deg}(E)=0,
	\end{equation}
	where the degree of $E$ is by definition the degree of its determinant bundle.
\end{prop}

\subsection{Riccati foliations and meromorphic rank \texorpdfstring{$2$}{2} connections}

Let $(E,\nabla)$ be a rank $2$ meromorphic connection on $C$. Consider a non-zero local section $s=\sum_{j=1}^2 y_je_j$ of $E$ that is flat for $\nabla$, i.e. $\nabla s=0$. As usual, $(e_1,e_2)$ is a local frame of $E$ and $\Omega=(\omega_{ij})_{ij}$ denotes the connection matrix of $\nabla$ with respect to this frame. Assume that $y_2\neq 0$ and denote $y=y_1/y_2$. Computing the logarithmic derivative of $y$ leads to the Riccati equation \[dy+\omega_{2,1} y^2+(\omega_{2,2}-\omega_{1,1})y-\omega_{1,2}=0.\]
This equation defines the projectivization $\mathbb{P}(E,\nabla):=(\mathbb{P}(E),\mathbb{P}(\nabla))$ of the connection $(E,\nabla)$, i.e. a Riccati foliation on the $\mathbb{P}^1$-bundle $\mathbb{P}(E)$.

Thus, to a rank $2$ connection we can associate its trace and its projectivization,
\begin{equation*}
(E,\nabla)\longmapsto\begin{cases*}(\det(E),\operatorname{tr}(\nabla))\\(\mathbb{P}(E),\mathbb{P}(\nabla)).\end{cases*}
\end{equation*}
Now, can we always find a rank $2$ connection with prescribed trace and projectivization? In other words, is it possible to \emph{lift} a given Riccati foliation $(Q,\mathcal{F})$ to a rank $2$ meromorphic connection $(E,\nabla)$ with a given trace $(L,\nabla_L)$? If it exists, such a lift is not unique. Indeed, $\mathbb{P}(E,\nabla)=\mathbb{P}(\tilde{E},\tilde{\nabla})$ if and only if there exists a rank $1$ meromorphic connection $(L',\nabla_{L'})$ such that $(\tilde{E},\tilde{\nabla})=(E\otimes L',\nabla\otimes\nabla_{L'})$. This can be seen by comparing their cocycles and connection matrices: if the cocycles of $E$ and $L'$ with respect to a common open cover $(U_i)$ of $C$ are respectively $(G_{ij})$ and $(g_{ij})$, then $E\otimes L'$ have cocycles $G_{ij}\cdot g_{ij}$. If $\nabla=d+\Omega_i$ and $\nabla_{L'}=d+\omega_i$ with respect to some trivialization charts, then $\nabla\otimes\nabla_{L'}=d+(\Omega_i+\omega_iI)$. We call this operation a \emph{twist} of $(E,\nabla)$ by the rank $1$ meromorphic connection $(L',\nabla_{L'})$. For any rank $1$ meromorphic connection $(L',\nabla_{L'})$,
\[\begin{cases*}
\operatorname{det}(E\otimes L')=\operatorname{det}(E)\otimes L'^{\otimes2}\\
\operatorname{tr}(\nabla\otimes\nabla_{L'})=\operatorname{tr}(\nabla)\otimes\operatorname{tr}(\nabla_{L'})^{\otimes 2}.
\end{cases*}\]
Consequently, such a lift $(E,\nabla)$ is unique up to a twist by a tensor square root of the rank~$1$ trivial connection $(\mathcal{O}_C,d)$. Note that this square root is necessarily a \emph{holomorphic} connection. If $g$ is the genus of $C$, there are $2^{2g}$ such elements (indeed, they correspond via the Riemann-Hilbert correspondence --which preserves the group structures-- to the $2^{2g}$ elements of $\operatorname{Hom}(\pi_1(C),\{\pm1\})$), whence at most $2^{2g}$ lifts.

Let us return to the question of the existence of such a lift. Every $\mathbb{P}^1$-bundle on a curve is the projectivization of a rank $2$ vector bundle (see \cite[Intro.]{MR0276243}), and this vector bundle is unique up to a tensor product with a line bundle $L'$. Let $E$ be a vector bundle such that $\mathbb{P}(E)=Q$. We are looking for a line bundle $L'$ such that $\det(E\otimes L')=\det(E)\otimes\det(L')^{\otimes2}=L$, that is, a square root of $L\otimes\det(E)^{\otimes(-1)}$ in the Picard group $\operatorname{Pic}(C)$. Such a square root exists if and only if the degree of $L\otimes\det(E)^{\otimes(-1)}$ is even. As for the connection $\nabla_L$ and the Riccati foliation $\mathcal{F}$, suppose they are given locally over each open set of an open cover $(U_i)$ as $d+\delta_i$ and $dy+y^2\alpha_i+y\beta_i+\gamma_i=0$, respectively. Then, the connection on $E\otimes L'$ defined locally by
\begin{equation*}
dY+\begin{pmatrix}
\frac{\delta_i-\beta_i}{2} & -\gamma_i\\
\alpha_i & \frac{\delta_i+\beta_i}{2}
\end{pmatrix}Y=0
\end{equation*}
has trace $\nabla_L$ and projectivization $\mathcal{F}$. In conclusion, we get the following proposition.

\begin{prop}\label{prop:lift}
	Let $(L,\nabla_L)$ and $(Q,\mathcal{F})$ be a rank $1$ meromorphic connection and a singular Riccati foliation on $C$, respectively. Let $E$ be a rank $2$ vector bundle on $C$ such that $\mathbb{P}(E)=Q$. There exists a rank $2$ meromorphic connection on $C$ with trace $(L,\nabla_L)$ and projectivization $(Q,\mathcal{F})$ if and only if \[\deg(L)-\deg(E)\equiv 0~ \operatorname{mod} ~2.\] This lift is unique up to a twist by a tensor square root of the rank $1$ trivial connection $(\mathcal{O}_C,d)$.
\end{prop}

\begin{rmk}
	Since we are interested in projective structures, there is a priori no reason for us to speak about linear connections. However, some results we will use are formulated in the literature in the setting of linear connection.
\end{rmk}

\subsection{Projective structures and opers for the group \texorpdfstring{$\operatorname{GL}(2,\mathbb{C})$}{GL(2,C)}}\label{sec:proj-struct-GL-opers}

\paragraph{Some properties of $\mathbb{P}^1$-bundles.} Let us first recall some useful facts regarding $\mathbb{P}^1$-bundles and their sections; proofs can be found in Maruyama's book \cite[Chap. I]{MR0276243}.

\begin{facts}\label{P1-bundle-facts}
	Let $E$ be a holomorphic rank $2$ vector bundle on a genus $g$ curve $C$, and denote $Q=\mathbb{P}(E)$ its projectivization. Every $\mathbb{P}^1$-bundle on $C$ arises in this way. To give a (holomorphic) section of $Q$ is equivalent to give a line subbundle of $E$, and there always exists one. Degrees of line subbundles of $E$ are bounded above; a line subbundle of $E$ is said to be \emph{maximal} if its degree is maximal. We denote the maximal degree by $M(E)$.
	
	Let $\sigma$ be a section of $Q$ and $L_\sigma$ be the corresponding line subbundle of $E$. The self-intersection number of $\sigma$ is equal to the degree of its normal bundle $N_\sigma$, which satisfies $\sigma^*(N_\sigma)=\operatorname{det}(E)\otimes L_\sigma^{\otimes(-2)}$. Thus, \[\sigma\cdot\sigma=\deg(E)-2\deg(L_\sigma).\] Maximal line subbundles of $E$ are in bijective correspondence with \emph{minimal} sections of $Q$, i.e. with section with minimal self-intersection number $N(E):=\deg(E)-2M(E)$. In fact, since this number depends only on $Q$, we also denote it by $N(Q)$. It is at most $g$. The parity of $\sigma\cdot\sigma$, for any section $\sigma$, is the same as the parity of $N(Q)$.
	
	In fact $N(Q)\equiv\operatorname{deg}(E)$ mod $2$ and this parity also determines the topological class of the $\mathbb{S}^2$-bundle subjacent to $\mathbb{P}(E)$ (see \cite{MR983870}, \cite{https://doi.org/10.48550/arxiv.1410.4976}, \cite[Sec. 2.1]{MR2647972}).

	If $\sigma$ is minimal and $\sigma\cdot\sigma=N(Q)<0$, then any other section $\sigma'$ has a self-intersection number $\sigma'\cdot\sigma'\geq-N(Q)$ (there is a gap in the possible values).
	
	Finally, we have that
	\begin{enumerate}
		\item if $N(Q)<0$, then there is a unique maximal line subbundle of $E$,
		\item if $N(Q)=0$ and $E$ is indecomposable, then again there is a unique maximal line subbundle of $E$,
		\item if $N(Q)=0$, $E$ is decomposable but $E\not\simeq L\oplus L$ for any line subbundle $L$ of $E$, then there are only two maximal line subbundles $L_1,L_2$ of $E$, and in fact $E=L_1\oplus L_2$,
		\item if $E=L\oplus L$ for some line subbundle $L$ of $E$, then $E$ has infinitely many maximal line subbundles, all of which are isomorphic to $L$.
	\end{enumerate}
	In fact, if $N(Q)=0$, $\deg(E)=0$ and $E$ is decomposable, we have $E=L\oplus L^{-1}$ for some line subbundle $L$ and the two last cases correspond respectively to $L^{\otimes 2}\not\simeq \mathcal{O}_C$ and $L^{\otimes 2}\simeq\mathcal{O}_C$.
\end{facts}

\begin{rmk}\label{rmk:invariance-section-subbundle}
	A line subbundle $L$ of $E$ is invariant by $\nabla$ if and only if the corresponding section of $\mathbb{P}(E)$ is invariant by the foliation induced by $\nabla$.
\end{rmk}

\paragraph{Opers for the group $\operatorname{GL}(2,\mathbb{C})$.}

\begin{defi}\label{defi:GL-oper}
	A \emph{(meromorphic) oper for the group $\operatorname{GL}(2,\mathbb{C})$ on a complex curve $C$} is a triple $(E,\nabla,L)$, where $E$ is a rank $2$ holomorphic vector bundle, $\nabla$ is a holomorphic (resp. meromorphic) connection on $E$ and $L\subset E$ is a line subbundle satisfying a \emph{non-degeneracy} (or \emph{strictness}) \emph{condition}: the $\mathcal{O}_C$-linear map $\varphi$ defined as the composite
	\begin{align}
	&\varphi:L\xrightarrow{\iota} E\xrightarrow{\nabla}E\otimes\Omega^1_C\xrightarrow{q\otimes\operatorname{id}} (E/L)\otimes\Omega^1_C\\
	(\text{resp. }&\varphi:L\xrightarrow{\iota} E\xrightarrow{\nabla}E\otimes\Omega^1_C(D)\xrightarrow{q\otimes\operatorname{id}} (E/L)\otimes\Omega^1_C(D))\label{strictness-mero}
	\end{align}
	(where $\iota$ is the inclusion, $q$ is the quotient map and $D$ is an effective divisor on $C$) is an isomorphism.
\end{defi}

A $\operatorname{PGL}(2,\mathbb{C})$-oper on $C$ may then be alternatively defined as an equivalence class of $\operatorname{GL}(2,\mathbb{C})$-opers under the following twist relation. Two $\operatorname{PGL}(2,\mathbb{C})$-opers $(E_1,\nabla_1,L_1)$ and $(E_2,\nabla_2,L_2)$ are considered to be equivalent if there is a rank $1$ holomorphic connection $(L,\nabla_L)$ and an isomorphism \[(E_1,\nabla_1)\simeq(E_2,\nabla_2)\otimes(L,\nabla_L)\] sending $L_1$ to $L_2\otimes L$.

The non-degeneracy condition tells us that $L$ is nowhere invariant by $\nabla$, which is exactly the transversality condition of the section $\sigma$ induced by $L$ on $\mathbb{P}(E,\nabla)$. Another condition called the \emph{Griffiths transversality} is asked in the general definition of a $\operatorname{GL}(k,\mathbb{C})$-oper in greater dimensions $k\geq 2$, but we don't need to consider it here since it is automatically satisfied in rank~$k=2$. 

\begin{rmk}
	This alternative definition of a $\operatorname{PGL}(2,\mathbb{C})$-oper on $C$ is equivalent to \autoref{def:mero-PGL-oper} (and \autoref{defi:holo-PGL-oper}, in the holomorphic case). Indeed, as we will show in the following paragraphs, the Riccati foliation associated to a $\operatorname{PGL}(2,\mathbb{C})$-oper must be the projectivization of a rank $2$ linear connection. However, our definitions of meromorphic opers are a little tighter than those suggested by Allegretti and Bridgeland since we impose that the map $\varphi$ in \eqref{strictness-mero} is an isomorphism (compare \cite[Rmk 3.3(i)]{Allegretti_2020}).
\end{rmk}

\paragraph{A self-intersection formula.} Let $(\pi:Q\rightarrow C,\mathcal{F},\sigma)$ be a meromorphic $\operatorname{PGL}(2,\mathbb{C})$-oper with the Riccati foliation having polar divisor $D=\sum_{i=1}^{d}m_ip_i$ on the genus $g$ curve $C$. Because the section $\sigma$ is transverse to $\mathcal{F}$, we have, according to Brunella \cite[Prop. 2.2]{MR3328860} (see also \cite[Prop. 25]{https://doi.org/10.48550/arxiv.1410.4976}),
\begin{equation}\label{self-intersection-number}
\sigma\cdot\sigma=2-2g-\operatorname{deg}(D).
\end{equation}
From this formula, we deduce in particular that the parity of $\deg(E)$, which is the same as the parity of $\sigma\cdot\sigma$, is also equal to the parity of $\deg(D)$. Another consequence is that
\begin{equation}\label{self-intersection-special-cases}
\sigma\cdot\sigma\geq 0\Leftrightarrow
\begin{cases*}
g=0,~\deg(D)=0,1\text{ or }2\\
g=1,~\deg(D)=0.
\end{cases*}
\end{equation}
Thus, apart from those four special values for $(g,\deg(D))$, the self-intersection number of $\sigma$ is negative. This implies that $\sigma$ is the unique minimal section of $Q$ and the only one with negative self-intersection number.

\begin{rmk}\label{rmk:hypothesis}
	Let us denote by $P$ the meromorphic projective structure described by the oper $(\pi:Q\rightarrow C,\mathcal{F},\sigma)$ and by $\mathbb{S}_P$ the associated bordered surface with marked points $\mathbb{M}_P$. Assume, as in \autoref{prop-smooth-moduli-proj-struct}, that if $g=0$ then $|\mathbb{M}_P|\leq 3$. Then, by \autoref{pole1form-vs-poleschwarzian}, this excludes the above three special values when $g=0$ and $\deg(D)=0,1$ or $2$. The fourth one, $(g,\deg(D))=(1,0)$, is not excluded but corresponds to a projective structure without pole; our main theorem (\autoref{thm:local-injectivity}) is already known in this case (cf. \autoref{sec:cplx-proj-struct}) so we might exclude it as well.
\end{rmk}

\paragraph{Lift of a $\operatorname{PGL}(2,\mathbb{C})$-oper to a $\operatorname{GL}(2,\mathbb{C})$-oper.} Suppose we want to \emph{lift} a $\operatorname{PGL}(2,\mathbb{C})$-oper $(\pi:Q\rightarrow C,\mathcal{F},\sigma)$ with polar divisor $D$ to a $\operatorname{GL}(2,\mathbb{C})$-oper on $C$. If $(E,\nabla)$ is a lift of $(Q,\mathcal{F})$, then $\sigma$ corresponds to a unique line subbundle $L$ of $E$ and by \autoref{rmk:invariance-section-subbundle}, $(E,\nabla,L)$ is a $\operatorname{GL}(2,\mathbb{C})$-oper on $C$ whose equivalence class under the twist relation is $(\pi:Q\rightarrow C,\mathcal{F},\sigma)$.

From the self-intersection formula \eqref{self-intersection-number} combined with \autoref{prop:lift}, the possible choices of trace connection for $(E,\nabla,L)$ depend on the parity of the degree of the polar divisor $D$. If it is even, then so is $\deg(E)$ and $(\pi:Q\rightarrow C,\mathcal{F})$ can be lifted to a rank $2$ connection with determinant having any even degree; a trace-free rank $2$ connection, for example. If $\deg(D)$ is odd, however, the determinant bundle must have an odd degree. Here is an example of a valid choice for the trace connection. Note that in this case, the $\operatorname{PGL}(2,\mathbb{C})$-oper has at least one pole, let us choose one, say $p_i$. We would like to choose a trace connection with a single simple pole at $p_i$, so that the corresponding lifts still have polar divisor $D$ (and it is minimal if it was for $(\pi:Q\rightarrow C,\mathcal{F})$). By the Fuchs relation, such a connection must have residue $-1$ at $p_i$, hence trivial local monodromy around $p_i$. By the Riemann-Hilbert correspondence, there exists a unique rank $1$ connection $(\mathcal{O}_C(p_i),\zeta)$ having a single pole of order $1$ at $p_i$ and a trivial global monodromy.

Recall that once we have made a choice of a trace connection, a rank $2$ lift $(E,\nabla)$ is unique up to a twist by one of the $2^{2^g}$ tensor square root of the rank $1$ trivial connection.

\subsection{Formal data of rank \texorpdfstring{$2$}{2} connections and Riccati foliations}

We denote by $\mathbb{C}[[x]]$ the ring of formal power series in the variable $x$, $\mathbb{C}((x))=\mathbb{C}[[x]][x^{-1}]$ its quotient field (the field of formal Laurent series in the variable $x$), and $\mathbb{C}\{x\}$ the ring of convergent power series in the variable $x$ (those with non-zero radius of convergence at $0$).

\paragraph{Local formal invariants.} Here, we work around a singular point, hence the concept of connection (or formal connection) is unnecessary. We thus consider formal rank $2$ linear systems in the complex variable $x$ around a possible pole at $x=0$.

To give a formal classification of such systems is to describe the orbits under the action of formal holomorphic gauge transformations. The following theorem tells us there are "good" representatives in each orbits, called formal normal forms. These in turn provides formal invariants of a system.

\begin{thm}[Section 2 of \cite{MR4423484}]\label{formal-normal-forms}
	Let $dY+\Omega Y=0$ be a rank $2$ linear differential system with \[\Omega=\left(\frac{A_{-m}}{x^m}+\frac{A_{-m+1}}{x^{m-1}}+\cdots+\frac{A_{-1}}{x}+A(x)\right)dx\]
	where $A_{-m},\dots,A_{-1}\in\operatorname{Mat}(2,\mathbb{C})$ are constant matrices, $A_{-m}\neq 0$ and $A\in\operatorname{Mat}(2,\mathbb{C}\{x\})$ is holomorphic. Assume that the pole order $m\in\mathbb{N}^*$ at $x=0$ is \emph{minimal}, i.e. cannot be decreased by a meromorphic gauge transformation of the system.
	
	Then, there exists a formal holomorphic gauge transformation $G\in\operatorname{GL}(2,\mathbb{C}[[x]])$ such that
	\begin{equation*}
	\tilde{\Omega}=G^{-1}\Omega G+G^{-1}dG
	\end{equation*}
	is of one of the following formal normal forms.
	\begin{itemize}
		\item \textbf{Regular case ($m=1$).}
		\begin{align}\label{regular-case}
		\tilde{\Omega}=\begin{pmatrix}
		\lambda^+&0\\
		0&\lambda^-
		\end{pmatrix}dx
		\text{~~ or ~~}
		\tilde{\Omega}=\begin{pmatrix}
		\lambda^+&x^k\\
		0&\lambda^-
		\end{pmatrix}dx,
		\end{align}
		for some non-negative integer $k\in\mathbb{N}$ and $\lambda^\pm=\frac{\lambda_{-1}^\pm}{x}$, where $\lambda^\pm_{-1}\in\mathbb{C}$ are the eigenvalues of $A_{-1}$, such that in the second case $\lambda^+_{-1}=\lambda^-_{-1}+k$. 

		Moreover, the latter formal normal form is unique, while the first one is unique up to the permutation of $\lambda^+$ and $\lambda^-$. %
		\item \textbf{Irregular ($m\geq2$) and unramified ($A_{-m}$ is semi-simple) case.}
		\begin{align}
		\tilde{\Omega}=\begin{pmatrix}
		\lambda^+&0\\
		0&\lambda^-
		\end{pmatrix}dx,
		\end{align}
		for some $\lambda^\pm=\frac{\lambda_{-m}^\pm}{x^m}+\cdots+\frac{\lambda_{-1}^\pm}{x}$ and $\lambda_i^\pm\in\mathbb{C}$ ($\lambda_{-m}^\pm$ are the eigenvalues of $A_{-m}$).
		
		Moreover, this formal normal form is unique up to the permutation of $\lambda^+$ and $\lambda^-$.
		\item \textbf{Irregular ($m\geq2$) and ramified ($A_{-m}$ is not semi-simple) case.}
		\begin{align}
		\tilde{\Omega}=\begin{pmatrix}
		\alpha&\beta\\
		x\beta&\alpha-\frac{1}{2x}
		\end{pmatrix}dx,
		\end{align}
		for some
		$\begin{cases}
		\alpha=\frac{\alpha_{-m}}{x^m}+\cdots+\frac{\alpha_{-2}}{x^2}+\frac{\alpha_{-1}}{x},~~\alpha_i\in\mathbb{C}\\
		\beta=\frac{\beta_{-m}}{x^m}+\cdots+\frac{\beta_{-2}}{x^2},~~\beta_i\in\mathbb{C}.
		\end{cases}$
		
		Moreover, this formal normal form is unique up to the multiplication of $\beta$ by $-1$.
	\end{itemize}
\end{thm}

\begin{proof}[Proof complement]
	A version of this theorem is proven in \cite[Sec. 2]{MR4423484}, allowing twists in addition to gauge transformations. This assumption can be circumvented in the following way. First, it is unnecessary if $\Omega$ has trace-free principal part (or merely trace-free leading coefficient $A_{-m}$). But we can always reduce to this case, by working separately on each terms of the decomposition \[\left[\Omega-\frac{\operatorname{trace}(\Omega^{<0})}{2}I_2\right]+\frac{\operatorname{trace}(\Omega^{<0})}{2}I_2,\] where $\Omega^{<0}$ is the principal part of $\Omega$. Indeed, the first term of this sum has trace-free principal part, while the last term is in the center of $\operatorname{GL}(2,\mathbb{C}[[x]])$ and is thus invariant by formal holomorphic gauge transformation (such a transformation acts by conjugacy on the negative part of the system's matrix).
\end{proof}

\begin{rmk}\label{rmk:invariant-space-vs-eigenspace}
	One must be careful here, not to confuse eigenspaces of the matrix $\Omega$ with subspaces invariant by $d+\Omega$, or with subspaces generated by solutions of the system $dY+\Omega Y=0$.
\end{rmk}

We call the coefficients $(\lambda^+,\lambda^-)$ and $(\alpha,\beta)$ in the above normal forms, considered up to a permutation of $\lambda^+$ and $\lambda^-$ and up to the multiplication of $\beta$ by $-1$, the \emph{basic formal data} of the system. Under the assumption that the pole order $m$ is minimal, the basic formal data of a system together with the value of $m$ fully characterize the formal structure of the system $dY+\Omega Y=0$ (i.e. they determine its equivalence class under formal holomorphic gauge equivalence), except in the regular case\footnote{More precisely: in the regular \emph{resonant} case, i.e. when the eigenvalues of $A_{-1}$ satisfy $|\lambda^--\lambda^+|\in\mathbb{N}$, the system can be formally equivalent to either of the two normal forms \eqref{regular-case}. In the regular non-resonant case, however, only the first one is possible (see \cite[Thm. 16.15]{MR2363178}).}.

\begin{rmk}\label{trace-free-principal-part}
	Since $G$ is formally holomorphic, so is $G^{-1}dG$. Therefore, the gauge transformation $G$ acts by conjugacy on the principal part $\Omega^{<0}$ of the system's matrix. As a consequence, the trace $\operatorname{tr}(\Omega^{<0})$ is a formal invariant of the system. Moreover, it is fully determined by the system's basic formal data.
\end{rmk}

In the regular resonant case, applying the meromorphic gauge transformation (which is a composition of elementary gauge transformations in rank $2$) \[G=\begin{pmatrix}
x^{-k}&0\\
0&1
\end{pmatrix}\] leads to a connection matrix of the form \[\tilde{\tilde{\Omega}}=\begin{pmatrix}
\lambda^-_{-1}&1\\
0&\lambda^-_{-1}
\end{pmatrix}\frac{dx}{x}.\]
Furthermore, in the irregular ramified case, applying the ramification $x=z^2$ and the gauge transformation \[G=\begin{pmatrix}
1&1\\
z&-z
\end{pmatrix}\]
leads to a connection matrix of the form \[\tilde{\tilde{\Omega}}=\begin{pmatrix}
\lambda^+&0\\
0&\lambda^-
\end{pmatrix}dz\]
with $\lambda^\pm=2z(\alpha\pm z\beta)=2(\pm x\beta+z\alpha)$. This ramified formal normal form is unique up to the permutation of $\lambda^+$ and $\lambda^-$.

With those additional transformations, the formal normal forms of \autoref{formal-normal-forms} are all \emph{Hukuhara-Turrittin ramified formal normal forms}. As the above relations shows, their diagonal elements (up to a permutation) can be deduced from the basic formal data, and conversely. Thus, it is just another way to encode the basic formal data, that we will use from now on and call \emph{basic formal data} as well.

In the literature, the system $dY+\Omega Y=0$ is often considered up to formal \emph{meromorphic} gauge transformations. The normal forms obtained in this way are also Hukuharra-Turrittin ramified formal normal forms (see \cite[Thm. 6.8.1]{MR1084379}), but only the set of their diagonal elements modulo $\mathbb{Z}dz/z$ is a formal meromorphic invariant of the system. Those diagonal elements, called "generalized eigenvalues" by Inaba in \cite{MR4545855}, thus lead to a looser notion of formal invariants that we do not use.

\paragraph{Local formal invariants for singular Riccati foliations.} 	Let $\mathcal{F}$ be a singular Riccati foliation on the trivial $\mathbb{P}^1$-bundle over a disc $0\in\mathbb{D}\subset\mathbb{C}$ defined by the equation
\begin{equation}\label{Riccati}
\mathcal{F}:dy+[a(x)y^2+b(x)y+c(x)]\frac{dx}{x^m}=0,~~ (x,y)\in\mathbb{D}_x\times\mathbb{P}^1_y,
\end{equation}
were $a,b,c\in\mathbb{C}\{x\}$ are holomorphic functions. Assume that the pole order $m\in\mathbb{N}^*$ at $x=0$ is \emph{minimal}, i.e. cannot be decreased by a meromorphic gauge transformation.

Locally, this Riccati equation is the projectivization of a unique trace-free rank $2$ system (this is just the local part of the proof of \autoref{prop:lift}). After projectivization of the Hukuhara-Turrittin ramified formal normal form of this system, we immediately deduce a ramified formal normal forms for the Riccati equation (see also \cite[Prop. 5.1]{MR3522824}). In each case (regular, irregular unramified and irregular ramified), the normal form only depends on a coefficient $\lambda=\lambda^--\lambda^+\in\mathbb{C}[z^{-1}]$, well-defined up to a multiplication by $-1$, that we call the \emph{basic formal data} of the Riccati equation.

We may assume without loss of generality that $\{y=0\}$ is transverse to $\mathcal{F}$ (if not, then we can send any transverse section to infinity by a holomorphic gauge transformation of the trivial bundle). Then, there is a unique meromorphic affine gauge transformation such that
\begin{equation*}
\mathcal{F}:dy+\left(y^2+\frac{g(x)}{x^{n}}\right)dx=0
\end{equation*}
with $g\in\mathcal{O}^*(\mathbb{D})$ holomorphic and non-vanishing and $m\leq n\leq 2m$ (the computation is just a meromorphic version of the proofs of Propositions \ref{foliation-vs-schwarzianeq}(i) and \ref{pole1form-vs-poleschwarzian}(ii)). The integer $n$ depends only on the initial singular Riccati equation \eqref{Riccati}.

\begin{defi}
	The \emph{irregularity index} of the equation \eqref{Riccati} at $x=0$ is defined as \[\nu:=\max\left(0,\frac{n-2}{2}\right)\in\frac{1}{2}\mathbb{N}.\]
\end{defi}

The irregularity index is equal to zero if and only if $x=0$ is a regular singularity of equation \eqref{Riccati}. It is a positive integer (resp. half of a positive integer) if and only if $x=0$ is an irregular unramified singularity (resp. an irregular ramified singularity). This is true in particular for Riccati foliations associated to a meromorphic projective structure (see \eqref{construction-oper}), and in that case $n$ is the order of the pole of the projective structure (see \autoref{pole1form-vs-poleschwarzian}(ii) and \cite[proof of Thm. 5.2]{Allegretti_2020}).

Note that the irregularity index is well-defined for any Riccati equation \eqref{Riccati}, i.e. even if $m$ is not minimal. The adjectives "regular", "irregular unramified" and "irregular ramified" then refer to a bimeromorphically gauge equivalent equation with minimal pole order.

\paragraph{Formal solutions, formal local monodromy and residues.} From the Hukuhara-Turrittin ramified formal normal forms, we deduce that to a system $dY+\Omega Y=0$ as in \autoref{formal-normal-forms} there is a formal fundamental solution of the form \[Y(z)=G(z)(z^2)^L e^{Q(z)},\] where $G\in\operatorname{GL}(2,\mathbb{C}((z)))$ corresponds to a gauge transformation putting the system into a normal form, $L\in\operatorname{Mat}(2,\mathbb{C})$ is the \emph{residue matrix} of the matrix $\tilde{\tilde{\Omega}}$ of that normal form, and $Q\in\operatorname{Mat}(2,\mathbb{C}[z^{-1}])$ is a diagonal matrix such that $dQ$ (the so-called \emph{irregular part}) corresponds to all other terms in the principal part of $\tilde{\tilde{\Omega}}$. For some historical details about the Hukuhara-Turrittin normal forms, see \cite[Thm. 3.3.1]{MR3495546}.

When going around the pole at $z=0$ in the counterclockwise direction, this formal solution gets multiplied (to the right) by $e^{2i\pi L}$. This is called the \emph{formal monodromy matrix} of the system, with respect to the fundamental solution $Y$. Another formal solution $YC$, for some $C\in\operatorname{GL}(2,\mathbb{C})$, has formal monodromy $C^{-1}e^{2i\pi L}C$.

Fixing the residue of a formal normal form (either of the theorem \autoref{formal-normal-forms} or of Hukuhara-Turrittin) also fixes the formal monodromy.

\paragraph{Global formal invariants.} We define the \emph{basic formal data} of a meromorphic connection $(E,\nabla)$ having minimal polar divisor $D=\sum_{i=1}^dm_ip_i$ as the tuple of its basic formal data at each pole, presented as the diagonal elements of the Hukuhara-Turrittin ramified formal normal form as above, up to a permutation. Thus, the basic formal data of a connection is an element of the quotient set $((\mathbb{C}[z^{-1}])^2)^d/(\mathfrak{S}_2)^d$. Having in mind the construction of moduli spaces of meromorphic connections with fixed formal data, we call an element \[\Lambda=(\lambda^{+,(i)},\lambda^{-,(i)})_{1\leq i\leq d}\in((\mathbb{C}[z^{-1}])^2)^d\] a \emph{formal data} (dropping the adjective "basic"). This is a basic formal data together with an extra piece of information (namely an ordering, encoded by a sign).

\begin{rmk}\label{Fuchs-relation-ramified}
	The Fuchs relation \eqref{Fuchs-relation} induces a relation on the formal data of $(E,\nabla)$ (cf. \cite[p. 12 and the proof of Prop. 1.3]{MR4545855}).%
\end{rmk}

Similarly, the \emph{formal data} of a Riccati equation $(Q,\mathcal{F})$ having minimal polar divisor $D=\sum_{i=1}^dm_ip_i$ is an element \[\bar{\Lambda}=(\lambda^{(i)})_{1\leq i\leq d}\in(\mathbb{C}[z^{-1}])^d\] such that its equivalence class under the action of $\{\pm1\}^d$ is the basic formal data of $(Q,\mathcal{F})$, i.e. the tuple of its basic formal data at each pole.  We denote by $\lambda^{(i)}_{-1}$ the residue of $\lambda^{(i)}$. Note that, in the irregular ramified case, this is the residue with respect to the variable $z$ and we always have $\lambda_{-1}=0$. In all other cases, this is the residue with respect to the variable $x$.

If $\Lambda=(\lambda^{+,(i)},\lambda^{-,(i)})$ is the formal data of $(E,\nabla)$, we also say $\bar{\Lambda}=(\lambda^{-,(i)}-\lambda^{+,(i)})$ is its \emph{projectivization}. It is the formal data of the projectivization $\mathbb{P}(E,\nabla)$.

\paragraph{The space $\mathcal{P}^\circ(\mathbb{S},\mathbb{M},(\lambda_{-1}^{(i)}))$.} An equivalence class of marked meromorphic projective structure in $\mathcal{P}^\circ(\mathbb{S},\mathbb{M})$ corresponds (cf. \eqref{marquing-covering-map}) to a unique meromorphic projective structure without apparent singularity $P$ on some complex curve $C$ on $S$. Let us denote by $D=\sum_{i=1}^d n_ip_i$ its polar divisor. By \autoref{indigenousbundle-meromorphic}, $P$ corresponds to a unique $\operatorname{PGL}(2,\mathbb{C})$-oper $(\pi:Q\rightarrow C,\mathcal{F},\sigma)$ on $C$, with minimal polar divisor $\tilde{D}=\sum_{i=1}^dm_ip_i$, where  $m_i:=\lceil \frac{n_i}{2}\rceil$. Then, $\nu_i:=\max(0,(n_i-2)/2)$ is the irregularity index of $(\pi:Q\rightarrow C,\mathcal{F})$ at $p_i$. Whether a pole $p_i$ is a regular, irregular, ramified or unramified singularity of the Riccati foliation depends only on $\nu_i$.

There is a well-defined map
\begin{align}
\tilde{R}:\mathcal{P}^\circ(\mathbb{S},\mathbb{M})\longrightarrow\mathbb{C}^d/\{\pm 1\}^d
\end{align}
sending an equivalence class of marked meromorphic projective structures without apparent singularity to the tuple of residues $(\pm\lambda^{(i)}_{-1})_{1\leq i\leq d}$ of the basic formal data of the associated $\operatorname{PGL}(2,\mathbb{C})$-oper. Since the residue at an irregular ramified singularity is equal to zero, the map $\tilde{R}$ may not be surjective. %

\begin{rmk}\label{rk:residues}
	We now have two quantities that we both call the residue at a regular singularity $p_i$ of $P$: $\pm\lambda_{-1}^{(i)}$ and $\operatorname{res}_2(p_i)$ (cf. the definition \eqref{def:residue-of-order-two}). However, they are different! They are related by the formula $(\lambda_{-1}^{(i)})^2=(1-2\operatorname{res}_2(p_i))/4=\theta(p_i)^2/4$.
\end{rmk}

Let $(\pm\lambda_{-1}^{(i)})$ be a point in $\mathbb{C}^d/\{\pm 1\}^d$. Then, we denote by
\[\mathcal{P}^\circ(\mathbb{S},\mathbb{M},(\pm\lambda_{-1}^{(i)}))=\tilde{R}^{-1}((\pm\lambda_{-1}^{(i)}))\]
the set corresponding to projective structures with residues $(\pm\lambda_{-1}^{(i)})$.

We will also denote this subset by $\mathcal{P}^\circ(\mathbb{S},\mathbb{M},(\lambda_{-1}^{(i)}))$, together with the extra information of a choice of a representative of each residue in $\mathbb{C}$.

\begin{prop}\label{prop:fix-residues}
	Assume that is $g=0$, then $|\mathbb{M}|\geq 3$. Then, if it is non-empty, the space $\mathcal{P}^\circ(\mathbb{S},\mathbb{M},(\lambda_{-1}^{(i)}))$ is a smooth submanifold of $\mathcal{P}^\circ(\mathbb{S},\mathbb{M})$ of non-negative even dimension \[6g-6+\sum_{i=1}^d(2\nu_i+3)-\sum_{\substack{i=1\\ \nu_i\in\mathbb{N}}}^d1.\]
\end{prop}

\begin{proof}
	The residue map $\tilde{R}$ factors through the restriction of the covering map $M_{|\mathcal{P}^\circ(\mathbb{S},\mathbb{M})}$ (cf. \eqref{marquing-covering-map}) and the residue map (defined in a similar way) \[\tilde{\tilde{R}}:\operatorname{im}(M_{|\mathcal{P}^\circ(\mathbb{S},\mathbb{M})})\longrightarrow \mathbb{C}^d/\{\pm 1\}^d,\] so that $\tilde{R}=\tilde{\tilde{R}}\circ M_{|\mathcal{P}^\circ(\mathbb{S},\mathbb{M})}$. Since smooth covering maps are submersions, we are left to prove that $\tilde{\tilde{R}}$ is a holomorphic submersion too.
	
	Consider a point $[(t,P_0),\phi_0]\in\mathcal{P}(X/B;\mathcal{D})$, contained the domain of definition of $\tilde{\tilde{R}}$ (here, we use the notations of \autoref{section-meromorphic-projective-structures}, in particular $B=\mathcal{T}(S,d)$). In fact, we are going to show that the residue map is a submersion even if we fix the complex structure $t$. There is a unique bijection between the spaces $\mathcal{P}(X/B;\mathcal{D})_t$ and $\mathcal{Q}(X/B;\mathcal{D})_t$ sending $[(t,P_0),\phi_0]$ to $(t,0)$.

	Consider a meromorphic projective structure $[(t,P),\phi]\in\mathcal{P}(X/B;\mathcal{D})_t$ on $C_t$ in the domain of definition of $\tilde{\tilde{R}}$ (endowed with the induced topology). The idea is to find, for each coordinate function of $\tilde{\tilde{R}}$, a non identically zero quadratic differential --that is to say a non-zero vector in $\mathcal{Q}(X/B;\mathcal{D})_t$-- such that the directional derivative of $\tilde{\tilde{R}}$ at $[(t,P),\phi]$ (i.e. at the corresponding point in $\mathcal{Q}(X/B;\mathcal{D})_t$) along that vector is non-zero. We are going to use the Riemann-Roch theorem as in the proof of \cite[Lem. 6.1]{MR3349833}, but we must first clarify the relation between residues of a projective structure and the coefficients in the expansion of the corresponding quadratic differential (in particular in the unramified irregular case).

	Let $p$ be a pole of $P$, of order $n$.	It is also a pole of the associated $\operatorname{PGL}(2,\mathbb{C})$-oper. To relate its residue at $p$ with the quadratic differential $\phi$, we first apply \autoref{foliation-vs-schwarzianeq-mero} locally around $p$. Then we lift the Riccati equation a rank $2$ system. This yields the system
	\[dY+\begin{pmatrix}
	0 & -q(x)/2\\
	1 & 0
	\end{pmatrix}Ydx=0.\]
	
	Assume first that $p$ is an irregular non-ramified singularity of $P$. Then, $n\geq 3$ and $n=2m$ for some $m\in\mathbb{N}$. After the elementary gauge transformation
	\[G=\begin{pmatrix}
	1 & 0\\
	0 & x^m
	\end{pmatrix},\]
	the system's matrix becomes
	\[\begin{pmatrix}
	0 & -q(x)x^m/2\\
	x^{-m} & mx^{-1}
	\end{pmatrix},\]
	and the pole order is now minimal. In order to make the computation simpler, we consider instead the system with trace-free matrix
	\[\begin{pmatrix}
	-mx^{-1}/2 & -q(x)x^m/2\\
	x^{-m} & mx^{-1}/2
	\end{pmatrix}.\]
	The basic formal data are the eigenvalues $\lambda^{\pm}=\frac{\lambda_{-m}^\pm}{x^m}+\cdots+\frac{\lambda_{-1}^\pm}{x}$ of the negative part of this matrix (cf. the proof of \autoref{formal-normal-forms}). Therefore, they must satisfy the identity
	\begin{align*}
	(\lambda^{\pm})^2&=\frac{(\lambda_{-m}^\pm)^2}{x^{2m}}+\cdots+\frac{a^\pm\lambda_{-1}^\pm+b^\pm}{x^{m+1}}+\cdots+\frac{(\lambda_{-1}^\pm)^2}{x^2}\\
	&=\left[\frac{q(x)}{2}+\frac{m^2}{4x^2}\right]^{<0},
	\end{align*}
	where $a^\pm=2\lambda_{-m}^\pm\neq0$ (because we are in a trace-free and non-ramified situation) and $b^\pm\in\mathbb{C}$ does not depend on $\lambda_{-1}^\pm$. As a consequence, if $q_{-(m+1)}$ denotes the $(-(m+1))$-coefficient in the power series development of $q$, then $\lambda_{-1}^\pm=q_{-(m+1)}/(2a^{\pm})-b^\pm/ a^{\pm}$. Hence , \[\pm\lambda_{-1}=\pm\frac{1}{2}\left(\frac{1}{a^-}-\frac{1}{a^+}\right)q_{-(m+1)}-\pm\left(\frac{b^-}{a-}-\frac{b^+}{a^+}\right)\] and we deduce that the residue map $\tilde{\tilde{R}}$ is holomorphic. Moreover, the derivative of $\pm\lambda_{-1}$ with respect to $q_{-(m+1)}$ is non-zero.

	By virtue of the Riemann-Roch theorem (cf. formula \eqref{Riemann-Roch}), \[\dim(H^0(C_t,(T^*C_t)^{\otimes2}((m+1)\cdot p)))=3g-3+(m+1),\]
	and this dimension is positive except if $(g,m)=(0,2)$ (in this particular case, $|\mathbb{M}|\geq 3$ hence $P$ must have another pole $p'$, distinct from $p$; thus we can consider quadratic differentials with polar divisor $(m+1)\cdot p + p'$ and adapt the following argument). The subset of quadratic differentials with a pole of order exactly $(m+1)$ at $p$ is open and dense. Let $\tilde{\phi}$ be one of its non-identically vanishing elements. Then, locally around $p$, its $-(m+1)$-coefficient $\tilde{q}_{-(m+1)}$ is non zero, and the above computations show that the directional derivative at $[(t,P),\phi]$ of the residue at $p$ is non-zero. This concludes the proof.
	
	Finally, if $p$ is a regular singularity, i.e. $n=1$ or $2$, then we find $\pm\lambda_{-1}=\pm q_{-1}$ or $\pm\lambda_{-1}=\pm q_{-2}\pm 1/2$, respectively, and similar arguments apply.
\end{proof}

\subsection{A smooth moduli space of singular Riccati foliations with fixed formal data}\label{smooth-moduli-space}

\paragraph{A moduli space of meromorphic rank $2$ connections with fixed formal data.}

In his papers, Inaba introduced an alternative way of encoding the formal data of a meromorphic connection, more suitable for his construction of moduli spaces, via the notion of a \emph{parabolic connection of generic ramified type with a given exponent} \cite[Def. 2.1]{MR4545855} (see also \cite{MR4563425}). This is, roughly speaking, a meromorphic connection together with the additional data of a \emph{parabolic structure} and a collection of \emph{ramified structures}. The precise definitions are quite involved and we won't state them here. Instead, we refer the Reader to Inaba's papers and explain only the features that are relevant to our purpose.

The parabolic structure (in the sense of Inaba, which we are going to make more precise shortly) serves two different functions. First, it is part of the data encoding the formal data of the connection. Second, it is key to the smoothness of moduli spaces, via a \emph{stability} condition.

Inaba constructed, in particular, a moduli space of (parabolic and stable) rank $2$ meromorphic connection on a given smooth curve $C$, with a given polar divisor $D$ and a given formal data $\Lambda$. Note that not all formal data $\Lambda$ (in the sense of a collection of elements of $\mathbb{C}[z^{-1}]$) is the formal data of a meromorphic connection: it must satisfy the Fuchs relation (cf. \autoref{Fuchs-relation-ramified}). Another condition appears in the work of Inaba, which is needed in order for his construction to be valid: he calls it the "genericity condition" (cf. the assumption of Proposition 1.3, in \cite{MR4545855}); this condition relates to the formal data at ramified irregular singular points: the $\lambda^{\pm}$ are called "generic" if their $\frac{1}{z^{2m-2}}$-coefficient is non-zero\footnote{In fact, Inaba defines the "genericity" of a pole in general, but the condition is empty if the singularity is regular, and it is superfluous (i.e. not needed in \cite[Prop. 1.3]{MR4545855}) if the singularity is irregular unramified. Thus, we only need to consider "genericity" for irregular ramified singularities.}, where $m$ is the order of the pole. As the following lemma shows, this condition is automatically satisfied by the formal data of meromorphic connections in our setting, because we assume that they have minimal polar divisor (which is more natural when defining formal data) and a trace-free leading coefficient at each irregular ramified singularities. Thus, we won't need to talk about this "genericity" condition in what follows.

\begin{lem}\label{lemma-genericity}
	Let $dY+\Omega Y=0$ be a rank $2$ linear differential system as in \autoref{formal-normal-forms}, with a pole order $m\in\mathbb{N}^*$ at $x=0$ that is minimal. Suppose that we are in the irregular ($m\geq2$) and unramified ($A_{-m}$ is semi-simple) case, and assume moreover that the leading coefficient $A_{-m}$ is trace-free. Then, the formal data of the system are "generic" in the sense of Inaba.
\end{lem}

\begin{proof}
	This is essentially a byproduct of the proof of \autoref{formal-normal-forms}. We are going to express the eigenvalues of $\tilde{\tilde{\Omega}}$ as functions of the entries of $\Omega$.%
	
	The leading term $A_{-m}$ of the matrix $\Omega$ admits an eigenvalue of multiplicity $2$, which is equal to $0$ since $A_{-m}$ is trace-free. After a preliminary constant gauge transformation, the leading term may be put into an upper-triangular form and we can write
	\[\Omega=\left[\frac{\begin{pmatrix}
		0&b_{-m}\\0&0
		\end{pmatrix}}{x^m}+\frac{\begin{pmatrix}
		a_{-m+1}&b_{-m+1}\\c_{-m+1}&d_{-m+1}
		\end{pmatrix}}{x^{m-1}}+\cdots\right]dx,\] with $b_{-m}\neq0$ (the points of ellipsis contain only higher order terms).
	
	If $c_{-m+1}=0$, the gauge transformation $G=\operatorname{diag}(1,x)$ kills the term of order $-m$, which contradicts our assumption that $m$ is minimal. Consequently, $c_{-m+1}\neq0$.
	
	Now, let us compute the first terms of the formal data $\lambda^{\pm}$ of this system. After the substitutions \[x=z^2\text{ followed by the elementary transformation }Y=\begin{pmatrix}
	1&0\\0&z
	\end{pmatrix}\tilde{Y}\] in the connection $d+\Omega$, we get a new connection matrix
	\begin{equation}
	\tilde{\Omega}=2\left[\frac{\begin{pmatrix}
		0&b_{-m}\\c_{-m+1}&0
		\end{pmatrix}}{z^{2(m-1)}}+\frac{\begin{pmatrix}
		a_{-m+1}&0\\0&d_{-m+1}
		\end{pmatrix}}{z^{2(m-1)-1}}+\cdots+\frac{\begin{pmatrix}
		0&b_{-1}\\c_0&1/2
		\end{pmatrix}}{z}+\cdots\right]dz,
	\end{equation}
	whose eigenvalues are $\lambda^{\pm}dz= (\pm\frac{i\sqrt{b_{-m}c_{-m+1}}}{z^{2m-2}}+\cdots)dz$. But $b_{-m}c_{-m+1}\neq0$, hence $\lambda^{\pm}$ are generic.
\end{proof}

\paragraph{Parabolic structures.}

If $p$ is a point of $C$, we denote by $\mathcal{O}_p$ the ring of germs of holomorphic functions at $p$, i.e. the stalk of the sheaf $\mathcal{O}_C$ at $p$. Let $I$ be the maximal ideal of $\mathcal{O}_p$, and denote by $\mathcal{O}_{mp}=\mathcal{O}_p/I^m$, for $m\in\mathbb{N}^*$. Concretely, in a local coordinate $x$ centered at $p$, $I=(\mathbf{x})$ ($\mathbf{x}$ is the germ of $x$ at $0$) and elements in $\mathcal{O}_{mp}$ are germs of holomorphic functions modulo $x^m$. Sticking to Inaba's notations, we denote by $mt=\operatorname{Spec}(\mathcal{O}_{mp})$ the $m$-th \emph{infinitesimal neighborhood} of $p$ in~$C$.

Let $E$ be a rank $2$ holomorphic vector bundle on $C$. The restriction $E_{|mp}$ of the sheaf $E$ to this infinitesimal neighborhood is a locally free sheaf of $\mathcal{O}_{mp}$-modules of rank $2$.

Let $\nabla$ be a meromorphic connection on $E$, with its polar divisor $D=\sum_{i=1}^d m_ip_i$, $m_i\in\mathbb{N}^*$, being minimal. Then, for all $1\leq i\leq d$, it induces a morphism \[\nabla_{|m_ip_i}:E_{|m_ip_i}\longrightarrow E_{|m_ip_i}\otimes\Omega^1_C(D),\]
Indeed, after choosing a local coordinate $x_i$ around $p_i$ as well as a trivialization coordinate of the bundle, we see that the image $\nabla s$, considered modulo $x_i^{m_i}$, of a section $s$ depends only on the $m_i$ first coefficients of the power series development of $s$ (and on the negative part of the connection matrix).

A \emph{parabolic structure} $\mathcal{L}=\{l_j^{(i)}\}^{1\leq i\leq d}_{0\leq j\leq s_i}$ (in the sense of \cite[Def. 2.1(iv)]{MR4545855}) (where $s_i=1$ or $2$) on $(E,\nabla)$, is the data, for each pole $p_i$, of a strictly descending filtration of $E_{|m_ip_i}$ by $\mathcal{O}_{m_ip_i}$-submodules \[l_0^{(i)}\supset\cdots\supset l_{s_i}^{(i)}\] of $E_{|m_ip_i}$, starting by $l_0^{(i)}=E_{|m_ip_i}$ and ending by $l_{s_i}^{(i)}=0$, such that \[\nabla_{|m_ip_i}(l^{(i)}_j)\subset l^{(i)}_j\otimes\Omega^1_C(D)\] for all $1\leq i\leq d$ and $0\leq j\leq s_i$, with all the $s_i\in\{1,2\}$ maximal among such filtrations.

In other words, if locally around $p_i$ there exists a length $1$ invariant $\mathcal{O}_{m_ip_i}$-submodule of the negative part $\nabla_{|m_ip_i}$ of the connection, then we must have $s_i=2$ and $l^{(i)}_1$ must be a length $1$ invariant submodule of $\nabla_{|m_ip_i}$. Otherwise, $s_i=1$.

\begin{rmk}\label{rmk:invariant-space-vs-eigenspace-negative-part}
	In contrast with \autoref{rmk:invariant-space-vs-eigenspace}, a $\mathcal{O}_{m_ip_i}$-submodule of $E_{|m_ip_i}$ is invariant by $\nabla_{|m_ip_i}$ if and only if it is an eigenspace of the connection matrix of $\nabla_{|m_ip_i}$.
\end{rmk}

A \emph{parabolic weight} for the \emph{parabolic connection} $(E,\nabla,\mathcal{L})$ is a choice of a rational number for each non-trivial $l_j^{(i)}$ ($1\leq j\leq s_i$), i.e. a tuple of rational numbers $\alpha=(\alpha_j^{(i)})^{1\leq i\leq d}_{1\leq j\leq s_i}$, satisfying \[\begin{cases}
0<\alpha_j^{(i)}<1\\\alpha_{j_1}^{(i)}<\alpha_{j_2}^{(i)}\text{ if }j_1<j_2,
\end{cases}\] for any $i$, and $\alpha_j^{(i)}\neq \alpha_{j'}^{(i')}$ for $(i,j)\neq (i',j')$.

\paragraph{Stability of parabolic connections.}

\begin{defi}
	Let $(E,\nabla,\mathcal{L})$, with $\mathcal{L}=\{l_j^{(i)}\}^{1\leq i\leq d}_{0\leq j\leq s_i}$, be a rank $2$ parabolic connection on $C$ with polar divisor $D$ and parabolic weight $\alpha$. The parabolic connection is said to be \emph{$\alpha$-stable} if all invariant line subbundle $F$ of $E$ satisfies the inequality
	\begin{equation}\label{alpha-stability}
	\deg(E)-2\deg(F)+\sum_{\substack{1\leq i\leq d\\F_{|m_ip_i}\not\subset l^{(i)}_1}} \mu^{(i)}-\sum_{\substack{1\leq i\leq d\\F_{|m_ip_i}\subset l^{(i)}_1}} \mu^{(i)}>0,
	\end{equation}
	where $\mu^{(i)}=\alpha^{(i)}_2-\alpha^{(i)}_1$.
\end{defi}

Note that $\mu^{(i)}$ is well-defined here, since the existence of an invariant subbundle $F$ implies $s_i=2$ for all $1\leq i\leq d$. If $s_i=1$ for some $i$, then the connection is irreducible around $p_i$, therefore there are no invariant line subbundle and the above condition is empty: the parabolic connection is thus stable with respect to any of its parabolic weights.

\paragraph{The moduli space $\mathcal{M}^{\alpha}_{C,D,\Lambda}$.} Let $C$ be a compact smooth curve, $D$ an effective divisor on $C$ and $\Lambda$ a "generic" formal data (compatible with $D$). The moduli space $\mathcal{M}^{\alpha}_{C,D,\Lambda}$ of $\alpha$-stable parabolic meromorphic rank $2$ connections on $C$ with polar divisor $D$ minimal and formal data $\Lambda$ was constructed by Inaba (note that the degree of the vector bundles on which the connections live is fixed by $\Lambda$ and the Fuchs relation).

\begin{thm}[cf. Theorem 3.1 of \cite{MR4545855}]\label{Inaba-thm}
	There is a structure of a smooth complex manifold on $\mathcal{M}^{\alpha}_{C,D,\Lambda}$, of dimension \[\operatorname{dim}(\mathcal{M}^{\alpha}_{C,D,\Lambda})=2(4(g-1)+1+\deg(D))\] if non-empty.
	
	Moreover, there exists a smooth manifold $\mathcal{M}'$, a local homeomorphism
	\[\mathcal{M}'\longrightarrow \mathcal{M}^{\alpha}_{C,D,\Lambda}\]
	and a universal family of parabolic meromorphic connections on $\mathcal{M}'$.

\end{thm}

\begin{rmk}
	Such a moduli space was studied, in the case of rank $2$ connections on $C=\mathbb{P}^1$, by Diarra and Loray in \cite{MR4423484}. See also the earlier work by Inaba and Saito regarding unramified meromorphic connections \cite{MR3079310} and the subsequent paper \cite{MR4563425} dealing with generalized isomonodromic deformations.
\end{rmk}

\paragraph{Parabolic structures for connections without projectively apparent singularity.} Assume that $(E,\nabla)$ has no \emph{projectively apparent singularity} (i.e. no simple pole with scalar local monodromy). Then, if its formal data $\Lambda$ are fixed (in particular, if we have chosen an ordering of the diagonal elements in its Hukahara-Turrittin ramified formal normal form), there is a canonical parabolic structure on  $(E,\nabla)$. It is determined as follows.

Consider a pole $p_i$. Locally around $p_i$, by \autoref{rmk:invariant-space-vs-eigenspace-negative-part}, any one-dimensional (formal) eigenspace of the negative part of a connection matrix around $p_i$ in a formal normal form $\tilde{\Omega}$ of \autoref{formal-normal-forms} induces a length~$1$ invariant submodule of the negative part $\nabla_{|m_ip_i}$ of $\nabla$.

Suppose first that $p_i$ is a regular singularity, and that $\tilde{\Omega}=\operatorname{diag}(\lambda^+,\lambda^-)dx$ and $\lambda^+=\lambda^-$. In this case, all subspaces are invariant by $\tilde{\Omega}$, hence there is no canonical choice of an eigenspace. But $\nabla$ must have projectively trivial local monodromy around $p_i$, hence this does not occur by assumption.

In all other cases when $p_i$ is a regular or irregular unramified singularity, there is a unique formal eigenspace of the negative part of $\tilde{\Omega}$ associated to $\lambda^+$. Thus, we have $s_i=2$ and we canonically define $l^{(i)}_1$ as the induced length $1$ invariant submodule of $\nabla_{|m_ip_i}$.

Finally, if $p_i$ is an irregular ramified singularity, then there are no length $1$ invariant submodule\footnote{This is a property of ramified singularity.} of $\nabla_{|m_ip_i}$, hence $s_i=1$ and the filtration $l^{(i)}_j$ must be the trivial one: $l_0^{(i)}=E_{|m_ip_i}\supset l_1^{(i)}=0$.

\paragraph{Opers correspond to $\alpha$-stable connections.} We are now going to show that (if we exclude special values of $(g,\deg(D))$) the connections associated with $\operatorname{GL}(2,\mathbb{C})$-opers are always $\alpha$-stable, for some parabolic weight $\alpha$. Furthermore, if we restrict ourselves to opers with a fixed formal data, then it is possible to find a common parabolic weight $\alpha$ with respect to which they are \emph{all} $\alpha$-stable.

\begin{lem}\label{lem-stability}
	Let $C$ be a genus $g$ curve, $D$ be an effective divisor on this curve and $\Lambda$ be a formal data adapted to $D$. Assume that $2-2g-\deg(D)<0$. Then, there exists a parabolic weight $\alpha$ adapted to $\Lambda$ such that: if $(E,\nabla,\mathcal{L})$ is a  parabolic rank $2$ meromorphic connections on $C$ with polar divisor $D$, formal data $\Lambda$ and admitting a subbundle $L\subset E$ such that $(E,\nabla,L)$ is a $\operatorname{GL}(2,\mathbb{C})$-opers over $C$, then $(E,\nabla,\mathcal{L})$ is $\alpha$-stable.
\end{lem}

\begin{proof}
	Let  $(E,\nabla,\mathcal{L})$ be a parabolic meromorphic connections on $C$ with polar divisor $D$ and formal data $\Lambda$. Note that the form of a parabolic weight $\alpha=(\alpha_j^{(i)})^{1\leq i\leq d}_{1\leq j\leq s_i}$ associated to this parabolic connection, i.e. the values of $d$ and the $s_i$, only depends on $\Lambda$.
	
	If $(E,\nabla)$ is irreducible (for example, if it admits a ramified irregular singularity), then it is stable with respect to any parabolic weight.
	
	Suppose $(E,\nabla)$ is reducible, and let $F$ be a line subbundle of $E$, invariant by $\nabla$. By assumption, there exists a line subbundle $L$ such that $(E,\nabla,L)$ is a $\operatorname{GL}(2,\mathbb{C})$-oper on $C$. Denote by $\sigma$ the associated section on the $\mathbb{P}^1$-bundle $\mathbb{P}(E)$.
	Our assumptions on $g$ and $\deg(D)$, together with the transversality of $\sigma$ with respect to the foliation $\mathbb{P}(\nabla)$ imply by equation \eqref{self-intersection-number} that $\sigma$ has negative self-intersection number. Further, we know that it is the unique section with negative self-intersection number.	Because $F$ is invariant by the connection, it is distinct from $L$. Hence  $\deg(E)-2\deg(F)\geq -\sigma\cdot\sigma>0$ (cf. \autoref{P1-bundle-facts}). Since the self-intersection number is an integer, we necessarily have $\deg(E)-2\deg(F)\geq1$. But this depend neither on the parabolic connection $(E,\nabla,\mathcal{L})$, nor on the invariant subbundle $F$. As a consequence, we readily see that it is possible to chose a parabolic weight $\alpha$ adapted to $\Lambda$ in such a way that the condition (\ref{alpha-stability}) is always fulfilled, hence the conclusion.
\end{proof}

\paragraph{Connections without projectively apparent singularity.} We denote by
\[\mathcal{M}^{\alpha,\circ}_{C,D,\Lambda}\] the subset of parabolic connections with no projectively apparent singularity. It is an open subset of $\mathcal{M}^{\alpha}_{C,D,\Lambda}$.

As we have observed previously, a parabolic connection $(E,\nabla,\mathcal{L})$ without apparent singularity is entirely determined by the subjacent meromorphic connection $(E,\nabla)$ and the formal data $\Lambda$. Hence, the parabolic structure is superfluous, in a sense, and from now on we will refer to elements of this space as meromorphic connections (keeping in mind that the $\alpha$-stability is defined with respect to the parabolic structure).\\

We wish now to construct smooth moduli spaces of singular Riccati foliations containing the foliations subjacent to $\operatorname{PGL}(2,\mathbb{C})$-opers associated to elements in the moduli spaces $\mathcal{P}^\circ(\mathbb{S},\mathbb{M},(\lambda_{-1}^{(i)}))$ of equivalence classes of marked meromorphic projective structures on $S$, with $d$ poles of orders $(n_i)$ and (signed) residues $(\lambda_{-1}^{(i)})$.

\paragraph{Connections with a fixed trace.} We are now going to study the subspaces of $\mathcal{M}^{\alpha,\circ}_{C,D,\Lambda}$ consisting of connection with a specified trace, equal to $(\mathcal{O},d)$ or $(\mathcal{O}(p_i),\zeta)$ for some pole $p_i$. According to \autoref{sec:proj-struct-GL-opers}, we know that, depending on the parity of its polar divisor, a $\operatorname{PGL}(2,\mathbb{C})$-oper can be lifted to such a connection. Once the trace is fixed, a moduli space of Riccati foliations is obtained as the quotient by a discrete group.

Let us denote by $\Gamma$ the moduli space of holomorphic rank $1$ connections on $C$, and $\Gamma_1$ the moduli space of rank $1$ meromorphic connections with a single simple pole having residue $-1$.

Recall from \autoref{trace-free-principal-part} that the negative parts of the traces of the connection matrices are equal to the traces of the corresponding normal forms of \autoref{formal-normal-forms}. Chose a formal data $\Lambda$ such that an element of $\mathcal{M}^{\alpha,\circ}_{C,D,\Lambda}$ has holomorphic trace\footnote{Note that without changing the projective equivalence class of $\Lambda$ we can make such a choice. In other words, given a formal data $\bar{\Lambda}$, we can always find a $\Lambda$ such that its projectivization is $\bar{\Lambda}$ and its trace is such that $T_0$ is well-defined. The same is true in the case of $T_1$ bellow.}. Then, the trace map
\begin{align*}
T_0:\mathcal{M}^{\alpha,\circ}_{C,D,\Lambda}&\longrightarrow\Gamma\\
(E,\nabla)&\longmapsto (\det(E),\operatorname{tr}(\nabla))
\end{align*}
is well-defined. It is holomorphic.

The space $\Gamma$ has a group structure and acts holomorphically on both the source and the target of the trace map $T_0$ in an equivariant way: by $\cdot\otimes(L,\nabla)$ and $\cdot\otimes(L,\nabla)^{\otimes2}$, respectively. The action of $\Gamma$ on itself by $\cdot\otimes(L,\nabla)^{\otimes2}$ is transitive, so that $T_0$ is surjective.

\begin{lem}\label{lem:fix-trace}
	The trace map $T_0$ is a surjective holomorphic submersion. In particular, the subset \[\mathcal{M}^{\alpha,\circ,0}_{C,D,\Lambda}:=T_0^{-1}(\mathcal{O}_C,d)\] of trace-free connections is a smooth submanifold of $\mathcal{M}^{\alpha,\circ}_{C,D,\Lambda}$ of dimension \[6g-6+2\deg(D).\]
\end{lem}

\begin{proof}
	If the trace map had a critical value in $\Gamma$, then all the elements of this critical value's orbit under the above action of $\Gamma$ on itself would be critical values as well, since the action of $\Gamma$ on $\mathcal{M}^{\alpha,\circ}_{C,D,\Lambda}$ induces isomorphisms between fibers. But the action of $\Gamma$ on itself is transitive (which, by the way, implies that $T_0$ is surjective) and $T_0$ is holomorphic, hence this would contradict Sard's theorem.
	
	The value of $\dim(\mathcal{M}^{\alpha,\circ,0}_{C,D,\Lambda})=\dim(\mathcal{M}^{\alpha,\circ}_{C,D,\Lambda})-\dim(\Gamma)$ is given by \cite[Thm. 3.1]{MR4545855} (Inaba's theorem applied in the rank $2$ case gives $\operatorname{dim}(\mathcal{M}^{\alpha,\circ}_{C,D,\Lambda})=2(4(g-1)+1+\deg(D))$ as in \autoref{Inaba-thm}, and in the rank $1$ case it gives $\dim(\Gamma)=2(g-1)+2$).
\end{proof}

A similar argument shows that, if $\Lambda$ is such that
\begin{align*}
T_1:\mathcal{M}^{\alpha,\circ}_{C,D,\Lambda}&\longrightarrow\Gamma_1\\
(E,\nabla)&\longmapsto (\det(E),\operatorname{tr}(\nabla))
\end{align*}
is well-defined, then $\mathcal{M}^{\alpha,\circ,1}_{C,D,\Lambda}:=T_1^{-1}(\mathcal{O}_C(p_i),\zeta)$ is a smooth submanifold of $\mathcal{M}^{\alpha,\circ}_{C,D,\Lambda}$ (the twist action of $\Gamma$ on $\Gamma_1$ is transitive). The dimensions are the same as before.

Note that the admissible $\Lambda$ in both cases must in particular satisfy the hypothesis of \autoref{lemma-genericity}, hence they are "generic".

\paragraph{A moduli space of singular Riccati foliations.} Let $\Gamma'$ be the subgroup of square roots of the rank $1$ trivial connection $(\mathcal{O}_C,d)$ in $\Gamma$. It is a discrete Lie group of cardinal $2^{2g}$, acting holomorphically on $\mathcal{M}^{\alpha,\circ,0}_{C,D,\Lambda}$ and $\mathcal{M}^{\alpha,\circ,1}_{C,D,\Lambda}$ by the twist operation. The complements $\mathcal{M}^{\alpha,\circ,0,*}_{C,D,\Lambda}$ and $\mathcal{M}^{\alpha,\circ,1,*}_{C,D,\Lambda}$ of fixed points (a fixed point is a point with a non-trivial stabilizer) are therefore invariant open subsets. In addition, the action of $\Gamma'$ on those spaces is proper for it is a finite group, so that the quotient spaces $\mathcal{M}^{\alpha,\circ,0,*}_{C,D,\Lambda}/\Gamma'$ and $\mathcal{M}^{\alpha,\circ,1,*}_{C,D,\Lambda}/\Gamma'$ have unique structures of smooth manifolds such that the quotient maps are local diffeomorphisms.

Let $(E,\nabla)$ and $(E',\nabla')$ be two elements in $\mathcal{M}^{\alpha,\circ,0,*}_{C,D,\Lambda}$ or $\mathcal{M}^{\alpha,\circ,1,*}_{C,D,\Lambda}$. Then,
\[\mathbb{P}(E,\nabla)=\mathbb{P}(E',\nabla') ~\Leftrightarrow~  (E,\nabla)=(E',\nabla')~~ \operatorname{mod} ~~\Gamma'.\]
For this reason, to an element in $\mathcal{M}^{\alpha,\circ,0,*}_{C,D,\Lambda}/\Gamma'$ or $\mathcal{M}^{\alpha,\circ,1,*}_{C,D,\Lambda}/\Gamma'$ corresponds a unique singular Riccati foliation $(Q,\mathcal{F})$ (with no apparent singularity) on $C$, with polar divisor $D$ and formal data $\bar{\Lambda}$. We identify those quotients to the corresponding spaces of singular Riccati equations and denote\footnote{Beware that the letter $\Lambda$ does not refer to the same element in both cases: the set of admissible formal data depends on the chosen trace.}
\[\mathfrak{Ricc}^{\alpha,\circ,0,*}_{C,D,\bar{\Lambda}}\leftrightarrow\mathcal{M}^{\alpha,\circ,0,*}_{C,D,\Lambda}/\Gamma'\text{~~ and ~~}\mathfrak{Ricc}^{\alpha,\circ,1,*}_{C,D,\bar{\Lambda}}\leftrightarrow\mathcal{M}^{\alpha,\circ,1,*}_{C,D,\Lambda}/\Gamma'.\]
This is well-defined if $\bar{\Lambda}$ is the projectivization of $\Lambda$. Indeed, $\Lambda$ is then uniquely determined by $\bar{\Lambda}$ and the negative part of the trace of a connection matrix.

The latter two moduli spaces contain, roughly speaking, foliations induced by $\operatorname{PGL}(2,\mathbb{C})$-opers on $C$. This is the meaning of the following lemma.

\begin{lem}\label{lem:opers-are-not-fixed}
	Assume that $2-2g-\deg(D)<0$. Let $(E,\nabla,\mathcal{L})\in\mathcal{M}^{\alpha,\circ,0}_{C,D,\Lambda}$ (resp. $\in\mathcal{M}^{\alpha,\circ,1}_{C,D,\Lambda}$) be a parabolic rank $2$ meromorphic connection admitting a subbundle $L\subset E$ such that $(E,\nabla,L)$ is a $\operatorname{GL}(2,\mathbb{C})$-oper over $C$. Then, $(E,\nabla,\mathcal{L})$ is not a fixed point for the action of $\Gamma'$ on $\mathcal{M}^{\alpha,\circ,0}_{C,D,\Lambda}$ (resp. $\mathcal{M}^{\alpha,\circ,1}_{C,D,\Lambda}$).
\end{lem}

\begin{proof}	

	We need to prove that any rank $1$ connection $(L_0,\nabla_{L_0})\in\Gamma'$ such that $(E,\nabla)\otimes (L_0,\nabla_{L_0})\simeq (E,\nabla)$ is necessarily trivial. First, note that $L_0$ must have degree $0$.
	
	Denote by $\sigma$ the section of $\mathbb{P}(E)$ associated to $L$. From the equation \eqref{self-intersection-number} and our assumptions, we glean the inequality $\sigma\cdot\sigma<0$. Recall from \autoref{P1-bundle-facts} that $\sigma$ is the only section with negative self-intersection number. Line subbundles $L\otimes L_0$ and $L$ (of $E\otimes L_0$ and $E$, respectively) have same degree and are both maximal. By uniqueness of the maximal line subbundle, if $E\otimes L_0$ is isomorphic to $E$ such an isomorphism must send $L\otimes L_0$ to $L$, which in turn are isomorphic. Thus, $L_0=\mathcal{O}_C$ and it follows that $\nabla_{L_0}$ must be trivial.
\end{proof}

Let $(\pi:Q\rightarrow C,\mathcal{F},\sigma)$ be a $\operatorname{PGL}(2,\mathbb{C})$-oper on $C$ with polar divisor $D$ minimal, formal data $\bar{\Lambda}$ and without apparent singularity. Suppose that $2-2g-\deg(D)<0$. Assume moreover that $\deg(D)\equiv 0~ \operatorname{mod} ~2$ (resp. $\deg(D)\equiv 1~ \operatorname{mod} ~2$). Then, by \autoref{lem-stability} and \autoref{lem:opers-are-not-fixed} there exists a parabolic weight $\alpha$ such that by the construction above, $(\pi:Q\rightarrow C,\mathcal{F},\sigma)$ corresponds to a unique element in $\mathfrak{Ricc}^{\alpha,\circ,0,*}_{C,D,\bar{\Lambda}}$ (resp. $\mathfrak{Ricc}^{\alpha,\circ,1,*}_{C,D,\bar{\Lambda}}$).

\begin{rmk}\label{rmk:injectivity}
	Note that no two such opers $(\pi:Q\rightarrow C,\mathcal{F},\sigma)$ and $(\pi':Q'\rightarrow C,\mathcal{F}',\sigma')$ are mapped to the same element, because then $(\pi:Q\rightarrow C,\mathcal{F})=(\pi':Q'\rightarrow C,\mathcal{F}')$ and moreover $\sigma=\sigma'$ because by assumption, $\sigma$ is the only section of $(\pi:Q\rightarrow C,\mathcal{F})$ with negative self-intersection.
\end{rmk}

\subsection{The irregular Riemann-Hilbert correspondence}	

Let $(E,\nabla)$ be a meromorphic rank $2$ connection on a genus $g$ curve $C$, with polar divisor $D=\sum_{i=1}^dm_ip_i$ and polar locus $\Sigma=\{p_i:1\leq i\leq d\}$. Choose a base point $p_0$ in $C\smallsetminus\Sigma$. A presentation fundamental group $\pi(C\smallsetminus\Sigma,p_0)$ of the punctured curve is
\[\pi(C\smallsetminus\Sigma,p_0)=\braket{\alpha_i,\beta_i,\gamma_j,~ 1\leq i \leq g,~ 1\leq j\leq d ~|~ [\alpha_1,\beta_1]\cdots [\alpha_g,\beta_g]=\gamma_1\cdots\gamma_d},\]
where $\{\alpha_i,\beta_i\}$ is a canonical system of generators of $\pi(C,p_0)$ not going trough $\Sigma$, and where the $\gamma_i$ are defined as follows. For each $i$, choose a small contractible open neighborhood $U_i$ of $p_i$ in $C\smallsetminus\Sigma$, and a path $\delta_i$ connecting $p_0$ to $p_i$, which lies in $C\smallsetminus \Sigma$ except for its endpoints. Then, $\gamma_i$ is defined by traveling from $p_0$ to $U_i$ along $\delta_i$, encircling $p_i$ in the counterclockwise direction by a small loop in $U_i$, and then returning to $x$ along $\delta_i$.

The (global) \emph{monodromy representation} of $(E,\nabla)$ is defined (up to conjugacy) as follows. Let $\Omega_0$ be a connection matrix of $(E,\nabla)$ in a neighborhood of $p_0$, and let $Y$ be a fundamental solution of the system $dY+\Omega_0 Y=0$ around $p_0$. The analytic continuation of $Y$ along a loop $\gamma\in\pi_1(C\smallsetminus \Sigma)$ yields a new fundamental solution $Y^\gamma=YM^\gamma$ for some matrix $M^\gamma\in\operatorname{GL(2,\mathbb{C})}$. The monodromy representation of the system with respect to $Y$ is defined as usual, by
\begin{align*}
\pi_1(C\smallsetminus\Sigma,p_0)&\longrightarrow\operatorname{GL}(2,\mathbb{C})\\
\gamma&\longmapsto M^\gamma.
\end{align*}
Another choice of fundamental solution $\tilde{Y}$ leads to a monodromy representation that is conjugated to the above one in $\operatorname{GL}(2,\mathbb{C})$.

\paragraph{The generalized monodromy data.} In order to extend the classical Riemann-Hilbert correspondence to connections with irregular singularities, we must enrich the monodromy representation with some extra data. This includes the global monodromy representation, together with a decomposition of the local monodromies into: the formal local monodromies, %
the \emph{Stokes matrices} and the \emph{links} (or \emph{connection matrices}).
\paragraph{The wild character varieties.} Let us denote the set of generalized monodromy data by $\mathcal{R}(\mathbb{S},\mathbb{M})$, and it Riccati foliation counterpart by $\bar{\mathcal{R}}(\mathbb{S},\mathbb{M})$.
Different approaches exist in the literature to endow subsets of $\mathcal{R}(\mathbb{S},\mathbb{M})$, and $\bar{\mathcal{R}}(\mathbb{S},\mathbb{M})$
with a smooth structure: see \cite{MR2250948,MR2649335,MR3126570,boalch2015twisted,MR3383320,MR3802126,MR3932256,Allegretti_2020}.%

Our approach is to transport the smooth structure of the moduli spaces $\mathcal{M}^{\alpha,\circ}_{C,D,\Lambda}$ constructed by Inaba in \cite{MR4545855} to the corresponding subsets of the wild character varieties via the \emph{irregular Riemann-Hilbert correspondence}.

\begin{thm}[Irregular Riemann-Hilbert correspondence]
	The monodromy map
	\begin{align*}
	\mathcal{M}^{\alpha,\circ}_{C,D,\Lambda}&\longrightarrow\mathcal{R}(\mathbb{S},\mathbb{M})
	\end{align*}
	is injective. %
\end{thm}

We refer to \cite{lorayRH}.
See also \cite[Thm. 19]{MR1128863} and \cite[App. A]{MR3126570}. Other useful references include \cite{MR672182} and \cite{MR2363178}.\\

\begin{cor}[Irregular Riemann-Hilbert correspondence for Riccati foliations]
	The monodromy maps
	\begin{align*}
	\mathfrak{Ricc}^{\alpha,\circ,0,*}_{C,D,\bar{\Lambda}}\longrightarrow\bar{\mathcal{R}}(\mathbb{S},\mathbb{M})\text{~~ and ~~}\mathfrak{Ricc}^{\alpha,\circ,1,*}_{C,D,\bar{\Lambda}}\longrightarrow\bar{\mathcal{R}}(\mathbb{S},\mathbb{M})
	\end{align*}
	are injective.
\end{cor}

We denote by $\bar{\mathcal{R}}^*(\mathbb{S},\mathbb{M},(\lambda_{-1}^{(i)}))$ the image of this map once the curve, the pole order and the residues are fixed, and endow it with the smooth complex structure induced by the monodromy map and the smooth structure of the moduli spaces of singular Riccati foliations.

\subsection{Universal isomonodromic deformations}\label{sec:isomonodromic-deformations}

In \cite{MR2667785}, Heu constructed the universal isomonodromic deformations of trace-free meromorphic rank $2$ connections on curves (see also her thesis \cite{heu:tel-00358039}). Recall that a topologically trivial, analytic deformation \[(E_t\rightarrow C_t,\nabla_t)_{t\in T}~~(\text{topologically trivial means $T$ is contractible})\] of some initial trace-free rank $2$ connection $(E_0\rightarrow C_0,\nabla_0)$ is called an \emph{isomonodromic deformation} if it is induced by a flat, locally constant connection $(\mathcal{E}\rightarrow X,\nabla)$ over the total space of the analytic family of curves $X\rightarrow T$ whose fiber over $t$ is $C_t$.

Here, \emph{locally constant} means that locally around each point of $X$, after a convenient change of coordinate (a combination of a gauge transformation of $\mathcal{E}\rightarrow X$ and a change of coordinate in $X$), the connection matrix of $\nabla$ does not depend on the parameter $t\in T$.

\begin{rmk}
	Isomonodromic deformations where also studied by Inaba in his paper \cite{MR4563425}.
\end{rmk}

Heu's construction can be implemented \emph{mutatis mutandis} to get the universal isomonodromic deformations of an initial Riccati foliation $(Q_0,\mathcal{F}_0):=\mathbb{P}(E_0,\nabla_0)$ (see \cite[Sec. 5.2]{MR2667785}).

\paragraph{Explicit construction.} Let us briefly recall the explicit construction of the universal isomonodromic deformations; for the full details in the case of meromorphic trace-free rank $2$ connections, we refer to \cite[Sec. 3.3.2]{MR2667785}.

It is natural to start with the case of non-singular Riccati foliation before dealing with the singular regular case and eventually the general meromorphic case (including irregular singularities). So let $(Q_0,\mathcal{F}_0)$ be an initial holomorphic Riccati foliation on a curve $C_0$ on $S$. It turns out that it is sufficient to deform the complex structure of the curve $C_0$ to obtain the universal isomonodromic deformation of this initial Riccati foliation. This is a direct consequence of the classical Riemann-Hilbert correspondence (which amounts to the suspension of the monodromy representation). The construction goes as follows. Let us denote by \[X_\mathcal{T}\longrightarrow\mathcal{T}(S)\] the Teichmüller universal family of curves. Because the Teichmüller space $\mathcal{T}(S)$ is contractible, the inclusion $C_0\hookrightarrow X_\mathcal{T}$ induces an isomorphism of the fundamental groups $\pi_1(C_0,*)\simeq\pi_1(X_\mathcal{T},*)$. Hence, the suspension of the monodromy representation of $(Q_0,\mathcal{F}_0)$ provides with a connection $(\mathcal{Q}\rightarrow X_\mathcal{T},\mathcal{F})$ inducing the universal isomonodromic deformation of $(Q_0,\mathcal{F}_0)$.

In the regular singular cases, we could use the generalized Riemann-Hilbert correspondence. However, Heu adapted a construction of Malgrange using only the classical correspondence. Let $(Q_0,\mathcal{F}_0)$ be an initial singular Riccati foliation with polar divisor $D_0$, admitting $d$ poles (of orders at most $1$). Here again, it suffices to deform the complex structure of $C_0$. Let us denote by \[X_\mathcal{T}\longrightarrow\mathcal{T}(S,d)\] the Teichmüller universal family of marked curves, and $\mathcal{D_\mathcal{T}}$ the divisor on $X_\mathcal{T}$ corresponding to marked points. On one hand, the suspension of the monodromy of the initial Riccati foliation gives a Riccati foliation $(\mathcal{Q}^*\rightarrow X_\mathcal{T}^*,\mathcal{F}^*)$ over $X_\mathcal{T}^*=X_\mathcal{T}\smallsetminus \mathcal{D}_\mathcal{T}$ extending uniquely $(Q_0,\mathcal{F}_0)_{|X_0^*}$, where $X_0^*=X_0\smallsetminus D_0$. On the other hand, around each component $\mathcal{D}^i_\mathcal{T}$ of $\mathcal{D}_\mathcal{T}$, there exists a tubular neighborhood $\mathcal{U}^i_\mathcal{T}$ adapted to the fibers of the Teichmüller family of curves. It is possible to extend $(Q_0,\mathcal{F}_0)_{|U^i_0}$, where $U^i_0=\mathcal{U}^i_\mathcal{T}\cap C_0$, as a product over this tubular neighborhood. Each of the resulting Riccati foliations $(\mathcal{Q}^i,\mathcal{F}^i)$ have the same monodromy as $(\mathcal{Q}^*\rightarrow X_\mathcal{T}^*,\mathcal{F}^*)$ in restriction to their common domain of holomorphy, so that, according to the classical Riemann-Hilbert correspondence, there is a unique isomorphism gluing them into a connection $(\mathcal{Q},\mathcal{F})$ over $X_\mathcal{T}$ inducing the universal isomonodromic deformation of $(Q_0,\mathcal{F}_0)$.

In the general meromorphic case (including irregular singularities), the above gluing is not always unique (cf. Lemma 3.6 in \cite{MR2667785}). Hence, the deformation of the complex structure is no longer sufficient to obtain the universal isomonodromic deformation of an initial Riccati foliation $(Q_0,\mathcal{F}_0)$. The additional freedom in the construction corresponds to the irregular part of the formal invariants of the Riccati foliation, that has to be deformed as well. In Heu's work, they appear as $(m_i-1)$-jet of germs of coordinate changes around each pole $p_i$ of order $m_i>1$, but this is equivalent (cf. \autoref{jet-vs-irregular-formal-invariants} below). The universal cover $J$ of the space of such jets is then added to the parameter space to form a "thickened Teichmüller space", or a moduli space of "irregular curves" $T=J\times\mathcal{T}(S,d)$ (passing to the universal covering ensures this new parameter space remains contractible). Then, the universal isomonodromic deformation of $(Q_0,\mathcal{F}_0)$ is constructed on a family of curves \[(X,\mathcal{D})=(J\times X_\mathcal{T},J\times\mathcal{D}_\mathcal{T})\] over $T$ via a gluing construction similar to the previous one.
\begin{rmk}\label{jet-vs-irregular-formal-invariants}
	Consider a Riccati equation \eqref{Riccati} in the irregular unramified case. Its Hukuhara-Tirrittin formal normal form reads
	\begin{equation}\label{Riccati-irregular-unramified}
	dy+\lambda ydx=0,
	\end{equation}
	for some $\lambda=\frac{\lambda_{-m}}{x^m}+\cdots+\frac{\lambda_{-1}}{x}$ and $\lambda_i\in\mathbb{C}$.
	By an additional change of the $x$-coordinate, this equation can be put into the normal form \begin{equation}\label{eq:normal-form-without-irregular-formal-invariants}
	dy+\left[\frac{1}{x^m}+\frac{\lambda_{-1}}{x}\right]ydx=0
	\end{equation}  where $\lambda_{-1}$ is the residue of $\lambda dx$ and is invariant under a change of the $x$-coordinate (see \cite[proof of Prop. 2.3]{loray-book}). A similar computation holds in the irregular ramified case, but in the variable $z$.%

	More generally, consider a Riccati equation \eqref{Riccati} with no assumption on the minimality of $m$, and denote by $\omega$ the $1$-form defining the foliation $\mathcal{F}$. If $\varphi$ is a change of the $x$ coordinate centered at $0$ such that \[\varphi(x)=\operatorname{id}(x)\text{~~ mod ~~}x^m,\] then there is a holomorphic gauge transformation $g$ such that $g\cdot(\varphi^*\omega)=\omega$ \cite[Lem. 3.6]{MR2667785}.

	In particular, this tells us that the Hukuhara-Turrittin formal normal form \eqref{Riccati-irregular-unramified} can be put into the form \eqref{eq:normal-form-without-irregular-formal-invariants} using only a $(m-1)$-jet of coordinate change.%

\end{rmk}

\subsection{Construction of a family of singular Riccati foliations}\label{sec:construction-local-family}

Our approach is a (local) combination of the constructions of Inaba and Heu, in order to obtain a smooth universal family of singular Riccati foliations that (locally) contains the ones subjacent to $\operatorname{PGL}(2,\mathbb{C})$-opers induced by meromorphic projective structures without apparent singularity (cf. \autoref{family-of-opers}).

Thanks to the irregular Riemann-Hilbert correspondence, we know that the smooth moduli space \[\mathcal{M}(\mathbb{S},\mathbb{M},(\lambda_{-1}^{(i)})):=\bar{\mathcal{R}}^*(\mathbb{S},\mathbb{M},(\lambda_{-1}^{(i)}))\times T,\]
where $T=J\times\mathcal{T}(S,d)$, contains all singular Riccati foliations associated with elements in $\mathcal{P}^\circ(\mathbb{S},\mathbb{M},(\lambda_{-1}^{(i)}))$ (if $2-2g-\sum_{i=1}^dn_i<0$).

\paragraph{Construction with additional parameters.} Let $I\subset\bar{\mathcal{R}}^*(\mathbb{S},\mathbb{M},(\lambda_{-1}^{(i)}))$ be a small open ball. The universal family of \autoref{Inaba-thm} can be transported all the way down to $I$, while projectivizing the connections of the family, making it a holomorphic family of singular Riccati foliation on  $I$.

This provides a smooth family of initial singular Riccati foliations $(Q_0,\mathcal{F}_0)$. Those parameters can be added to Heu's construction of isomonodromic deformations in order to obtain a holomorphic family \[(Q_t\rightarrow C_t,\mathcal{F}_t)_{t\in I\times J\times \mathcal{T}(S,d)}\] of singular Riccati foliations on the family of curves
\[(X,\mathcal{D})=(I\times J\times X_\mathcal{T},I\times J\times\mathcal{D}_\mathcal{T}).\]

\section{The monodromy map is a local~biholomorphism}

To an equivalence class of marked meromorphic projective structures with signed residues in $\mathcal{P}^\circ(\mathbb{S},\mathbb{M},(\lambda_{-1}^{(i)}))$ corresponds (cf. \eqref{marquing-covering-map}) a unique meromorphic projective structure without apparent singularity $P$ on some complex curve on the genus $g$ real surface $S$, with residues $(\pm\lambda_{-1}^{(i)})$. Let us denote by $(n_i)$ its pole orders (cf. \autoref{rmk:pole-orders}). To $P$ corresponds a unique $\operatorname{PGL}(2,\mathbb{C})$-oper $(\pi:Q\rightarrow C,\mathcal{F},\sigma)$ on $C$, with minimal polar divisor $\tilde{D}=\sum_{i=1}^dm_ip_i$, where  $m_i:=\lceil \frac{n_i}{2}\rceil$. Assume that if $g=0$, then $|\mathbb{M}|\geq3$, and that if $g=1$, then $|\mathbb{M}|\geq 1$ (cf. \autoref{rmk:hypothesis}). The constructions of the preceding sections provide us with a well-defined (thanks to \autoref{indigenousbundle-meromorphic}) injective map
\[e:\mathcal{P}^\circ(\mathbb{S},\mathbb{M},(\lambda_{-1}^{(i)}))\longrightarrow\mathcal{M}(\mathbb{S},\mathbb{M},(\lambda_{-1}^{(i)})).\]
The injectivity is a direct consequence of \autoref{rmk:injectivity} (which is, in fact, an analogous for meromorphic projective structures of the Poincaré theorem \ref{Poincare-thm}) and of the injectivity of the Riemann-Hilbert map. Moreover, this injection is holomorphic.%

\subsection{Factorization of the monodromy map}

\begin{defi}
	The composition of the above map with the first projection is denoted by
	\begin{align}
	\operatorname{Mon}_{S,(n_i),(\lambda_{-1}^{(i)})}:\mathcal{P}^\circ(\mathbb{S},\mathbb{M},(\lambda_{-1}^{(i)}))\longrightarrow \bar{\mathcal{R}}^*(\mathbb{S},\mathbb{M},(\lambda_{-1}^{(i)}))
	\end{align}
	and is called the \emph{monodromy map}.
\end{defi}

This map is holomorphic. The monodromy of a projective structure is the monodromy of the underlying connection.

\subsection{Ehresmann transversality}

Let $(Q\rightarrow X,\mathcal{F})$ be the universal isomonodromic deformation constructed by Heu (see \autoref{sec:isomonodromic-deformations}) on the analytic family of compact curves $f:X \rightarrow T$. Recall that $\mathcal{F}$ is a codimension one singular holomorphic Riccati foliation over $\mathcal{Q}$, with polar divisor $\mathcal{D}$ on $X$.

\begin{lem}[Local "$\mathcal{C}^\infty$ product of bundles structure"]\label{localproductstructure}
	Let $t_0\in T$. There exists an open neighborhood $U$ of $t_0$ in $T$ such that $\mathcal{F}$ has a "$\mathcal{C}^\infty$ product structure" over $f^{-1}(U)$.
\end{lem}

Let us denote $C_0:=f^{-1}(t_0)$. By a \emph{$\mathcal{C}^\infty$ product structure}, we mean that there exist $\mathcal{C}^\infty$-diffeomorphisms $\Phi$ and $\Theta$ making the diagram

\[\xymatrix{
	\mathcal{Q}_{|f^{-1}(U)} \ar[d]^{} \ar[r]^{\Phi} & U\times\mathcal{Q}_{|C_0} \ar[d]^{}\\
	f^{-1}(U) \ar[r]^{\Theta} & U\times C_0
}\]
commute, and conjugating the foliation $\mathcal{F}_{|f^{-1}(U)}$ to the product Riccati foliation $\Pi^*\mathcal{F}_{|C_0}$ where $\Pi:U\times \mathcal{Q}_{|C_0}\rightarrow \mathcal{Q}_{|C_0}$ is the projection to the second factor.

\begin{proof}[Proof of \autoref{localproductstructure}]
	
	Let $t_0\in T$ and denote by $C_0=f^{-1}(t_0)$ and $p_{i,0}$ the unique point in $\mathcal{D}^i\cap C_0$, where $\mathcal{D}^i$ is an irreducible component of $\mathcal{D}$. By construction, the foliation $\mathcal{F}$ has a \emph{holomorphic} product structure along $\mathcal{D}^i$, locally over an open neighborhood $V_{p_{i,0}}\subset X$ containing $p_{i,0}$. This is the local constancy property (see \cite{MR2667785}). We are going to extend this product structure in a $\mathcal{C}^\infty$-smooth way over a $f$-saturated neighborhood of $C_0$.
	
	The analytic family $f$ is in particular a $\mathcal{C}^\infty$-submersion with compact connected fibers. Thus, Ehresmann's theorem \cite[Sec. 1]{MR0042768} implies that it is the projection map of a locally trivial bundle of class~$\mathcal{C}^\infty$. However, in order to show the lemma we would like to get local trivializations adapted to the meromorphic connection $\mathcal{F}$ and its polar divisor $\mathcal{D}$.

	We start by covering $C_0\subset X$ with an open cover $(V_j)_j$ containing the open neighborhood $V_{p_{i,0}}$ of all $p_{i,0}$ as above, and such that on all $V_j$ (except maybe for the $V_{p_{i,0}}$) the constant rank theorem for the submersion $f$ holds (meaning there are coordinate charts in which $f$ is a projection) and such that the only $V_j$ containing $p_{i,0}$ is $V_{p_{i,0}}$. Then, by compacity of $C_0$, we extract a finite subcover of $(V_j)_j$ which we denote in the same way. Note that this subcover still contains all of the $V_{p_{i,0}}$.
	
	There exists an open neighborhood of $C_0$ in $X$ of the form $V:=f^{-1}(U)\subset(\cup_j V_j)$, with $U$ an open neighborhood of $t_0$. We denote again $V_j:=V_j\cap V$. On each $V_j$ we can find (using the coordinates provided by the constant rank theorem and the coordinates on $V_{p_{i,0}}$ in which the $1$-form defining $\mathcal{F}$ does not depend on the parameter $t$) a $\mathcal{C}^\infty$-foliation which is the pull-back of the radial foliation on $U$ centered at $t_0$. Those foliations are induced by $\mathcal{C}^\infty$-vector fields. We use partitions of unity to glue them and form a vector field $v$ on $V$. Thanks to the compacity of the fibers of $f$, the flow of $v$ exists over all $V$ and gives a $\mathcal{C}^\infty$ trivialization $f^{-1}(U)\simeq~U\times C_0$.

	Using this particular local trivialization, we use again an argument similar to the proof of Ehresmann's theorem in order to obtain $\Phi$. We see the foliation $\mathcal{F}$ as an Ehresmann connection on $\pi:\mathcal{Q}\rightarrow X$. There exist lifts of the previous vector field, horizontal with respect to the connection $\mathcal{F}$. Away from the polar divisor $\mathcal{D}$ of the connection, this is a classical fact. On a neighborhood of $\mathcal{D}^i$, it is the result of our choice of coordinates on the $V_{p_{i,0}}$.

	The compacity of the fibers of $\pi$ ensures the existence of the flow. It gives an isomorphism $\Phi$ satisfying the claimed properties.
\end{proof}

\subsection{Proof of the main theorem}

We are now able to prove ou main theorem.

\begin{thm}
	\label{thm:local-injectivity}
	Assume that if $g=0$, then $|\mathbb{M}|\geq3$, and that if $g=1$, then $|\mathbb{M}|\geq 1$. Then, the monodromy map \[\operatorname{Mon}_{S,(n_i),(\lambda_{-1}^{(i)})}:\mathcal{P}^\circ(\mathbb{S},\mathbb{M},(\lambda_{-1}^{(i)}))\longrightarrow \bar{\mathcal{R}}^*(\mathbb{S},\mathbb{M},(\lambda_{-1}^{(i)}))\] is a local biholomorphism.
\end{thm}

\begin{proof}
	We are going to show that the monodromy map is locally injective.  Since the holomorphy of the monodromy map have already been established, %
	the conclusion will follow immediately by invariance of domain. The fact that the source and the range of the monodromy map have equal dimensions follows from \autoref{prop:fix-residues} and \autoref{lem:fix-trace}.

	Pick a point $x_0\in\mathcal{P}^\circ(\mathbb{S},\mathbb{M},(\lambda_{-1}^{(i)}))$, and denote by $y_0\in\mathcal{M}(\mathbb{S},\mathbb{M},(\lambda_{-1}^{(i)}))=\bar{\mathcal{R}}^*(\mathbb{S},\mathbb{M},(\lambda_{-1}^{(i)}))\times T$ its image by $e$. Let $V=I\times U$ be an open neighborhood of $y_0$, such that $I\subset\bar{\mathcal{R}}^*(\mathbb{S},\mathbb{M},(\lambda_{-1}^{(i)}))$ is an open ball sufficiently small to carry the universal analytic family of singular Riccati foliation constructed in \autoref{sec:construction-local-family}, and such that $U\in T$ is sufficiently small to get, by local constancy of isomonodromic deformations and Ehresmann transversality as in \autoref{localproductstructure}, a $\mathcal{C}^\infty$ product structure along isomonodromic deformations of elements in $V$.

	Since the map $e$ is continuous, there exists an open neighborhood $W$ of $x_0$ such that $e(W)\subset V$. Up to shrinking $W$ to a smaller open neighborhood of $x_0$, we know that it carries a family of meromorphic $\operatorname{PGL}(2,\mathbb{C})$-opers in the sense of \autoref{family-of-opers}. This is, in particular, a family of singular Riccati foliations. By injectivity of $e$ and the universal property of the analytic family of singular Riccati foliations over $V$, it induces an injection at the level of the families.

	We are going to show that the restriction of the monodromy map to $W$ is one-to-one. Pick two points $x_1,x_2\in W$, and suppose that their images by the monodromy map are equal. Denote $z:=\operatorname{Mon}_{S,(n_i),(\lambda_{-1}^{(i)})}(x_1)=\operatorname{Mon}_{S,(n_i),(\lambda_{-1}^{(i)})}(x_2)$, $y_1=e(x_1)$, $y_2=e(x_2)$ and finally $t_0$, $t_1$ and $t_2$ the projection of $y_0$, $y_1$ and $y_2$, respectively. As depicted on \autoref{fig:transversality}, the injectivity of the monodromy map at $x_0$ arise from the transversality of the image of $W$ by $e$ with respect to the fibers of the projection to the first factor, i.e. with respect to the monodromy map of singular Riccati foliations.
	\begin{figure}[H]
		\centering
		\def\svgwidth{1\textwidth}
		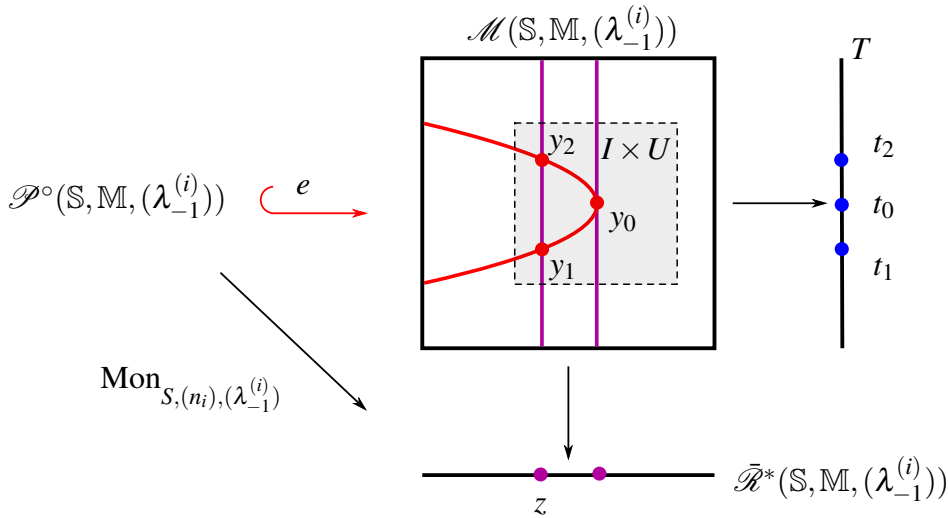
		\caption{This situation where $y_1$ and $y_2$ are distinct is impossible: the image of $\mathcal{P}^\circ(\mathbb{S},\mathbb{M},(\lambda_{-1}^{(i)}))$ by $e$ is transverse to the fibers of the projection onto $\bar{\mathcal{R}}^*(\mathbb{S},\mathbb{M},(\lambda_{-1}^{(i)}))$.}
		\label{fig:transversality}
	\end{figure}
	Let $(\pi_i:Q_i\rightarrow C_i,\mathcal{F}_i,\sigma_i)$ be the opers corresponding to $x_i$, for each $i=1,2$, in the family of \autoref{family-of-opers}. They are isomonodromic deformations of one another. Recall that isomonodromic deformations are induced by a codimension one foliation, whose holonomy locally gives a $\mathcal{C}^\infty$-retraction of the Riccati foliations of the family on, say, $(\pi_1:Q_1\rightarrow C_1,\mathcal{F}_1)$. This is the content of \autoref{localproductstructure}, depicted on \autoref{fig:product-structure}.
	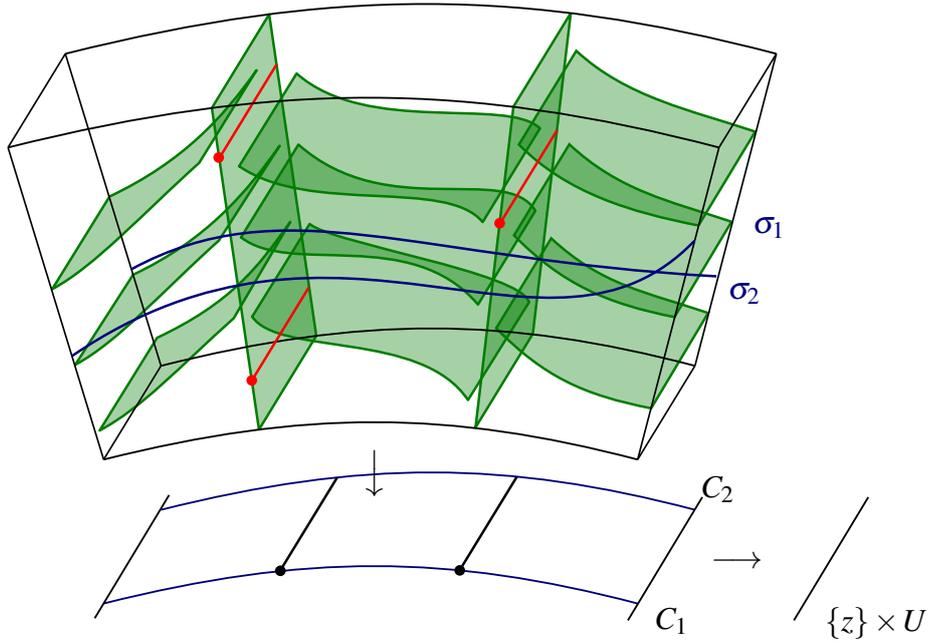
\begin{figure}[H]
		\centering
		\def\svgwidth{1\textwidth}
\begingroup%
  \makeatletter%
  \providecommand\color[2][]{%
    \errmessage{(Inkscape) Color is used for the text in Inkscape, but the package 'color.sty' is not loaded}%
    \renewcommand\color[2][]{}%
  }%
  \providecommand\transparent[1]{%
    \errmessage{(Inkscape) Transparency is used (non-zero) for the text in Inkscape, but the package 'transparent.sty' is not loaded}%
    \renewcommand\transparent[1]{}%
  }%
  \providecommand\rotatebox[2]{#2}%
  \newcommand*\fsize{\dimexpr\f@size pt\relax}%
  \newcommand*\lineheight[1]{\fontsize{\fsize}{#1\fsize}\selectfont}%
  \ifx\svgwidth\undefined%
    \setlength{\unitlength}{547.7814778bp}%
    \ifx\svgscale\undefined%
      \relax%
    \else%
      \setlength{\unitlength}{\unitlength * \real{\svgscale}}%
    \fi%
  \else%
    \setlength{\unitlength}{\svgwidth}%
  \fi%
  \global\let\svgwidth\undefined%
  \global\let\svgscale\undefined%
  \makeatother%
  \begin{picture}(1,0.60070075)%
    \lineheight{1}%
    \setlength\tabcolsep{0pt}%
    \put(0,0){\includegraphics[width=\unitlength,page=1]{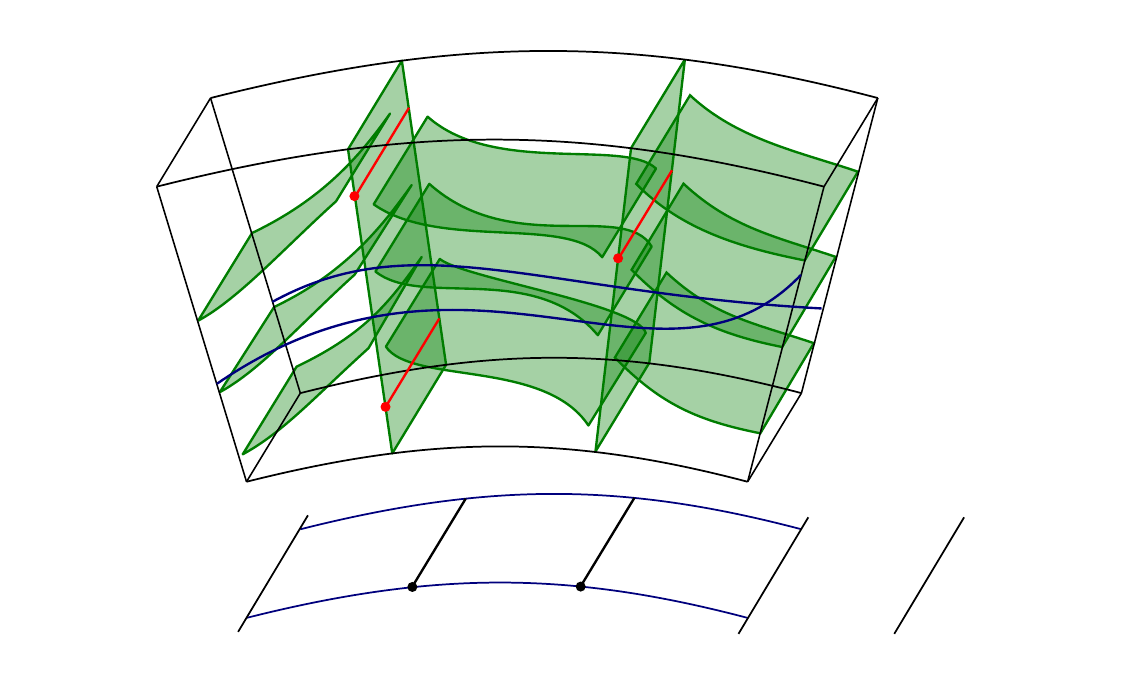}}%
    \put(0.74991271,0.36629977){\color[rgb]{0,0,0.49019608}\makebox(0,0)[lt]{\lineheight{1.25}\smash{\begin{tabular}[t]{l}$\sigma_1$\end{tabular}}}}%
    \put(0.43197648,0.18877463){\color[rgb]{0,0,0}\rotatebox{-90}{\makebox(0,0)[lt]{\lineheight{1.25}\smash{\begin{tabular}[t]{l}$\longrightarrow$\end{tabular}}}}}%
    \put(0.71545241,0.08872803){\color[rgb]{0,0,0}\makebox(0,0)[lt]{\lineheight{1.25}\smash{\begin{tabular}[t]{l}$\longrightarrow$\end{tabular}}}}%
    \put(0.72951153,0.31176732){\color[rgb]{0,0,0.49019608}\makebox(0,0)[lt]{\lineheight{1.25}\smash{\begin{tabular}[t]{l}$\sigma_2$\end{tabular}}}}%
    \put(0.80917303,0.03696733){\color[rgb]{0,0,0}\makebox(0,0)[lt]{\lineheight{1.25}\smash{\begin{tabular}[t]{l}$\{z\}\times U$\end{tabular}}}}%
    \put(0.67043812,0.0382416){\color[rgb]{0,0,0}\makebox(0,0)[lt]{\lineheight{1.25}\smash{\begin{tabular}[t]{l}$C_1$\end{tabular}}}}%
    \put(0.70828995,0.14737777){\color[rgb]{0,0,0}\makebox(0,0)[lt]{\lineheight{1.25}\smash{\begin{tabular}[t]{l}$C_2$\end{tabular}}}}%
  \end{picture}%
\endgroup%

		\caption{The local product structure.}
		\label{fig:product-structure}
	\end{figure}
	Hence, the two foliated bundles $(\pi_i:Q_i\rightarrow C_i,\mathcal{F}_i)$, $i=1,2$ are isomorphic in the real differentiable sense. Denote by $\tilde{\sigma}_2$ the image in $Q_1$ of $\sigma_2$ by this isomorphism $H$. Up to shrinking $W$ again, we can assume that $\tilde{\sigma}_2$ is close to $\sigma_1$ in the sense that it is contained in a tubular neighborhood of the latter section, adapted to the foliation (cf. \autoref{fig:tubular-neighborhood}). This is a consequence of the continuity of the map of the family of meromorphic opers into the universal family of singular Riccati equations.

	\begin{figure}[H]
		\centering
		\def\svgwidth{0.6\textwidth}
\begingroup%
  \makeatletter%
  \providecommand\color[2][]{%
    \errmessage{(Inkscape) Color is used for the text in Inkscape, but the package 'color.sty' is not loaded}%
    \renewcommand\color[2][]{}%
  }%
  \providecommand\transparent[1]{%
    \errmessage{(Inkscape) Transparency is used (non-zero) for the text in Inkscape, but the package 'transparent.sty' is not loaded}%
    \renewcommand\transparent[1]{}%
  }%
  \providecommand\rotatebox[2]{#2}%
  \newcommand*\fsize{\dimexpr\f@size pt\relax}%
  \newcommand*\lineheight[1]{\fontsize{\fsize}{#1\fsize}\selectfont}%
  \ifx\svgwidth\undefined%
    \setlength{\unitlength}{491.08856441bp}%
    \ifx\svgscale\undefined%
      \relax%
    \else%
      \setlength{\unitlength}{\unitlength * \real{\svgscale}}%
    \fi%
  \else%
    \setlength{\unitlength}{\svgwidth}%
  \fi%
  \global\let\svgwidth\undefined%
  \global\let\svgscale\undefined%
  \makeatother%
  \begin{picture}(1,0.65273116)%
    \lineheight{1}%
    \setlength\tabcolsep{0pt}%
    \put(0,0){\includegraphics[width=\unitlength,page=1]{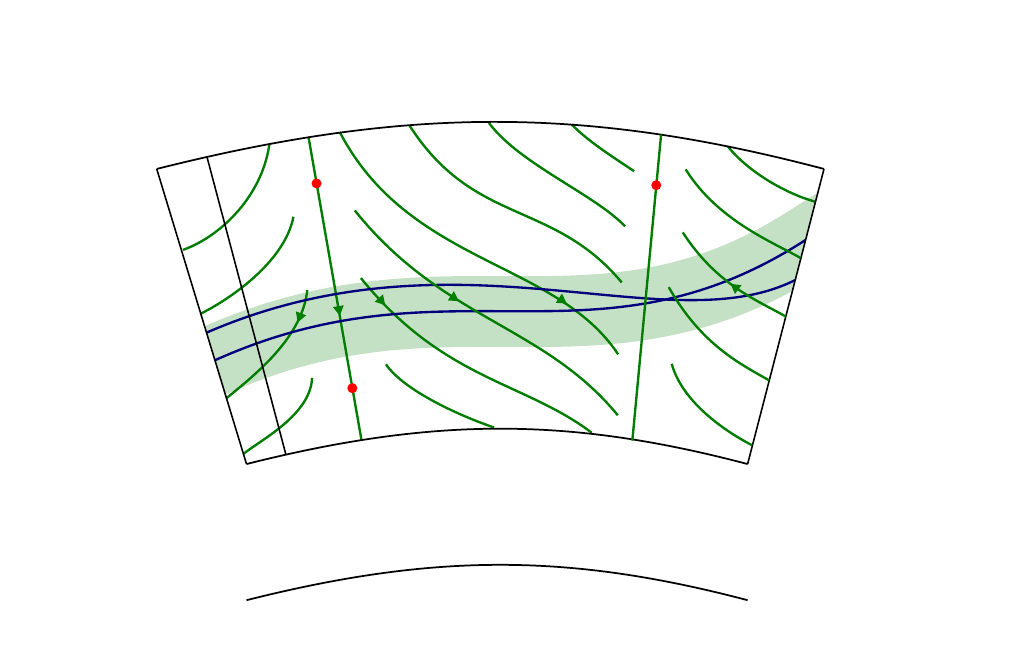}}%
    \put(0.75848321,0.08571242){\color[rgb]{0,0,0}\makebox(0,0)[lt]{\lineheight{1.25}\smash{\begin{tabular}[t]{l}$S$\end{tabular}}}}%
    \put(0.78666075,0.35053564){\color[rgb]{0,0,0.49019608}\makebox(0,0)[lt]{\lineheight{1.25}\smash{\begin{tabular}[t]{l}$\tilde{\sigma}_2$\end{tabular}}}}%
    \put(0.4818453,0.21601786){\color[rgb]{0,0,0}\rotatebox{-90}{\makebox(0,0)[lt]{\lineheight{1.25}\smash{\begin{tabular}[t]{l}$\longrightarrow$\end{tabular}}}}}%
    \put(0.80520098,0.41834183){\color[rgb]{0,0,0.49019608}\makebox(0,0)[lt]{\lineheight{1.25}\smash{\begin{tabular}[t]{l}$\sigma_1$\end{tabular}}}}%
  \end{picture}%
\endgroup%

		\caption{A tubular neighborhood of $\sigma_1$, containing $\tilde{\sigma}_2$.}
		\label{fig:tubular-neighborhood}
	\end{figure}
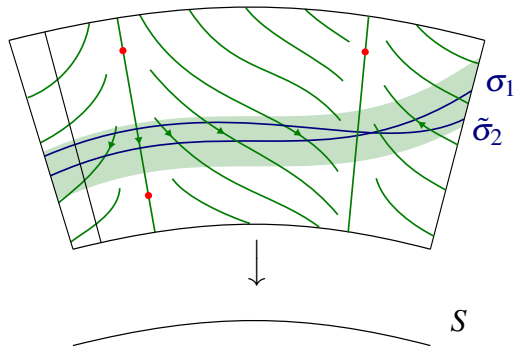
	
	The holonomy of the foliation $\mathcal{F}_1$ defines a $\mathcal{C}^\infty$-diffeomorphism \[\operatorname{hol}:\tilde{\sigma}_2\longrightarrow\sigma_1.\]

	The $\mathcal{C}^\infty$-diffeomorphism \[f:C_2\longrightarrow C_1\] defined by $f=\pi_1\circ\operatorname{hol}\circ H\circ\sigma_2$ is an isomorphism of projective structures. In conclusion, $x_1=x_1$.
\end{proof}

\clearpage
\bibliography{bibli.bib}
\bibliographystyle{alphaurl}

\end{document}